\documentclass[onefignum,onetabnum]{siamonline250211}

\usepackage[a4paper, margin=2.5cm]{geometry}
\usepackage{hyperref}
\usepackage{amsmath}
\usepackage{cleveref}
\usepackage{amsfonts}
\usepackage[x11names]{xcolor}
\usepackage{mathtools}
\usepackage{nicefrac}
\usepackage{bm}
\usepackage{todonotes} %disable
\usepackage{enumitem}
\usepackage{subcaption}
\newcommand{\R}{\mathbb{R}}

\newcommand{\E}{\mathbb{E}}
\newcommand{\bbP}{\mathbb{P}}

\newcommand{\bU}{\bm{U}}

\newcommand{\bY}{\bm{Y}}
\newcommand{\bZ}{\bm{Z}}

\newcommand{\bJ}{\bm{J}}
\newcommand{\bK}{\bm{K}}

\newcommand{\bgamma}{\bm{\gamma}}
\newcommand{\bdelta}{\bm{\delta}}

\newcommand{\Dcal}{\mathcal{D}}
\newcommand{\Ecal}{\mathcal{E}}
\newcommand{\Pcal}{\mathcal{P}}
\newcommand{\Acal}{\mathcal{A}}
\newcommand{\Tcal}{\mathcal{T}}

\newcommand{\Ycal}{\mathcal{Y}}
\newcommand{\Zcal}{\mathcal{Z}}
\newcommand{\Vcal}{\mathcal{V}}
\newcommand{\Wcal}{\mathcal{W}}
\newcommand{\Mcal}{\mathcal{M}}
\newcommand{\Xcal}{\mathcal{X}}
\newcommand{\Ncal}{\mathcal{N}}

\newcommand{\Hcal}{\mathcal{H}}

\newcommand{\Lcal}{\mathcal{L}}

\newcommand{\hWcal}{\widehat{\mathcal{W}}}
\newcommand{\hVcal}{\widehat{\mathcal{V}}}
\newcommand{\hXcal}{\widehat{\mathcal{X}}}
\newcommand{\hbY}{\widehat{\bm{Y}}}

\newcommand{\norm}[1]{\| #1 \|}
\newcommand{\trnorm}[1]{{\left\vert\kern-0.25ex\left\vert\kern-0.25ex\left\vert #1 
    \right\vert\kern-0.25ex\right\vert\kern-0.25ex\right\vert}}

\newcommand{\rd}{\mathrm{d}}
\newcommand{\St}{\mathrm{St}}
\newcommand{\sym}{\mathrm{sym}}

\newcommand{\mvec}{\mathrm{vec}}
\newcommand{\signal}{\mathrm{signal}}

\newcommand{\Soh}{\Sigma^{\nicefrac{1}{2}}}

\newcommand{\Goh}{\Gamma^{\nicefrac{1}{2}}}
\newcommand{\Gmoh}{\Gamma^{-\nicefrac{1}{2}}}

\newcommand{\D}{\Delta}

\newcommand{\hP}{\widehat{P}}

\newcommand{\Mat}[2][]{\mathrm{Mat}_{#1}(#2)}

\newcommand{\oh}{\nicefrac{1}{2}}

\newcommand{\mS}{\mathbb{S}}

\newcommand{\Law}{\mathrm{Law}}
\newcommand{\true}{\mathrm{true}}

\newcommand{\md}{\mathrm{d}}

\newcommand{\Ito}{It\^o}
\newcommand{\tr}{\mathrm{tr}}

\newcommand{\DLR}{\mathrm{DLR}}

\newcommand{\KBP}{\mathrm{KBP}}

\newcommand{\kbp}{\textsc{Kbp}}

\newcommand{\enkf}{\textsc{Enkf}}

\newcommand{\dlr}{\textsc{Dlr}}
\newcommand{\fom}{\textsc{Fom}}

\newcommand{\st}{\,|\,}

\newcommand{\modt}[1]{{\color{black}#1}}
\newcommand{\mods}[1]{{\color{black}#1}}
\newcommand{\modr}[1]{{\color{black}#1}}
\newcommand{\modq}[1]{{\color{black}#1}}

\newcommand{\ra}[1]{{\color{black}#1}}
\newcommand{\rb}[1]{{\color{black}#1}}
\newcommand{\rc}[1]{{\color{black}#1}}
\newcommand{\rf}[1]{#1}

\newcommand{\hY}{\widehat{Y}}

\setlist[enumerate]{leftmargin=.5in}
\setlist[itemize]{leftmargin=.5in}

\newsiamremark{remark}{Remark}
\newsiamremark{hypothesis}{Hypothesis}
\crefname{hypothesis}{Hypothesis}{Hypotheses}
\newsiamthm{claim}{Claim}
\newsiamremark{fact}{Fact}
\crefname{fact}{Fact}{Facts}

\headers{Dynamical Low-Rank Approximations for Kalman Filtering}{F. Nobile, T. Trigo Trindade}

\title{Dynamical Low-Rank Approximations for Kalman Filtering\thanks{Submitted to the editors October 28, 2025.
\funding{This work was funded by the Swiss National Science Foundation project ``Dynamical low rank methods for uncertainty quantification and data assimilation'' (n. 200518)}}}

\author{Fabio Nobile\thanks{Institute of Mathematics, \'Ecole Polytechnique F\'ed\'erale de Lausanne (EPFL), 1015 Lausanne, Switzerland (\email{fabio.nobile@epfl.ch}).}
\and Thomas Trigo Trindade \thanks{Computer, Electrical and Mathematical Sciences and Engineering, King Abdullah University of Science and Technology (KAUST), Thuwal 23955-6900, Kingdom of Saudi Arabia (\email{thomas.trigotrindade@kaust.edu.sa}).}}

\begin{document}

\maketitle

\begin{abstract}
	We propose a dynamical low rank approximation of the Kalman-Bucy process (DLR-KBP), which evolves the filtering distribution \mods{of a partially continuously observed linear SDE} on a small time-varying subspace at reduced computational cost. 
	This reduction is valid in presence of small noise and when the filtering distribution concentrates around a low dimensional subspace. 
	We further extend this approach to a DLR-ENKF process, where particles are evolved in a low dimensional time-varying subspace at reduced cost. 
	This allows for a significantly larger ensemble size compared to standard EnKF at equivalent cost, thereby lowering the Monte Carlo error and improving filter accuracy.
	Theoretical properties of the DLR-KBP and DLR-ENKF are investigated, including a propagation of chaos property.
	Numerical experiments demonstrate the effectiveness of the technique. 	
\end{abstract}


% REQUIRED

\section{Introduction}

\modt{
\rc{
Recent increases in computational power have enabled data assimilation at levels of accuracy and efficiency previously out of reach.
Despite this progress, accurate filtering of high-dimensional systems, such as those arising from the discretisation of PDEs, remains a challenging task, often incurring a prohibitive computational cost.
Several approximate filtering methods alleviate this burden, notably the Ensemble Kalman filter (\enkf), which evolves an ensemble of particles under dynamics combining the model equations with observations.

This work focuses on the settings in which the filtering distribution possesses a low-rank structure, that is, it concentrates around a low-dimensional (time-dependent) subspace. 
While not universal, this behaviour naturally arises in filtering: from a Bayesian viewpoint, for sufficiently informative observations and assimilation times the posterior contracts except along a few directions poorly constrained by the data; from a dynamics viewpoint, errors contract along stable directions, so that the uncertainty aligns with the (possibly few) unstable and neutral directions of the dynamics \cite{LyapunovPalatella2013}. 
More precisely, in the linear setting, the Kalman filter error covariance provably collapses onto the unstable–neutral subspace under suitable assumptions \cite{DegenerateCarrassi2017}; the \enkf~is then expected to track the filtering density reasonably well with a moderate number of particles. Still, a larger ensemble would reduce the Monte Carlo error and improve accuracy. This, however, is costly: while the \enkf~analysis step is confined to the span of the ensemble, the propagation of the particles takes place in the full state space, so that the computational cost scales with the full dimension rather than with that of the ``filtering subspace'' carrying most of the assimilation information. 
Identifying and exploiting this subspace is further complicated by the fact that it may vary significantly in time.}

\rc{We propose} to tackle these issues by establishing a mathematically principled framework for efficiently evolving a time-varying filtering subspace. 
The \mods{ultimate goal} is the development of a principled formulation for dynamically evolving low-rank Ensemble Kalman-type filters \mods{relying on the so-called Dynamical Low-Rank (\dlr) approximation method}. 
\modr{As a first step in this direction, in the first half of the paper we focus on} continuous-time Kalman-Bucy processes (\kbp) driven by linear-affine dynamics, and \mods{derive a \dlr~approximation thereof, named \dlr-\kbp}.  
\mods{As the} theory of linear–affine \kbp~is well established, \mods{it} provides a convenient setting for a detailed analysis of the proposed method. 

\mods{Dynamical Low-Rank approximations, originally developed in the context of matrix differential \mods{equations~\cite{DynamicalLowRKoch2007}}, evolve a low-rank, SVD-like approximation of the true matrix solution \mods{via suitable time-marching schemes~\cite{AnUnconventionCeruti2022,AProjectorSplLubich2014}}, while avoiding the repeated costly SVD truncation at each time-step. 
	\modr{DLR methods have} also been used in the context of Uncertainty Quantification for random PDEs~\cite{StabilityPropeKazash2021,ErrorAnalysisMushar2015,DynamicallyOrtSapsis2009} and SDEs~\cite{DynamicalLowZoccolan2023}.
	\modr{In this work, we follow the \rb{Dynamically Orthogonal (DO) formalism}, which was proposed in~\cite{DynamicallyOrtSapsis2009} in the context of approximation of time-dependent random fields, independently of the matrix DLR formulation~\cite{DynamicalLowRKoch2007}.}
It was subsequently observed that those formulations are equivalent~\cite{ErrorAnalysisMushar2015}, at least for non-degenerate dynamics.
Our work builds on~\cite{DynamicalLowZoccolan2023}, in which the authors formalise the notion of DLR equations and solutions \modr{for} Stochastic Differential Equations (SDEs).
The first step in our approach is to suitably generalise the DLR method to processes governed by an additional \Ito~process (the innovation process). 
We consider a framework with (conditional) mean-separation, and derive a set of DLR equations associated to these generalised processes. 
Those are subsequently applied to \mods{the} \kbp~\mods{process} to obtain the \dlr-\kbp~equations. 
\modr{In the second part of the paper, we leverage the formulation of the \dlr-\kbp\ equations} to propose a mathematically principled low-rank system of interacting particles, named \dlr-\enkf. 
\modr{We prove several important properties of the \dlr-\enkf, which support the validity of our proposed formulation.}
}

\mods{
	The main contributions of this work are summarised as follows. 
	We rigorously derive the \dlr-\kbp~system under the assumption that the filtered process has a low-rank representation.
	We prove several important properties of that system, among which its characterisation as a Gaussian process; from there, we derive a set of equations characterising the mean and the covariance (low-dimensional Riccati equation) in a reduced space, which can be computed at reduced cost. 
	\modr{
		Additionally, we quantify the approximation properties of the moments of the \dlr-\kbp~with respect to those of \kbp~using Gronwall-type bounds. 
	}
	Building on this framework, we propose a particle-based \modr{approximation} of the \modr{system}, the \dlr-\enkf, and establish its well-posedness.
	Our formulation is furthermore justified by a propagation of chaos result, linking the \dlr-\enkf~to its mean-field limit, the \dlr-\kbp.  
	\modr{We assess on some numerical test cases the performance of our \dlr-\kbp\ and \dlr-\enkf\ methods and validate the correctness of the theoretical claims.
		Our numerical results clearly show the improved accuracy of our \dlr-\enkf\ evolving $m \gg n$ particles in an $n$-dimensional subspace with respect to a plain \enkf\ which evolves $n$ high-dimensional particles, while having a comparable computational cost. 
	}

	Our framework naturally merges \dlr-type approximations with the continuous time Data Assimilation setting of Kalman-Bucy processes, and is thus part of the line of research combining (dynamical) low-rank approximation with data assimilation techniques. 
	Several other works fall into this category, which we now briefly discuss. 
	Closest to our setting of (linear-affine) time-continuous systems and time-continuous observations, we mention~\cite{LowRankApproximatedTsuzukiB2024,ComparisonOfEstYamada2021,OnANewLowRYamada2021} in which the authors propose a set of \textit{ad hoc} equations to \modr{evolve} the covariance in a low-rank fashion.
	When characterising the covariance of \modr{our} \dlr-\kbp~(and despite our setting being rooted in SDEs \modr{rather than} Linear Time-Invariant systems), we \modr{obtain the exact same equations as a consequence of the low-rank structure of the process. 
	We give therefore a rigorous justification for the Oja flow approach proposed in~\cite{LowRankApproximatedTsuzukiB2024}.}
	Our approach is furthermore supported by the well-posedness of the \dlr-\kbp~system, which we establish \modr{in this paper, as well as our result characterising the approximation properties of the \dlr-\kbp\ moments with respect to their \kbp\ counterparts}. 
	
	\rf{Several other works address} the case of discrete \rf{observations. 
	The} work~\cite{TheRankRedSchmidt2023} merges Kalman filtering and (matricial) Dynamical Low-Rank Approximations to efficiently perform the filtering of continuous-time/discrete-observations Linear Time Invariant systems.
\ra{A different approach is given by the blended particle filters framework \cite{BlendedPaMajda2014}, which builds on the (modified) quasilinear Gaussian-DO method \cite{BlendedRedSapsis2013} to perform data assimilation of chaotic large-scale systems.
Further from our setting, \cite{DataAssimilatiSonder2013,DataAssimilatiPt2Sonder2013} combines DLR approximations with Gaussian mixture models.}
}

\ra{In a broader scope, the integration of dimension reduction techniques \rf{in data assimilation} spans a vast literature.
	Approaches include using sequential SVD truncations to evolve low-rank approximations of the state ensemble \cite{PerformanceMajda2018}; tracking a few preferential -- often unstable -- directions in order to alleviate computational cost and/or improve filter stability \cite{ASingularDinh1998,StateEstBrian2001,OnDimensionSolonen2016}; multi-fidelity/\rf{multi-level} variants \cite{AnAdaptiveSilva2025}; and more recently machine learning techniques such as deep generative modeling \cite{StateobsZhuoyuan2025}, variational autoencoders \cite{EnsembleKIvo2025} and surrogate latent space models \cite{ReducedYuming2023}. 
	Our work distinguishes itself from the above approaches by (i) designing a set of principled evolution equations based on the (assumed known) model and observation dynamics, (ii) eschewing the need to perform a full-state SVD truncation at each update and (iii) being a purely online, data-free method that does not require a training phase. 
}

This paper is organised as follows. 
\Cref{sec:setup} introduces the concepts and notation relevant to Kalman-Bucy processes. 
In~\Cref{sec:dlr-for-da}, the DLR framework for SDEs of~\cite{DynamicalLowZoccolan2023} is extended to formulations suitable for Data Assimilation \modr{(including Kalman-Bucy processes)}, incorporating an innovation term. 
Once established, that framework is applied in~\Cref{sec:dlr-kbp} to the Kalman-Bucy Process, yielding the Dynamical Low-Rank Kalman-Bucy Process (\dlr-\kbp) \modr{whose well-posedness and main properties are also analysed.}
In \Cref{sec:dlr-enkf}, we propose and analyse an Ensemble Kalman-type extension of the \dlr-\kbp, the \dlr-\enkf, \modr{and present a propagation of chaos result.}
Finally, the numerical results in~\Cref{sec:numerical-experiments} validate several of the theoretical claims made in the work and demonstrate the potential of this technique.
}

\section{Setup} \label{sec:setup}

This section introduces the concepts, assumptions and related notation that will be used throughout the article. 
Denote by $\Mat{m \times n}$ the set of real-valued $m\times n$ matrices, and $\Mat[+]{m}$, resp. $\Mat[0]{m}$, the set of symmetric positive resp. semi-positive-definite $m \times m$ matrices.
We consider the evolution in time of the state $\Xcal_t \in \R^d$ of a system hereafter called the signal driven by the (linear-affine) \textit{true dynamics} 
\begin{align} 	
	\md \Xcal_t^{\mathrm{signal}} &= ( A \Xcal_t^{\mathrm{signal}} + \modt{f}) \md t + \Soh \md \widetilde{\Wcal}_t, \label{eqn:truth-model-dynamics}
\end{align}
with~$A \in \Mat{d \times d}$, $\modt{f} \in \R^d$\mods{,} $\Sigma \in \Mat[0]{d}$ \mods{and $\widetilde{\Wcal}_t$ a standard $d$-dimensional Brownian motion.}
The system is partially observed with observation process $\Zcal_t \in \R^k$ given by
\begin{align}
	\md \Zcal_t &= H \Xcal_t^{\mathrm{signal}} \md t + \Goh \md \widetilde{\Vcal}_t \label{eqn:truth-noisy-observations}, 
\end{align}
with $H \in \Mat{k \times d}$\mods{,} $\Gamma \in \Mat[+]{k}$\mods{, and $\widetilde{\Vcal}_t$ a standard $k$-dimensional Brownian motion independent of $\widetilde{\Wcal}_t$}.
The initial condition $\Xcal_0^{\mathrm{signal}}$ is assumed to be a $d$-dimensional Gaussian random variable independent of $(\widetilde{\Vcal}_t, \widetilde{\Wcal}_t)$. 

In Data Assimilation applications, the state space dimension $d$ is \modt{typically} much larger than the observation space dimension $k$. 
We set $\Zcal_0 = \bm{0} \in \R^k$ and denote $\Acal_{\Zcal_t} = \sigma(\Zcal_s, 0 \leq s \leq t)$ the sub-$\sigma$-algebra generated by the observations up to time $t$ \modt{(on a probability space that will be formally defined in~\cref{sec:dlr-for-da}).}
Broadly speaking, data assimilation concerns itself with the task of combining data streams with model knowledge in order to improve one's prediction about the state of the system or statistics thereof. 
In this work, we approach this task by computing (or approximating) the filtering measure $\eta_t = \Law(\Xcal^{\mathrm{signal}}_t | \Acal_{\Zcal_t})$. 

In the linear-affine Gaussian case of~\cref{eqn:truth-model-dynamics,eqn:truth-noisy-observations}, the filtering measure is Gaussian at all times, with mean and covariance $(m_t, P_t)$ verifying
\begin{align}
	\rd m_t &= (A m_t + \modt{f})\rd t + P_t H^{\top} \Gamma^{-1} (\rd \Zcal_t - H m_t \rd t), \label{eqn:kb-mean}
	\\
	\frac{\md }{\md t} P_t &= A P_t + P_t A^{\top} - P_t S P_t + \Sigma, \label{eqn:kb-cov}
\end{align}
where $S = H^{\top} \Gamma^{-1} H$. 
The latter equation is a~\textit{Riccati} equation; note that it does not depend on the observation process $\Zcal_t$.
%We will use the short-hand notation $\Ricc(P_t, A,\Sigma,H,\Gamma)$ to denote the right hand side.

In filtering applications, one possible approach \modt{to compute the filtering measure} consists in solving~\cref{eqn:kb-mean,eqn:kb-cov} directly, however, in situations where $d \gg 1$, updating~\cref{eqn:kb-cov} may be prohibitively expensive.
This situation commonly arises in geosciences such as weather forecasting~\cite{AnEnsemBuehner2017} or oceanography~\cite{MultilevelDaBeiser2025}.
%, inference of spatio-temporal Gaussian processes~\cite{InfiniteDimSarkka2012}, high-dimensional inverse problems~\cite{EnsembleKIglesias2013}, etc.
\modt{An alternative} approach consists in seeking a diffusion process $\Xcal_t$ whose probability law $\mu_t$ equals that of the true filtering measure. 
This is of particular interest when combined with particle methods, in which the empirical measure produced by the particles provides an approximation of the true filtering \modt{measure}. 
\modt{
Let $\mu$ be any probability measure on $\R^d$ (such as the law of a random variable), and define its covariance
\begin{equation*} 
	P_{\mu} = \int_{\R^d}  (x - m)(x - m)^\top \md \mu(x), \quad \text{where}\, m = \int_{\R^d} x \md \mu(x).
\end{equation*}
When the measure depends on time and no confusion arises, we will denote $P_t \equiv P_{\mu_{t}}$.}
In what follows, $\Wcal_t$ (resp. $\Vcal_t$, $\Xcal_0$) is an independent copy of $\widetilde{\Wcal}_t$ (resp. $\widetilde{\Vcal}_t$, $\Xcal^{\mathrm{signal}}_0$).
Then, consider the \modt{Kalman-Bucy Process~(\kbp)}
\begin{equation} \label{eqn:vkb}
	\rd \Xcal^{\mods{\KBP}}_t = (A \Xcal_t^{\mods{\KBP}}  + \modt{f}) \rd t + \Soh \rd \Wcal_t + P_{t} H^{\top} \Gamma^{-1} \left[ \rd \Zcal_t - (H \Xcal_t^{\mods{\KBP}} \rd t + \Goh \rd \Vcal_t)  \right] ,
\end{equation}
where $\mu_t = \mathrm{Law}(\Xcal_t \st \Acal_{\Zcal_t} )$. 
Since~\cref{eqn:vkb} is a inhomogeneous Ornstein-Uhlenbeck process, its Gaussian distribution is characterised by its mean and covariance, which can be shown to verify~\cref{eqn:kb-mean,eqn:kb-cov}, and hence $\mu_t = \eta_t$ \modt{at all times~\cite{OnTheMathematBishop2023}}.

\modt{
	In order to adhere to the natural setting in data assimilation where there effectively exists a unique ``true signal'', all of our discussion is from now on framed conditionally on $\Xcal_t^{\signal}$; that is, we assume to be working with a continuous realisation of $\Xcal_t^{\signal}$, and view $t \mapsto \Xcal_t^{\signal}$ as a continuous map $\R_{\geq 0} \to \R^d$.
	This is a reasonable assumption since the solution to~\cref{eqn:truth-model-dynamics} has continuous sample paths almost surely.
	Note that in turn, $H \Xcal_t^{\signal}$ is continuous too.
	The probability space on which we carry out the construction of the \kbp~(and its variants) therefore need not be linked to the probability space of the process $\Xcal_t^{\signal}$.
}

\begin{remark}
	Other diffusion processes also evolve the filtering distribution $\eta_t$ and could be considered instead.
	In particular we mention the popular ``deterministic'' Kalman-Bucy filter~\cite{OnTheMathematBishop2023,ADetermSakov2008}, given by
	\begin{equation}
		\rd \Xcal_t^{\mods{\KBP}} = (A \Xcal_t^{\mods{\KBP}} + \modt{f})\rd t + \Soh \rd \Wcal_t + P_{t} H^{\top} \Gamma^{-1} \left[ \rd \Zcal_t - H \left( \frac{\Xcal_t^{\mods{\KBP}} +  \modt{\E[\Xcal_t^{\mods{\KBP}} \st \Acal_{\Zcal_t}]}}{2} \right)  \rd t  \right]. \label{eqn:dkb}
	\end{equation}
\end{remark}

\subsection{Additional notation}

We briefly introduce some additional notation that will be used throughout the text. 
The (real) Stiefel manifold is denoted by $\St(d,R) = \{ U \in \Mat{d \times R}, U^{\top} U = \bm{I}_R \}$. 
For $A,B \in \Mat{n \times m}$, the Frobenius scalar product $\langle A,B \rangle_F = \tr(A^{\top}B)$ induces the Frobenius norm $\langle A,A \rangle_F = \norm{A}^2_F$.
Note that $\tr(A^{\top}B) \leq \norm{A}_F \norm{B}_F$ for $A,B \in \Mat{n \times m}$.
We will also use the property
\begin{equation} \label{eqn:tr-lbda-ineq}
	\lambda_{\min}(A) \tr(B) \leq \tr(A^{\top} B) \leq \lambda_{\max}(A) \tr(B),
\end{equation}
which holds true whenever $A$ is symmetric and $B \in \Mat[0]{d \times d}$.
The cyclical property $\tr(ABC) = \tr(CAB)$ holds true whenever the matrix dimensions match.
For $A \in \Mat{m \times n}, B \in \Mat{p \times q}$, the Kronecker product $A \otimes B \in \Mat{mp \times nq}$ verifies $(A \otimes B)_{pr + v, qs + w} = A_{rs} B_{vw}$.
For $A \in \Mat{m \times n}$, the vectorisation $\mathrm{vec}(A)$ consists in stacking the columns of $A$ on top of each other. 
Note the useful identity $\mathrm{vec}(ABC) = (C^{\top} \otimes A) \mathrm{vec}(B)$.
We will occasionally index a matrix $A(i,j) = A_{ij}$, and use MATLAB-style \modt{notation} $A(i,:)$, resp. $A(:,i)$, to access the $i$-th row resp. column.
For $A,B \in \Mat[0]{d \times d}$, we consider the (partial) Loewner order $A \geq (>) B$ if $A - B \in \Mat[0(+)]{d \times d}$.
Importantly,
\begin{equation} \label{eqn:loew-frob}
	A \geq B \implies \norm{A}^2_F \geq \norm{B}^2_F
\end{equation}
(this can be verified via the definition of trace as sum of eigenvalues and using Weyl's theorem.)
Finally, \mods{for} $A \in \Mat{m \times n}$, \mods{we denote by} $\Tcal_R(A) \in \mathrm{argmin} \{ \norm{A - B}_F, \mathrm{rank}(B) \leq R \}$ \modt{a best rank-$R$ approximation of $A$; it is unique if the $R$-th singular value $\sigma_{R}(A)$ \modr{is strictly larger than} $\sigma_{R+1}(A)$}. 

For two given processes $\Xcal_t, \Ycal_t$, the \textit{quadratic variation} is denoted by $[\Xcal_t]_t$, while the \textit{quadratic co-variation} is denoted~$[\Xcal_t, \Ycal_t]_t$.
The notation $\Xcal_t(i)$ indexes the $i$-th component of that process.

\section{Dynamical Low-Rank framework for Data Assimilation} \label{sec:dlr-for-da}

Our aim is to propose DLR equations that \rf{approximate} processes like~\cref{eqn:vkb} or~\cref{eqn:dkb}.
To this end, \modt{we formulate the \rf{filtering} problem in a more general form, so as to accommodate both processes~\cref{eqn:vkb} and~\cref{eqn:dkb} (and possibly more).
	Let $(\Omega, \Acal , \bbP; (\Acal_t)_{t \geq 0})$ be a complete filtered probability space, also right-continuous, $\{\overline{\Wcal}_t, \Zcal_t\}_{t \geq 0}$ two independent, $\Acal_t$-adapted processes with $\overline{\Wcal}_t$ a $\bar{d}$-dimensional Brownian motion, with $\bar{d} \geq d$ (this to accomodate processes like~\cref{eqn:vkb}, which depend on two Brownian motions $\Wcal_t$ and $\Vcal_t$; in that situation, $\overline{\Wcal}_t$ will be the concatenation of both Brownian motions), and $\Zcal_t$ a $k$-dimensional \Ito~process with a.s. continuous sample paths (consequently $\Acal_{\Zcal_t}$ is right-continuous) verifying
	%let $\widetilde{\Wcal}_t$ and $\widetilde{\Vcal}_t$ be $d$-dimensional resp. $k$-dimensional pairwise independent, $\Acal_t$-adapted standard Wiener processes. }
	%Let 
	\begin{equation}
		\md \Zcal_t = \Hcal_t \md t + \gamma_t \md \widetilde{\Vcal}_t,
	\end{equation}
	with $\Hcal_t$, $\gamma_t$ continuous, adapted and where $\E[\int_{0}^T |\Hcal_t| \md t] < \infty$, $\E[\int_{0}^T |\gamma_t|^2 \md t] < \infty$ and $\{\widetilde{\Vcal}_t\}_{t \geq 0}$ a $k$-dimensional Brownian motion.
	We consider the \modt{(McKean-Vlasov)} SDE
	\begin{equation} \label{eqn:general-SDE}
		\md \Xcal_t^{\ra{\true}} = a(t, \Xcal_t^{\ra{\true}}, \mu_t^{\ra{\true}}) \md t + b(t, \Xcal_t, \mu_t^{\ra{\true}}) \md \overline{\Wcal}_t + c(t, \Xcal_t^{\ra{\true}}, \mu_t^{\ra{\true}}) \md \Zcal_t,
	\end{equation}
	\modt{where $\mu_t^{\ra{\true}} = \Law(\Xcal_t^{\ra{\true}}\st \Acal_{\Zcal_t})$}, $a : \R_{+} \times \R^d \times \Pcal(\R^d) \to \R^d$, $b : \R_{+} \times \R^d \times \Pcal(\R^d) \to \R^{d \times \bar{d}}$, $c : \R_{+} \times \R^d \times \Pcal(\R^d) \to \R^{d \times k}$.
	In what follows, to alleviate notation, we will denote the coefficients \modr{in~\cref{eqn:general-SDE}} by $a(t, \omega), b(t, \omega), c(t, \omega)$.	
	\ra{\rf{As mentioned above, \cref{eqn:general-SDE} should be seen as a generalisation of \cref{eqn:vkb,eqn:dkb}.
For example, \cref{eqn:vkb} can be recovered by setting $a(t, \Xcal_t^{\true}, \mu_t^{\true}) = A \Xcal_t^{\true} + f - P_t H^{\top} \Gamma^{-1} H \Xcal_t^{\true}$, $b(t, \Xcal_t^{\true}, \mu_t^{\true}) = [\Soh, -P_t H^{\top} \Gmoh]$ and $c(t, \Xcal_t^{\true}, \mu_t^{\true}) = P_t H^{\top} \Gamma^{-1}$, with $P_t = \E[(\Xcal_t^{\true}) (\Xcal_t^{\true})^{\top} \st \Acal_{\Zcal_t}]$ depending on the conditional law of $\Xcal_t^{\true}$.}}
	We will make a repeated use of the conditional expectation taken with respect to the sub-$\sigma$-algebra $\Acal_{\Zcal_t}$,} which for $f \in L^1(\Omega)$ is denoted by  $\E[f | \Acal_{\Zcal_t} ]$.
Furthermore, define 
\begin{equation*} %\label{eqn:l2zmean-conditional}
	L^2_{0,\Acal_{\Zcal_t}}(\Omega) = \{Z \in L^2(\Omega), \E[Z \st \Acal_{\Zcal_t}] = 0 \}, 	\end{equation*}
\mods{and, for $Y \in L^1(\Omega)$,}
\begin{equation*}
	\mods{Y^{\star} = Y - \E[Y \st \Acal_{\Zcal_t}].}
\end{equation*}
We further assume that \mods{\cref{eqn:general-SDE} admits a unique strong solution and that} sample paths \rf{of $\Xcal_t^{\true}$} are continuous almost surely.

\ra{The purpose of this section is to propose a set of equations characterising an auxiliary process $\Xcal_t^{\DLR}$ admitting a low-rank decomposition and which approximates $\Xcal_t^{\true}$. 
	We begin by detailing the requirements we impose on that auxiliary process.  
	\begin{enumerate}[label=(\roman*)]
		\item (Low-rank decomposition)  \label{item:lrdecomp}
			The process $\Xcal_t^{\DLR}$ verifies
			\begin{equation} \label{eqn:low-rank-ansatz}
				\Xcal_t^{\ra{\DLR}} = U^0_t + \sum_{i=1}^R U_t^i Y_t^i = U^0_t + \bU_t \bY_t^{\top}, \quad \forall t \geq 0,
			\end{equation}
			where $U^0_t$ and $\bU_t = [U_t^1, \ldots, U_t^R]$, \rf{hereafter called physical modes,} are $\Acal_{\Zcal_t}$-measurable (and therefore independent of $\overline{\Wcal}_t$), $\bY_t = [Y_t^1, \ldots, Y_t^R]$, \rf{hereafter called stochastic modes, are} such that $Y_t^i \in L^2_{0,\Acal_{\Zcal_t}}(\Omega)$ {for $i = 1, \ldots, R$}.
		\item (Modes characterisation)  \label{item:modecharacterisation}
			The modes $(U^0_t, U^i_t, Y^i_t)$ {satisfy stochastic differential equations} \rf{of the form}
			\begin{align} \label{eqn:modes-eqs}
				\md U_t^0 &= \alpha_t^0 \md t + \beta_t^0 \md \Zcal_t, & 
				\md U_t^i &= \alpha_t^i \md t + \beta_t^i \md \Zcal_t,
					  & \md Y_t^i &= \gamma_t^i \md t + \delta_t^i \md \overline{\Wcal}_t + \varepsilon_t^i \md \Zcal_t,
			\end{align}
			where the processes {$\{\alpha_t^j\}_{t \geq 0}$ ($\R^d$-valued), $\{\beta_t^j\}_{t \geq 0}$ ($\Mat{d \times k}$-valued)}, for $j=0, \ldots, R$ are {continuous $\Acal_{\Zcal_t}$-adapted}, {whereas} $\gamma_t^i \in \R$, $\delta_t^i \in \Mat{1 \times \bar{d}}$ and $\varepsilon_t^i \in \Mat{1 \times k}$ for $i = 1, \ldots, R$ are {continuous {$\Acal_{t}$-adapted}}.
		\item (\rf{Physical modes constraint}) \label{item:coeffconstraint}
			\rf{The physical modes $U_t^i$ satisfy the orthogonality conditions}	
			\begin{align} \label{eqn:mod-do}
				(U^i_t)^{\top} \alpha_t^j &= 0 & (U^i_t)^{\top} \beta_t^j &= 0 & 1 \leq i,j \leq R.
			\end{align}
	%	\item (Consistency in the idealised regime) \label{item:consistencyidealisesregime}
	%		If the true process~\eqref{eqn:general-SDE} admits the decomposition $\Xcal_t^{\true} = \tilde{U}_t^0 + \sum_{i=1}^R \tilde{U}_t^i \tilde{Y}_t^i$ for $t \geq 0$ and some $R \geq 1$, and furthermore $(\tilde{U}_t^0, \tilde{U}_t^i, \tilde{Y}_t^i)$ satisfy \ref{item:modecharacterisation} and \ref{item:coeffconstraint}, then $\Xcal_t^{\DLR} = \Xcal_t^{\true}$ and $U_t^j = \tilde{U}_t^j$, $Y_t^i = \tilde{Y}_t^i$ for $i = 1, \ldots, R$, $j=0,\ldots,R$. 
	\end{enumerate}

Concerning \ref{item:modecharacterisation}, the stochastic processes $Y_t^i$, $i=1, \ldots,R$ are imposed to have zero conditional mean $\E[Y_t^i \st \Acal_{\Zcal_t}] = 0$ in order to avoid redundancy in the representation~\eqref{eqn:low-rank-ansatz} with the mean process $U^0_t$.
This ensures that physical modes may recover information from the observation process, while leaving the stochastic fluctuations of the Brownian motion $\overline{\Wcal}_t$ to be captured solely by the stochastic modes.
Note that, by separating the mean, the approximation is suboptimal with respect to the $(R+1)$ best low-rank approximation, but as will be discussed below, this construction displays favourable properties in the case of linear-affine models. 

Condition \ref{item:coeffconstraint} is imposed in order to uniquely characterise the evolution equations of the modes.  
In the usual DLR setting, this is obtained by imposing the \rb{DO condition} \cite{DynamicallyOrtSapsis2009} (also called \textit{gauge conditions} in the matricial DLR context~\cite{AProjectorSplLubich2014}), which reads $\dot{\bU}_t^{\top} \bU_t = 0$ ({alternatively, in Dual DO~\cite{DualDynamicallMushar2018}, $\dot{\bY}_t^{\top} \bY_t = 0$}).
This is a priori not possible here since $\bU_t$ \rf{(resp. $\bY_t$) is not a differentiable process}.
\rf{We replace such condition with $\md \bU_t^{\top} \bU_t = 0$ which leads to \cref{eqn:mod-do}}.}

\rb{We now present a system \rf{of equations to evolve $(U_t^0, U_t^i, Y_t^i)$} verifying \rf{Conditions \ref{item:lrdecomp}--\ref{item:coeffconstraint}.} 
	Before doing so, we introduce the following notation, which will be useful in what follows.
	Define the Gram matrices 
$(M_{\bU_t})_{ij} = (U_t^i)^\top U_t^j$ and $(M_{\bY_t})_{ij} = \E [Y_t^i Y_t^j \st \Acal_{\Zcal_t}]$; furthermore denote by $\Pi_{\bU_t} = \bU_t (M_{\bU_t})^{-1} \bU_t^{\top}$ the orthogonal projector onto the range of $\bU_t$, and let $\Pi_{\bU_t}^{\perp} = I - \Pi_{\bU_t}$ denote the projector onto the orthogonal complement of $\bU_t$.}
\ra{
	\rf{With this in place, we propose the following \textbf{mean-separated DO equations:}}
	\begin{align}
		\md U_t^0 &= \E[\rf{a_t} \st \Acal_{\Zcal_t}] \rd t + \E[\rf{c_t} \st \Acal_{\Zcal_t} ] \rd \Zcal_t, \label{eqn:zeromean-mean} \\
    \md \bU_t &= {\Pi^{\perp}_{\bU_t}} \E \left[( \rf{a_t^{\star}} - \rf{G_t} ) (\bY_t M_{\bY_t}^{-1} ) \st \Acal_{\Zcal_t}  \right] \md t + \rc{\sum_{j=1}^k{\Pi^{\perp}_{\bU_t}} \mathbb{E} [ (c^{\star}_t)_{:,j} (\bY_t M_{\bY_t}^{-1} ) \st \Acal_{\Zcal_t}] \md \Zcal_t^j} , \label{eqn:zeromean-phys} \\
		\md \bY_t^{\top} &= M_{\bU_t}^{-1} \bU_t^{\top} \rf{a_t^{\star}} \md t + M_{\bU_t}^{-1} \bU_t^{\top} \rf{b_t} \md \overline{\Wcal}_t + M_{\bU_t}^{-1} \bU_t^{\top} \rf{c_t^{\star}} \md \Zcal_t, \label{eqn:zeromean-stoch} 
	\end{align}
	where
	\begin{equation} \label{eqn:correction-term-general}
		\rf{G_t} =
		\sum_{i=1}^R
		\sum_{s,s^\prime = 1}^k 
		\left[
		{\Pi^{\perp}_{\bU_t}}  \mathbb{E} [\rf{c^{\star}_t} (\bY_t  M_{\bY_t}^{-1})_{i} \st \Acal_{\Zcal_t} \right]
		(:,s) 
		\cdot 
    \left[ (M_{\bU_t}^{-1} \bU_t^{\rc{\top}})_i \rf{c^{\star}_t} \right](s') 
		\cdot
		\rd [\Zcal_t^s, \Zcal_t^{s^{\prime}}]_t,
	\end{equation}
	\rf{where $a_t$, $b_t$, $c_t$ denote $a$, $b$, $c$ evaluated at $(t, \Xcal_t^{\DLR}, \Law(\Xcal_t^{\DLR}))$ with $\Xcal_t^{\DLR} = U_t^0 + \bU_t \bY_t^{\top}$.}
	%where in~\cref{eqn:zeromean-mean,eqn:zeromean-phys,eqn:zeromean-stoch,eqn:correction-term-general} the terms $a,b,c$ are evaluated at $(t, \Xcal_t^{\DLR}, \Law(\Xcal_t^{\DLR}))$ with $\Xcal_t^{\DLR} = U_t^0 + \bU_t \bY_t^{\top}$.

	\rf{The justification of \cref{eqn:zeromean-mean,eqn:zeromean-phys,eqn:zeromean-stoch} comes from the consistency result given by the following theorem, which roughly says ``if $\Xcal_t^{\true}$ is exactly low-rank, $\Xcal_t^{\true} = \tilde{U}_t^0  + \sum_{i=1}^R \tilde{U}_t^i \tilde{Y}_t^i$, then $(\tilde{U}_t^0, \tilde{U}_t^i, \tilde{Y}_t^i)$ must satisfy \cref{eqn:zeromean-mean,eqn:zeromean-phys,eqn:zeromean-stoch}''.}
\begin{theorem} \label{th:meansepDO}
	\rf{The system \cref{eqn:zeromean-mean,eqn:zeromean-phys,eqn:zeromean-stoch} satisfies Conditions \ref{item:lrdecomp}--\ref{item:coeffconstraint}. 
		Moreover, if for $t \geq 0$, $\Xcal_t^{\true}$ admits the decomposition $\Xcal_t^{\true} = \tilde{U}_t^0 + \sum_{i=1}^R \tilde{U}_t^i \tilde{Y}_t^i$ for some $R \geq 1$, with $(\tilde{U}_t^0, \tilde{U}_t^i, \tilde{Y}_t^i)$ processes of type \ref{item:modecharacterisation} satisfying the corresponding orthogonality conditions \ref{item:coeffconstraint}, and under the technical conditions of \cref{lem:cond-expec}, then $(\tilde{U}_t^0, \tilde{U}_t^j, \tilde{Y}_t^j)$ solve \cref{eqn:zeromean-mean,eqn:zeromean-phys,eqn:zeromean-stoch}.
	}
\end{theorem}
\begin{proof}
	See Appendix \ref{app:a}.
\end{proof}
}

We conclude this section by noting that the (conditional) mean and covariance of the DLR solution are characterised as follows.
\begin{lemma}
	Let~$\Xcal_t^{\ra{\DLR}} = U_t^0 + \bU_t \bY_t^{\top}$ \modr{be} a solution to~\cref{eqn:zeromean-mean,eqn:zeromean-phys,eqn:zeromean-stoch}. 
	Then, the conditional mean $m_t$ is $\E[\Xcal_t^{\ra{\DLR}} \st \Acal_{\Zcal_t}] = U^0_t$ and the \rf{conditional covariance $P_t = \E[(\Xcal_t^{\DLR} - m_t) (\Xcal_t^{\DLR} - m_t)^{\top} \st \Acal_{\Zcal_t}]$} verifies
	\begin{equation} 
		P_t = \bU_t M_{\bY_t} \bU_t^{\top}. \label{eqn:cov-lr}
	\end{equation}
\end{lemma}
In particular, the covariance has (co-)range in $\mathrm{range}(\bU_t)$.

\section{DLR Kalman-Bucy Process} \label{sec:dlr-kbp}

We now apply the mean-separated DO equations\modt{~\eqref{eqn:zeromean-mean} to~\eqref{eqn:zeromean-stoch} to the specific case of the \kbp~\eqref{eqn:vkb}}.
Reordering the terms~\modt{in~\cref{eqn:vkb}} yields
\begin{align*}
	\rd \Xcal_t^{\ra{\KBP}} &= ((A - P_t S) \Xcal_t^{\ra{\KBP}} + \modt{f}) \rd t + [\Soh, - P_t H^{\top} \Gmoh] [\rd \Wcal_t, \rd \Vcal_t]^{\top} + P_t H^{\top} \Gamma^{-1} \rd \Zcal_t \\
				&=  \rf{a_t} \rd t + \rf{b_t} \rd \overline{\Wcal}_t  + \rf{c_t} \rd \Zcal_t,
\end{align*}
where $\md\overline{\Wcal}_t = [\rd \Wcal_t, \rd \Vcal_t]^{\top} $ is the concatenation of the two Brownian motions in order to adhere to the setting of~\cref{eqn:general-SDE}.
\modr{Introducing the notation $\tilde{A} = A - P_{t} S$, observe} that 
\begin{align*}
	\E[\rf{a_t} \st \Acal_{\Zcal_t}] &= \tilde{A} \E [ \Xcal_t^{\ra{\KBP}} \st \Acal_{\Zcal_t} ] + \modt{f} , 
					 & \E[\rf{c_t} \st \Acal_{\Zcal_t}] &= P_t H^{\top} \Gamma^{-1}, 
					 & \rf{b_t} = [\Soh, - P_t H^{\top} \Gmoh], \\
	\rf{a_t^{\star}} &= \tilde{A}  \Xcal_t^{\star}, 
		  & \rf{c_t^{\star}}& = 0.
\end{align*}
Applying the DLR~equations~\eqref{eqn:zeromean-mean}-\eqref{eqn:zeromean-stoch} yields the \dlr-\kbp~equations \modt{given below}. 
Focussing on~\cref{eqn:zeromean-phys}, the dynamics \modt{simplifies} substantially to $\md \bU_t = \modt{\Pi^{\perp}_{\bU_t}} A \bU_t \md t $, as \modt{both} the correction term $\rf{G_t}$ \modt{and} the noise source vanish owing to $\rf{c_t^{\star}} = 0$, and \modt{moreover} $\modt{\Pi^{\perp}_{\bU_t}} \tilde{A} = \modt{\Pi^{\perp}_{\bU_t}} A$ owing to $\modt{\Pi^{\perp}_{\bU_t}} P_t S = 0$.
In that situation, $\bU_t$ is now absolutely continuous and without loss of generality we may assume that $\bU_0 \in \mathrm{St}(d,R)$, which implies $\bU_t \in \mathrm{St}(d,R)$ for \modt{all} $t \geq 0$, and hence $M^{-1}_{\bU_t} = \bm{I}_{R}$.
Thus, the \textbf{DLR-KBP equations} become
\begin{align}
	\md U_t^0 &= (A U_t^0 + \modt{f}) \rd t + P_t H^{\top} \Gamma^{-1} (\rd \Zcal_t - H U^0_t \rd t)  \label{eqn:kb-dlr-U0} \\
	\md \bU_t &= \modt{\Pi^{\perp}_{\bU_t}} A \bU_t \md t \label{eqn:kb-dlr-U} \\
	\md \bY_t^{\top} &= \bU_t^{\top} \tilde{A} \modt{\bU_t \bY_t^{\top}} \md t + \bU_t^{\top} \Soh \rd \Wcal_t - \bU_t^{\top} P_{t} H^{\top} \Gmoh \md \Vcal_t, \label{eqn:kb-dlr-Y}
\end{align}
\modt{where the covariance $P_t$ is characterised by~\cref{eqn:cov-lr}.}
It is noteworthy that the observation process has no impact on the physical modes $\bU_t$ and the stochastic modes $\bY_t$. 
This however is in line with the fact that the evolution equation of the covariance of the \kbp~\cref{eqn:kb-cov} is independent \mods{of} the observations. 

%\modt{
%\begin{remark}
%	The fact that~\cref{eqn:kb-dlr-U} does not depend on the observation process is a consequence of $\md \Zcal_t$ being left-multiplied by $P_t$, hence in the range of $\bU_t$.
%	This holds true even when not separating the mean in the DLR construction.
%	On the other hand, the interest of using a mean-separated DO formalism becomes apparent for linear-affine dynamics. 
%	Notice indeed that~\cref{eqn:kb-dlr-U0} for $U^0_t$ corresponds to~\cref{eqn:kb-mean} for the mean of the \kbp.
%	Had the mean not been treated separately, the affine dynamics would in general not be recovered exactly, leading to a mean equation different from that of \kbp.	
%\end{remark}
%}

The \ra{\dlr-\kbp\ \cref{eqn:kb-dlr-U0,eqn:kb-dlr-U,eqn:kb-dlr-Y} is in general \emph{not}} equivalent to \ra{the \kbp\ }~\cref{eqn:vkb}. 
Indeed, \ra{it is a direct check to verify that the \dlr-\kbp\ \rf{solution} $\Xcal_t^{\DLR-\KBP} = U_t^0 + \bU_t \bY_t^{\top}$ solves}
\begin{equation} \label{eqn:dlr-vkb}
	\rd \Xcal_t = (A \Xcal_t +  \modt{f}) \rd t + \modt{\Pi_{\bU_t}} \Soh \rd \Wcal_t + P_{t} H^{\top} \Gamma^{-1} \left[ \rd \Zcal_t - (H \Xcal_t \rd t + \Goh \rd \Vcal_t)  \right],
\end{equation}
\ra{with $U_t^0, \bU_t, \bY_t$ solving \cref{eqn:kb-dlr-U0,eqn:kb-dlr-U,eqn:kb-dlr-Y}, and where \rf{for leaner notation, we denote $\Xcal_t \equiv \Xcal_t^{\DLR-\KBP}$.}}
\ra{It is clear that equality between the right-hand-sides \cref{eqn:vkb} and \cref{eqn:dlr-vkb} holds} only if $\modt{\Pi_{\bU_t}} \Sigma  = \Sigma$ for \modt{all} $t \geq 0$. 
\ra{The \dlr-\kbp\ is therefore expected to accurately approximate the \kbp\ in the case where $\Pi_{\bU_t} \Soh \approx \Soh$, or equivalently, when the \textit{orthogonal defect}
\begin{equation} \label{eqn:mea-sigma}
	\ra{\|\Pi^{\perp}_{\bU_t} \Soh\|_F \eqqcolon \varepsilon_R(t)}
\end{equation}
remains suitably small; this statement is made precise in \cref{ssub:error-approx}, which provides a Gronwall-type quantification of the deviation of the \dlr-\kbp\ mean and covariance with respect to their \kbp\ counterparts that explicitly depends on the magnitude of the orthogonal defect.
A small orthogonal defect is therefore a sufficient condition to ensure that the \kbp\ admits a suitable low-rank approximation.
Such low-rank-approximable filtering processes naturally occur in data assimilation problems, as evidenced for example by processes displaying a strongly anisotropic (posterior) covariance; or, for the Ensemble Kalman filter \rf{by} the fact that a moderate number of particles (hence rank) is often sufficient to perform signal-tracking.

The orthogonal defect is evidently \emph{not} small for isotropic models $\Sigma = \sigma I$ with moderate/large $\sigma$. 
In data assimilation settings, this \rf{may naturally occur} when certain scales of motion are under-resolved by the model, and accounting for those errors is necessary \cite{AccountingLewis2015}. 
Still, several distinct data assimilation settings are naturally compatible with small orthogonal defects; we list some of them below. 
\begin{itemize}
	\item Small and zero-noise models (``perfect models'') are commonly considered in data assimilation problems, either as an independent object of study \cite{WellPosedKelly2014}; \rf{or} in idealised test regimes for twin experiments or Observing System Simulation Experiments (OSSEs) \cite{AssessingSzunyogh2005}; or because the problem enables high confidence in the model dynamics (e.g. short synchronisation periods for smart grid state estimation models \cite{SmartGridTodescato2020}). %; because the model noise is negligible compared to the observation noise. 
Pragmatically, \rf{identifying a model together with} the observation/model noises can be challenging, and methods such as the Strong Constraint 4D-Var \cite{DataAsEvensen2022} (for smoothing) forego this by considering zero-model noise with noisy observations regime.  
	\item The noise covariance may itself \rf{have} a low-rank structure $L_t L_t^{\top} = \Soh$, which can in principle be captured by a sufficiently large basis $\bU_t$. 
This setting is often assumed in works relying on rank-reduced formalisms \cite{TheRankRedSchmidt2023,TheRedRankGillijns2006}.  
One example arises in problems where the disturbances are structured and the noise enters through a low-dimensional forcing map such as Simultaneous Localisation And Mapping (SLAM) problems for autonomous driving, where the augmented state process only injects noise for the vehicle position \cite{ASolutionDissanayake2001}. 
While there is no a priori guarantee that $\bU_t$ aligns with the low-rank $L_t L_t^{\top} = \Soh_t$, alternative formulations that enrich $\bU_t$ by a basis tracking $L_t$ could be of interest, but will not be investigated further in this work. 
\item In yet another context, SPDE-based data assimilation problems often focus on the case where operator and noise kernels share a common spectrum \cite{AccuracyBrett2013}; as we show below, under suitable assumptions $\bU_t$ aligns with the dominant eigenmodes of the operator, and, provided the model noise spectrum decays quickly, can thus also capture the dominant noise effects.
\end{itemize}
}

\modr{
	In the following sections, \rf{we analyse the main properties of the \dlr-\kbp\ under} the assumption that the \rf{system~\cref{eqn:kb-dlr-U0,eqn:kb-dlr-U,eqn:kb-dlr-Y} admits} a unique strong solution.
	We postpone the discussion on the well-posedness of the system at the end of~\cref{ssec:red-ricc}.
}

\subsection{Basic properties of DLR-KBP}

We now explore several properties of~\cref{eqn:kb-dlr-U0,eqn:kb-dlr-U,eqn:kb-dlr-Y}.
In particular, we show that the \dlr-\kbp~is a Gaussian process, with mean and covariance consistent with those of the original~\kbp.

\subsubsection{Gaussian law}

\mods{
	As it will be useful in the following, the (conditional) zero-mean part $\Xcal_t^{\star} = \Xcal_t - \E[\Xcal_t \st \Acal_{\Zcal_t}]$ verifies the following SDE
	\begin{equation} \label{eqn:dXtstar}
		\md \Xcal_t^{\star} = A \Xcal_t^{\star} \rd t + \modt{\Pi_{\bU_t}} \Soh \rd \Wcal_t - P_{t} H^{\top} \Gamma^{-1} H \Xcal_t^{\star} \rd t + P_{t} H^{\top} \Gmoh \rd \Vcal_t.
	\end{equation}
}
 
We readily obtain the following characterisation of the \rf{(conditional)} mean and covariance of the \mods{\dlr-\kbp} process.
%$\Xcal_t = U^0_t + \bU_t \bY_t^{\top}$}. %by \mods{applying \modt{\Ito's}~lemma}. 
%\begin{lemma} \label{lem:dlr-ricatti}
%	Let $P_t = \E[\Xcal_t^{\star} (\Xcal_t^{\star})^{\top} \st \Acal_{\Zcal_t}] = \E[(\bU_t \bY_t^{\top})  (\bY_t \bU_t^{\top}) \st \Acal_{\Zcal_t} ]$, and $\tilde{A} = A - P_t S$. Then $P_t$ satisfies
%        \begin{equation} \label{eqn:dlr-vkbp-cov-tmp}
%		\partial_t P_t = \tilde{A} P_t + P_t \tilde{A}^{\top} + \modt{\Pi_{\bU_t}} (\Sigma + P_t S P_t) \modt{\Pi_{\bU_t}}.
%        \end{equation}
%\end{lemma}

\begin{proposition} \label{prop:mean-cov-vkb}
        The mean $U_t^0$ and covariance $P_t$ of the \dlr-\kbp~satisfy
        \begin{align}
		\md U_t^0 &=  (A U_t^0 + \modt{f}) \rd t + P_t H^{\top} \Gamma^{-1} (\rd \Zcal_t - H U_t^0), \label{eqn:dlr-vkbp-mean} \\
		\partial_t P_t &= A P_t + P_t A^{\top} - P_t S P_t + \modt{\Pi_{\bU_t}} \Sigma \modt{\Pi_{\bU_t}}. \label{eqn:dlr-vkbp-cov}
        \end{align}
\end{proposition}
\begin{proof}
	\mods{For the conditional mean},~\cref{eqn:dlr-vkbp-mean} is the same as~\cref{eqn:kb-dlr-U0}. 
	\mods{For the covariance, applying \Ito's lemma to $\Xcal_t^{\star} (\Xcal_t^{\star})^{\top}$ yields
	\begin{align*}
		\md (\Xcal_t^{\star} (\Xcal_t^{\star})^{\top}) &= \md \Xcal_t^{\star} (\Xcal_t^{\star})^{\top} + \Xcal_t^{\star} (\md \Xcal_t^{\star})^{\top} + \md [\Xcal_t^{\star}, (\Xcal_t^{\star})^{\top}]_t 
		\\
							       &=
							       (A - P_t S) \Xcal_t^{\star} (\Xcal_t^{\star})^{\top} \rd t 
							       + \Xcal_t^{\star} (\Xcal_t^{\star})^{\top} (A - P_t S) \rd t
							       + \bJ(\md \Wcal_t) + \bK(\md \Vcal_t)
							       \\
							       &\hspace{2em} + \Pi_{\bU_t} \Sigma \Pi_{\bU_t} \md t + P_t S P_t \md t,
	\end{align*}
	where $\bJ(\md \Wcal_t) \in \R^{d \times d}$ resp. $\bK(\md \Vcal_t) \in \R^{d \times d}$ contain the terms depending on $\Wcal_t$ resp. $\Vcal_t$. 
	Those terms vanish when taking conditional expectations. 
	Furtheremore, since $P_t = \E[\Xcal_t^{\star} (\Xcal_t^{\star})^{\top} \st \Acal_{\Zcal_t}]$,
	taking conditional expectations and re-arranging yields~\cref{eqn:dlr-vkbp-cov}.
}

        %For the covariance, starting from~\cref{eqn:dlr-vkbp-cov-tmp},
        %\begin{align*}
        %        \partial_t P_t 
	%		&= \tilde{A} P_t + \tilde{A}^{\top} P_t + \modt{\Pi_{\bU_t}} (\Sigma + P_t S P_t) \modt{\Pi_{\bU_t}} 
        %        \\
	%		&= A P_t + P_t A^{\top} - 2 P_t S P_t + \modt{\Pi_{\bU_t}} \Sigma \modt{\Pi_{\bU_t}} + \modt{\Pi_{\bU_t}} P_t S P_t \modt{\Pi_{\bU_t}} 
	%		\\
	%		&= A P_t + P_t A^{\top} - P_t S P_t + \modt{\Pi_{\bU_t}} \Sigma \modt{\Pi_{\bU_t}},
        %\end{align*}
        %having used 
        %%$\tilde{C} \tilde{C}^{\top} = P_t H^{\top} \Gamma^{-1} H P_t$ and 
        %$\modt{\Pi_{\bU_t}} P_t \modt{\Pi_{\bU_t}} = P_t$. 
	%\begin{equation*} %\label{eqn:Sigma-range-U}
        %        P_U \Sigma P_U = \Sigma, \quad \forall t \geq 0.
        %\end{equation*} 
\end{proof}
\Cref{prop:mean-cov-vkb} characterises the mean and covariance of the process; even more can be said assuming the initial condition has a (low-rank) Gaussian distribution. 
\begin{proposition}
	Assume the initial condition of the \rf{\dlr-\kbp}\ $\Xcal_0 \sim \Ncal(m_0, P_0)$ where $P_0$ has rank $R$. 
	Then, the \dlr-\kbp\ \rf{solution $\Xcal_t$} is a (Gaussian) inhomogeneous Ornstein-Uhlenbeck process \modt{(conditional on $\Zcal_t$)}. 
	%hence a Gaussian process characterised by its mean and covariance. 
\end{proposition}
\begin{proof}
	\mods{Since $\Xcal_0 \sim \Ncal(m_0, P_0)$ with $P_0$ of rank $R$, the initial condition has a low-rank structure $U_0^0 + \bU_0 \bY_0^{\top}$}.	
	Then         
	\begin{multline*}
                \rd \left( U_t^0 + \bU_t \bY_t^{\top} \right) 
		= \tilde{A} \Xcal_t \rd t + {f} \md t + P_t H^{\top} \Gamma^{-1} \rd \Zcal_t + {\Pi_{\bU_t}} \Soh \rd \Wcal_t - P_t H^{\top} \Gmoh \rd \Vcal_t 
         \\
		= \left( \tilde{A} \Xcal_t + P_t S \Xcal_t^{\mathrm{signal}} + {f} \right) \md t 
	     	+ [P_t H^{\top} \Gmoh, {\Pi_{\bU_t}} \Soh, - P_t H^{\top} \Gmoh] [\rd \widetilde{\Vcal}_t, \rd \Wcal_t , \rd \Vcal_t]^{\top} 
	\\
		= \left(\hat{A}_t \Xcal_t + {\hat{f}}_t \right)\rd t + \hat{Q}_t \rd \widehat{\Wcal}_t. %\label{eqn:dXcal}
	\end{multline*}
%        \begin{align*}
%                \rd \left( U_t^0 + \bU_t \bY_t^{\top} \right) 
%	%	&= \rd U_t^0 + \rd \bU_t \bY^{\top} + \bU_t \rd \bY_t^{\top} + \rd [\bU_t, \bY_t]_t
%        % \\ 
%	%	&= A U_t^0 \rd t + b \md t + P_t H^{\top} \Gamma^{-1} (\rd \Zcal_t - H U^0_t \rd t)
%        % 		+  P^{\perp}_U \tilde{A} \bU_t \bY_t^{\top} \md t
%        % 		+ \modt{\Pi_{\bU_t}} \tilde{A} \Xcal_t^{\star} \md t 
%        % \\
%	%	& \quad + \modt{\Pi_{\bU_t}} \Soh \rd \Wcal_t - P_t H^{\top} \Gmoh \md \Vcal_t 
%        % \\
%		&= \tilde{A} \Xcal_t \rd t + \modt{f} \md t + P_t H^{\top} \Gamma^{-1} \rd \Zcal_t + \modt{\Pi_{\bU_t}} \Soh \rd \Wcal_t - P_t H^{\top} \Gmoh \rd \Vcal_t 
%         \\
%		&= \left( \tilde{A} \Xcal_t + P_t S \Xcal_t^{\mathrm{signal}} + \modt{f} \right) \md t 
%	     \\ &\quad+ [P_t H^{\top} \Gmoh, \modt{\Pi_{\bU_t}} \Soh, - P_t H^{\top} \Gmoh] [\rd \widetilde{\Vcal}_t, \rd \Wcal_t , \rd \Vcal_t]^{\top} 
%	\\
%		&= \left(\hat{A}_t \Xcal_t + \modt{\hat{f}}_t \right)\rd t + \hat{Q}_t \rd \widehat{\Wcal}_t. %\label{eqn:dXcal}
%        \end{align*}
	$P_t$ and $\bU_t$ are determined by deterministic equations independent of the process, therefore $\hat{Q}_t$ \mods{and} $\hat{A}_t$ are deterministic and $\mods{\hat{f}_t}$ is independent of the process. 
	The process is therefore an inhomogeneous Ornstein-Uhlenbeck process, \modr{hence} Gaussian (\mods{for completeness, a proof of that fact is given in}~\cref{lem:in-ou-gauss}). 
%The structure of the proof is exactly the same if $(U_t^0, \bU_t, \bY_t)$ solve~\cref{eqn:dkb-dlr-U0,eqn:dkb-dlr-U,eqn:dkb-dlr-Y}. 
\end{proof}

%\begin{remark}
%        Rearranging~\cref{eqn:dXcal}, since $\Xcal_t = U^0_t + \bU_t \bY^{\top}$, we obtain the relationship
%        \begin{equation} \label{eqn:consist-dX}
%                \rd \Xcal_t = A \Xcal_t \rd t + P_U \Soh \rd \Wcal_t + P_t H^{\top} \Gamma^{-1}(\rd \Zcal_t - (H \Xcal_t \rd t + \Goh \rd \Vcal_t)).
%        \end{equation}
%        This is not surprising, but does verify that the ``full order model'' respects the original dynamics. 
%\end{remark}

%The following two sections detail properties related to the conditional zero-mean part of the \dlr-\kbp~process, specifically the asymptotic properties of~\cref{eqn:kb-dlr-U} and a rank-reduction of~\cref{eqn:cov-lr}.

\subsubsection{Limit state of physical modes} \label{sssec:phys-modes}

\ra{The evolution of the subspace $\bU_t$ is described by \cref{eqn:kb-dlr-U} and its behaviour is tied to the spectrum of $A$. 
	Intuitively, $\bU_t$ is attracted to the most unstable directions of the dynamics, and that property can be leveraged to identify modes associated to Lyapunov exponents of $A$ \cite{ComputLyapReich2001}.  
	It is furthermore known (in the discrete-time setting) that the Ensemble Kalman covariance aligns with the unstable subspace \cite{DegenerateCarrassi2017}; by naturally restricting and steering the system towards the unstable modes, the \dlr-\kbp\ therefore displays a connection to the Assimilation in the Unstable Subspace literature, in which the Ensemble Kalman filter states are analysed only in the subspace generated by the most unstable vectors (see e.g \cite{LyapunovPalatella2013}). 
}

\ra{Concerning well-posedness \rf{of \eqref{eqn:kb-dlr-U}}, it} is immediate that $\bU_t$ is well-defined over $[0, +\infty)$ by local Lipschitzianity \rf{of the r.h.s of \eqref{eqn:kb-dlr-U}} and boundedness of $\bU_t$ for \rf{all $t \geq 0$.} 
\rf{Hereafter} we denote $\varphi_t(\bU_0)$ the solution flow to~\cref{eqn:kb-dlr-U} with initial condition $\bU_0$.
In general, $\bU_t$ might not converge \modt{to a limit state $\bU_{\infty}$ as $t \to \infty$} (as an example, consider any skew-symmetric matrix $A$ \mods{whose} dynamics \mods{reduces} to $\md \bU_t = A \bU_t \md t$ with solution $\bU_t = \exp(t A) \bU_0$ where $\exp(t A)$ is orthogonal for $t \geq 0$). 
\modt{If $A$ is symmetric,~\cref{eqn:kb-dlr-U} turns out to be an Oja flow, i.e. the gradient flow of a generalised Rayleigh quotient on the Stiefel manifold, \mods{and, provided that $\lambda_R(A) > \lambda_{R+1}(A)$, is} known to converge to \mods{the} linear subspace associated to the $R$ dominant eigenvalues of $A$~\cite{GlobalAnaWei1994}.
We present a \mods{slight} generalisation of that result, in that the matrix need only be symmetric to ensure convergence to \mods{a} dominant subspace \mods{for almost any initial condition.}}
\begin{proposition}
	Assume $A$ symmetric. 
	Then, the range of $\bU_t$ asymptotically align with a subspace spanned by $R$ distinct eigenvectors corresponding to the first $R$ dominant eigenvalues \mods{for almost any $\bU_0 \in \St(d, R)$ with respect to the Riemannian volume measure.}
\end{proposition}
\begin{proof}
Sort $A$'s eigenvalues in non-increasing order $\lambda_1 \geq \cdots \geq \lambda_d$.
Assuming $\bU_0~\in~\mathrm{St}(\rf{d},R)$, the dynamics of~\cref{eqn:kb-dlr-U} \mods{constitutes} a Riemannian gradient flow ascent on the Stiefel manifold, driven by the energy functional 
\begin{equation} \label{eqn:energy-stiefel}
	E(\rf{\bU}) = \frac{1}{2}\mathrm{tr}(\rf{\bU}^{\top} A \rf{\bU}).
\end{equation}
Note that~\cref{eqn:energy-stiefel} is a special case of the well-studied Brockett cost function,
%\todo{cite?}, 
with Riemannian gradient given by 
\begin{equation} \label{eqn:riem-grad-U}
	\mathrm{grad}E(\rf{\bU}) =  (I - \rf{\bU} \rf{\bU}^{\top}) A \rf{\bU}
\end{equation}
(see, e.g.,~\cite[Example 9.40]{IntroductionToBoumal2023}).
Convergence is then ensured by a standard argument in the gradient systems on manifolds literature, which we reproduce here for clarity, following the steps outlined in~\cite[Appendix A]{AMathematicalPerspGeshkovski2023}.

Firstly, the Stiefel manifold $\St(d,R)$ is a real-analytic, compact manifold, and the energy functional is real-analytic. 
By~\cite[Theorem 5.2]{OnTheRelaxHa2018}, those properties are sufficient to ensure that the dynamical system admits a limit $\lim_{t \rightarrow +\infty} \bU_t = \bU_{\infty} \in \St(d,R)$. 

Then,~\cite[Lemma A.1]{AMathematicalPerspGeshkovski2023} ensures that the set of initial conditions $\bU_0 \in \St(d,R)$ for which the gradient ascent converges to a strict saddle point of $\St(d,R)$ has \modt{(Riemannian volume)} measure zero, in other words the convergence to a global maximum for a given initialisation $\bU(0) = \bU_0 \in \St(d,R)$ holds true almost everywhere.

Finally, the landscape analysis provided in~\cite{BenignLandsBoumal2023} shows that the second-order critical points of~\cref{eqn:energy-stiefel} are global maxima, with value $\sum_{i=1}^R \lambda_i$.
As stated beforehand, this maximum is attained by any $R$ distinct eigenvectors corresponding to the $R$ dominant eigenvalues.
If \mods{an} eigenvalue \mods{among the $R$ dominant} is distinct, $\bU_t$ asymptotically recovers \mods{the corresponding} eigenvector, while if the multiplicity of the eigenvalue $m(\lambda) >1$, $\bU_t$ asymptotically recovers a rotation of the associated eigenvectors. 
	In the situation where $\lambda_i = \ldots = \lambda_j$ with $i \leq R < j$, $\bU_t$ recovers a rotation of the associated subspace, spanning a subspace of dimension $R + 1 - i$. 
\end{proof}

In the special case of distinct eigenvalues, it therefore holds that for almost any initialisation, $\bU_t$ asymptotically recovers the eigenvectors associated to the dominant $R$ eigenvalues of $A$.
Furthermore, owing to the local Lipschitzianity of the dynamical system, unless starting at one,
 the system can never reach a true steady state in finite time.

\subsubsection{Mean and covariance error bounds} \label{ssub:error-approx}

	%Other natural metrics of interest would be the mean square deviation conditional over all possible possible realisations of $(\Xcal_t^{\signal}, \Zcal_t)$ or conditional on $\Zcal_t$ (effectively removing the source of randomness in the mean equation, and hence the dynamics should be understood and analysed via rough path theory).
	%We do not pursue these issues further here.

As stated beforehand, the low-rank approximability of stochastic systems with \modt{linear-affine} drift is strongly tied to the diffusion term $\Soh$. 
The \dlr-\kbp~is expected to \ra{suitably approximate the \kbp\ provided that~\cref{eqn:mea-sigma} is sufficiently small; this is quantified in the proposition below, which establishes error bounds {on the first two moments of \dlr-\kbp. 
	As the \dlr-\kbp~mean is a stochastic process, the bound on the mean presented here is in a mean-square sense, that is the mean square deviation from the \kbp~mean averaged over all possible realisations of $\Zcal_t$ (but conditional on a realisation of the process $\Xcal_t^{\signal}$). 
}}
\begin{proposition} \label{prop:gr-bound}
	Let $(m_t^{\mathrm{DLR}}, P_t^{\mathrm{DLR}})$ resp. $(m_t^{\mods{\KBP}}, P_t^{\mods{\KBP}})$ \mods{be} the solutions \rf{to~\eqref{eqn:dlr-vkbp-mean}--\eqref{eqn:dlr-vkbp-cov}} resp. \rf{\eqref{eqn:kb-mean}--\eqref{eqn:kb-cov}}. 
	Then it holds
        \begin{align}
		\E \|m_T^{\DLR} - m_T^{\mods{\KBP}}\|^2 &\leq C_1(T) (\E \|m_0^{\DLR} - m_0^{\mods{\KBP}}\|^2 + \|P_0^{\mods{\KBP}} - P_0^{\DLR}\|^2_{\modr{F}}  + \ra{ \sup_{0 \leq s \leq T} \varepsilon_R^2(s)}),
        \\
        \|P_T^{\DLR} - P_T^{\mods{\KBP}}\|^2_F &\leq C_2(T) ( \|P_0^{\mathrm{DLR}} - P_0^{\mods{\KBP}}\|^2_F + \ra{ \sup_{0 \leq s \leq T} \varepsilon_R^2(s)}).
        \end{align}	
\end{proposition}
\begin{proof}
        Taking the difference between~\cref{eqn:kb-cov} and~\cref{eqn:dlr-vkbp-cov}, denoting
        the symmetric matrix $\D P_t = P_t^{\mods{\KBP}} - P_t^{\mathrm{DLR}}$, yields
        \begin{equation*}
		\partial_t \D P_t = A \D P_t + \D P_t A^{\modt{\top}} - P_t^{\mathrm{DLR}} S \D P_t - \D P_t S P_t^{\mods{\KBP}} + \Sigma - \modt{\Pi_{\bU_t}} \Sigma \modt{\Pi_{\bU_t}}.
        \end{equation*}
	Defining $e(t) = \|\D P_t(t)\|_F$, it follows that
        \begin{align*}
                \frac{1}{2}\frac{\rd}{\rd t} e^2(t) 
		%= \langle \D P_t, \dot{\D P}\rangle_F 
		&=  \langle \D P_t, A \D P_t + \D P_t A^{\modt{\top}} - P_t^{\mathrm{DLR}} S \D P_t - \D P_t S P_t^{\mods{\KBP}} + \Sigma - \modt{\Pi_{\bU_t}} \Sigma \modt{\Pi_{\bU_t}} \rangle_F 
                \\
		&= 2 \tr(\D P_t^2 A) 
                - \tr(P^{\mathrm{DLR}}_t S \D P_t^2)
                - \tr(\D P_t^2 P^{\mods{\KBP}}_t S)
                + \tr(\D P_t \D \Sigma_t) 
                \\
		&\leq \left(2 \norm{A}_F + \norm{S}_F (\norm{P^{\mathrm{DLR}}_t}_{\modr{F}} + \norm{P^{\mods{\KBP}}_t}_F)\right) e^2(t) + e(t) \|\D \Sigma_t\|_F 
                \\ 
		&\leq \left(2 \|A\|_F + \|S\|_F (\|P_t^{\mods{\KBP}}\|_F + \|P_t^{\mathrm{DLR}}\|_F ) + \frac{1}{2} \right) e^2(t) +  \frac{1}{2}\|\D \Sigma_t\|^2_F, 
        \end{align*}
	where $\Delta \Sigma_t = \Sigma - \modt{\Pi_{\bU_t}} \Sigma \modt{\Pi_{\bU_t}}$, implying $\norm{\D \Sigma_t}^2_F = \norm{\modt{\Pi_{\bU_t}^{\perp}} \Soh \Soh \modt{\Pi_{\bU_t}}}^2_F + \norm{\Soh \Soh \modt{\Pi_{\bU_t}^{\perp}}}^2_F  \leq 2 \mods{\lambda_{\max}(\Sigma)} \ra{\varepsilon^2_R(t)}$
	and whence we can apply the Gronwall lemma~\cite{NumericalApproQuarte1994} to obtain \mods{the upper bound $\|\D P_T\|^2_{\mods{F}}  \leq  C_2(T) ( \|\D P_0\|^2_F + \ra{\sup_{0 \leq s \leq T} \varepsilon^2_R(s)})$, with}
        \begin{equation*} 
		\mods{C_2(T) = \max\{1, 2 T \mods{\lambda_{\max}}(\Sigma)\} \exp \left( T (4 \|A\|_F + 1) + 2\|S\|_F \int_0^T (\|P_s^{\mods{\KBP}}\|_F + \|P_s^{\mathrm{DLR}}\|_F) \rd s \right).}
        \end{equation*}
	Owing to the boundedness of $P^{\mods{\KBP}}_t$ and $P^{\DLR}_t$, the right hand side is bounded. 	
        Next, subtracting~\cref{eqn:dlr-vkbp-mean} from~\cref{eqn:kb-mean} yields the process
        \begin{multline*}
                \rd M_t = 
		\rd (m_t^{\mods{\KBP}} - m_t^{\DLR}) 
		= 
	\left( (A
		- P_t^{\mods{\KBP}} S ) M_t 
		+ \D P_t S \Xcal_t^{\mathrm{signal}}
		- \D P_t S m_t^{\DLR} 
	\right) \md t
		\\
		- \D P_t H^{\top} \Gmoh \md \tilde{\Vcal}_t.
        \end{multline*}
	Defining $e^2(t) = \|M_t\|^2 = \langle M_t, M_t \rangle$, using \Ito's formula and taking expectations 
        \begin{multline*}
		\E e^2(T) - \E e^2(0) 
			= \E\int_{0}^T {2} \langle M_t, \rd M_t \rangle + \sum_{i=1}^d \md [M_t(i),M_t(i)]_t 
		\\
			= \E \int_{0}^T 2 \langle M_t, (A - P^{{\KBP}}_t S) M_t \rangle \rd t 
			+ \E \int_{0}^T 2 \langle M_t, \D P_t S \Xcal_t^{\mathrm{signal}} \rangle  \rd t 
			- \E  \int_0^T 2 \langle M_t, \D P_t S m_t^{\DLR} \rangle \rd t 
		\\
			+ \int_0^T  \| \D P_t H^{\top} \Gmoh \|^2_F \rd t
                \\
			\leq 
			\int_0^T 2 (\|A \|_F  + \| P_t^{{\KBP}} S\|_F) \E e^2(t) \rd t
			+ \E \int_0^T e(t) \| \D P_t \|_F (\| S \Xcal^{\mathrm{signal}}_t \| + \|S\| \| m_t^{\DLR} \|) \rd t
			\\
			+ \int_0^T \| \D P_t \|_F^2 {\lambda_{\max}(S)} \rd t
		\\
			\leq \! \! \!
			\int_0^T \! \left[ 2 (\|A \|_F \! + \| P_t^{{\KBP}} S\|_F + \frac{1}{2}) \E e^2(t) \!
			+
			\| \D P_t \|_F^2 
			(\| S \|_F + \|S \Xcal_t^{\mathrm{signal}}\|^2 + \|S\|^2 \E \|m_t^{\DLR}\|^2)\!
	\right] \md t,
        \end{multline*}
%        \begin{align*}
%		\E e^2(T) - \E e^2(0) 
%			&= \E\int_{0}^T {2} \langle M_t, \rd M_t \rangle + \sum_{i=1}^d \md [M_t(i),M_t(i)]_t \\
%			&= \E \int_{0}^T 2 \langle M_t, (A - P^{{\KBP}}_t S) M_t \rangle \rd t 
%			\\
%			&\quad + \E \int_{0}^T 2 \langle M_t, \D P_t S \Xcal_t^{\mathrm{signal}} \rangle  \rd t 
%			 \\
%			& \quad - \E  \int_0^T 2 \langle M_t, \D P_t S m_t^{\DLR} \rangle \rd t 
%			\\
%			&\quad 
%			+ \int_0^T  \| \D P_t H^{\top} \Gmoh \|^2_F \rd t
%                \\
%			& \leq 
%			\int_0^T 2 (\|A \|_F  + \| P_t^{{\KBP}} S\|_F) \E e^2(t) \rd t
%                \\
%			&\quad + \E \int_0^T e(t) \| \D P_t \|_F (\| S \Xcal^{\mathrm{signal}}_t \| + \|S\| \| m_t^{\DLR} \|) \rd t
%                \\
%			&\quad + \int_0^T \| \D P_t \|_F^2 {\lambda_{\max}(S)} \rd t
%		\\
%			& \leq
%			\int_0^T 2 (\|A \|_F  + \| P_t^{{\KBP}} S\|_F + \frac{1}{2}) \E e^2(t) \rd t 
%			\\
%			& \quad + \int_0^T  \left( 
%			\| \D P_t \|_F^2 
%			(\| S \|_F + \|S \Xcal_t^{\mathrm{signal}}\|^2 + \|S\|^2 \E \|m_t^{\DLR}\|^2)
%			\right) \md t,
%        \end{align*}
	having used~\cref{eqn:tr-lbda-ineq} to bound 
	$\norm{\Delta P_t H^{\top} \Gmoh}_F^2 
	= \tr (\Delta P_t S \Delta P_t) 
	%\leq \lambda_{\max}(S) \tr(\Delta P_t^2) 
	\leq \mods{\lambda_{\max}(S)} \norm{\Delta P_t}^2_F$ in the first inequality, and Cauchy-Schwarz in the second.
        By the boundedness of all involved terms, we conclude again using a Gronwall inequality, 
	\begin{multline*}
		\E \|m_T^{\mods{\KBP}} - m_T^{\DLR}\|^2 \leq (\E \|m_0^{\mods{\KBP}} - m_0^{\DLR}\|^2 + \|\D P_0\|^2_{\mods{F}} +  \ra{\sup_{0 \leq s \leq T} \varepsilon^2_R(s)} ) \\
		\max\left\{1, \int_0^T \modt{C_2(t) } (\mods{\lambda_{\max}(S)} + \|S \Xcal_t^{\mathrm{signal}}\|^2 + \|S\|^2 \E \|m_t^{\DLR}\|^2) \md t \right\} \\
		\exp\left(\int_0^t (2 \|A\|_F + 2 \|P_s^{\mods{\KBP}}\| + 1) \md s \right).
	\end{multline*}
\end{proof}

 \subsection{Reduced Kalman-Bucy equations} \label{ssec:red-ricc}

%When used in applications, the \kbp~is typically tracked by computing the evolution of the mean and covariance, as those quantities are sufficient to fully characterise the Gaussian process.
%The covariance in particular might be challenging to update, as it requires evolving a (potentially high-dimensional) $d\times d$ matrix.
 \modt{In this section, we explore a fully deterministic alternative formulation \modr{of} the \dlr-\kbp~system by directly solving its mean and reduced covariance}. 

 In our setting, the covariance is low-rank by ansatz; this fact can be leveraged to further reduce~\cref{eqn:dlr-vkbp-cov} \mods{using the characterisation of $P_t$ in~\eqref{eqn:cov-lr}, and derive an evolution equation for} the Gram matrix $M_{\bY_t}$ -- again a Riccati equation, but with time-dependent factors. 
 \begin{lemma} \label{lem:gram-ex-bounds}
        The Gram matrix $M_{\bY_t}$ satisfies
        \begin{equation} \label{eqn:ricc-red}
                \dot{M}_{\bY_t} = A_{\bU_t} M_{\bY_t} + M_{\bY_t} A_{\bU_t}^{\top} 
                - M_{\bY_t} S_{\bU_t} M_{\bY_t} + \Sigma_{\bU_t},
        \end{equation}
        where $A_{\bU_t} = \bU_t^{\top} A \bU_t$, $S_{\bU_t} = \bU_t^{\top} S \bU_t $, $\Sigma_{\bU_t} = \bU_t^{\top} \Sigma \bU_{t}$. 
	For a given initial condition $M_{\bY_0}$, the equation admits a unique global solution.
	Furthermore, if $\lambda_{\max}(A + A^{\top}) < 0$, the solution is uniformly bounded in time. 
\end{lemma}
\begin{proof}
	The global well-posedness follows from~\cite[Lemma 2]{GlobalTheBucy1967}.
	Consider $\Phi(\mods{t, t_0})$ the fundamental matrix of $A_{\bU_t}$, i.e. the solution to $\dot{\Phi}(\mods{t, t_0}) = A_{\bU_t} \Phi(\mods{t, t_0})$ with $\Phi(\mods{t_0, t_0}) = \bm{I}_{R \times R}$.	
	\modt{One has}
	\begin{multline*}
		\frac{\md}{\md t} \norm{\Phi(\mods{t, t_0})}^2_F = \frac{\md}{\md t} \tr(\Phi(\mods{t, t_0})^{\top} \Phi(\mods{t,t_0}) ) 
		%=
		%\tr(\Phi(t)^{\top} (A_{\bU_t} + A_{\bU_t}^{\top})\Phi(t) ) 
		\leq \lambda_{\max}(A_{\bU_t} + A_{\bU_t}^{\top}) \norm{\Phi(\mods{t,t_0})}^2_F 
		%\\
		\\
		\leq \lambda_{\max}(A + A^{\top}) \norm{\Phi(\mods{t,t_0})}^2_F,
	\end{multline*}
	\modt{and,} by Gronwall's inequality~\cite{NumericalApproQuarte1994} it holds $\norm{\Phi(t, t_0)}^2_F \leq \norm{\bm{I}_{R \times R}}^2_F \exp(\mods{(t - t_0)} \lambda_{\max}(A + A^{\top}))$.
	\mods{Therefore}, using~\cite[Lemma 1]{GlobalTheBucy1967}, the solution to~\cref{eqn:ricc-red} \mods{with $t_0 = 0$}, given by the flow $\varphi_t(M_{\bY_{\mods{0}}})$, verifies
	\begin{equation*}
		0 \leq \varphi_t(M_{\bY_{\mods{0}}}) \leq \Phi(\mods{t,0}) M_{\bY_{\mods{0}}} \Phi(\mods{t,0})^{\top} + \int_{{\mods{0}}}^t \Phi(t, s) \Sigma_{\bU_s} \Phi(t,s) \md s,
	\end{equation*}
	where the inequalities are understood in the Loewner partial order sense.
	\modt{If $\lambda_{\max}(A + A^{\top}) < 0$, this implies}
	\begin{multline*}
		\norm{\varphi_t(M_{\bY_{\mods{0}}})}_F \leq \norm{M_{\bY_{\mods{0}}}}_F \norm{\Phi(t,{\mods{0}})}^2_F + \int_{{\mods{0}}}^t \norm{\Sigma_{\bU_t}}_F \norm{\Phi(t,s)}^2_F \md s 
		\\
		\leq \norm{M_{\bY_{\mods{0}}}}_F \modt{R} \exp( \mods{t} \lambda_{\max}(A + A^{\top})) + \modt{R}\frac{1 - \exp( \mods{t} \lambda_{\max}(A + A^{\top}))}{|\lambda_{\max}(A + A^{\top})|} \modt{\norm{\Sigma}_F}
		\\
		\leq C(M_{\bY_{\mods{0}}},\Sigma, \lambda_{\max}(A + A^{\top}), \modt{R}),
	\end{multline*}
	which proves the second claim.
\end{proof}
%\begin{proof}
%\Cref{eqn:ricc-red} is immediately obtained by combining~\cref{eqn:cov-lr,eqn:dlr-vkbp-cov}.
%The right-hand-side is well-defined since $\bU_t$ is. 
%By~\cite[Theorem 2.1]{OnAMatrixWonham1968}, the solution of $M_{\bY_t}$ exists and is absolutely continuous on any interval $[0,T]$ almost everywhere, provided that all time-dependent matrices in the right-hand-side are Lebesgue measurable and bounded in norm.
%Therefore, $\dot{M}_{\bY_t} = \mathrm{rhs}(t, M_{\bY_t})$ almost everywhere (in fact, the derivative is also continuous since the right-hand side is continuous and they agree almost everywhere). 
%The continuity of $M_{\bY_t}$ over $[0,T]$ naturally implies the boundedness of the DLR covariance $P_t^{\mathrm{DLR}}$.
%\end{proof}

Combining~\cref{eqn:kb-dlr-U0,eqn:kb-dlr-U,eqn:ricc-red} therefore yields a set \mods{of \emph{deterministic}} equations, the solution of which fully \mods{characterises} the mean and covariance of \dlr-\kbp~(and consequently the entire process).
For clarity, we rewrite the \textbf{reduced Kalman-Bucy equations} here.
\begin{equation} 
\begin{aligned}
        \md U_t^0 &= A U_t^0 \rd t + P_t H^{\top} \Gamma^{-1} (\rd \Zcal_t - H U^0_t \rd t),  \\
	\md \bU_t &= \modt{\Pi^{\perp}_{\bU_t}} A \bU_t \md t, \\
	\dot{M}_{\bY_t} &= A_{\bU_t} M_{\bY_t} + M_{\bY_t} A_{\bU_t}^{\top} 
	- M_{\bY_t} S_{\bU_t} M_{\bY_t} + \Sigma_{\bU_t}.
\end{aligned}
\end{equation}

\modt{
This set of equations bears a connection to two distinct works.
Firstly, these equations coincide with those of the ``Low-Rank Kalman-Bucy Filter'' introduced in~\cite{OnANewLowRYamada2021} and further developed in~\cite{LowRankApproximatedTsuzukiB2024,ComparisonOfEstYamada2021,LowRankApproxiTsuzuki2024}, which were obtained via an \textit{ad-hoc} modification of the Riccati equation for LTI systems. 
Secondly, the Rank-Reduced Kalman Filter proposed in~\cite{TheRankRedSchmidt2023} uses matricial DLR to evolve the covariance of the LTI system -- in the context of continuous-time state/discrete observation setting, this is a Lyapunov equation $\partial_t P_t = A P_t + P_t A^{\top} + \Sigma$, with initial condition $P(0) = \bU_0 D_0 \bU_0^{\top}$ . 
Specifically, they use the popular Unconventional/BUG integrator~\cite{AnUnconventionCeruti2022} to do so; a substep of that method consists in evolving $D$ by projecting the equation onto the range of $\bU_t$, yielding the ``reduced Lyapunov'' equation
\begin{equation*}
	\dot{D}_t = A_{\bU_t} D_t + D_t  A_{\bU_t}^{\top} +  \Sigma_{\bU_t}.
\end{equation*}
In either case, the low-rank ansatz is applied at the level of the covariance of the LTI system; in contrast, our approach is based on the low-rank stochastic process ansatz, from which the reduced Kalman-Bucy are naturally derived. 
The Gaussian characterisation of the \dlr-\kbp~furthermore validates that the mean and reduced covariance are the exact evolution equations of the moments of the low-rank approximation process, rather than solely being low-rank approximations of the \kbp~moments.
Furthermore, our framework is naturally amenable to extensions such as an Ensemble-Kalman formulation, which is the object of~\cref{sec:dlr-enkf}.
}

\modr{
	We conclude this section by stating the \dlr-\kbp\ well-posedness result, having introduced all the notions and notations on which the proof relies.
	For the sake of conciseness, the proof is presented in Appendix~\ref{app:b}. 
\begin{lemma} \label{lem:dlr-kbp-wp}
	The \dlr-\kbp~\cref{eqn:kb-dlr-U0,eqn:kb-dlr-U,eqn:kb-dlr-Y}~admit a unique strong solution.
\end{lemma}
}

\section{DLR-EnKF formulation} \label{sec:dlr-enkf}

The Ensemble Kalman Filter (\enkf) is an extension of the \kbp, and consists in directly simulating~\cref{eqn:vkb} via a particle approximation, replacing the covariance by the sample covariance of the particles. 
This avoids the computation of the covariance by~\cref{eqn:kb-cov} (potentially intractable for high-dimensional systems), is easily extended to non-linear dynamics (and even non-linear observations), and \modt{often manages to track the signal} even with a modest number of particles~\cite{SequentialDataEvense1994,DataAssimilatiHoutek1998}.
In particular, it was \modt{proved} \mods{in~\cite{LongTimeStabiDeWil2018}} that, \modt{under certain assumptions,} an alternative formulation of the \enkf~can accurately track the true dynamics in the case of non-linear dynamics, full-state observations and small observation noise.
On the other hand, the method in its basic form \modt{(straightforward particle approximation of~\cref{eqn:vkb})} is known to display non-robust behaviour (instability induced by accumulating numerical errors) and the filter may be subject to catastrophic filter divergence~\cite{ConcreteEnsembKelly2015,CatastrHarlim2010}. 
In practice, the~\enkf~is often modified (via variance inflation, localisation~\cite{HighDimensionaSunH2024}, mollification~\cite{AMollifiedEnsBergem2010}, square root filters~\cite{ADetermSakov2008} etc) and/or combined with other methods~\cite{AHybridParticGrooms2021} to improve its stability.
\mods{Nevertheless, this work focuses on the basic version of the \enkf, as our primary objective is to establish a suitable \dlr-\enkf~extension. 
	The integration of the aforementioned stabilisation techniques into the \dlr-\enkf~framework is left for future research and is not further addressed here.
}

In this section, we \mods{aim} to develop a~\dlr-\enkf~formulation in the case of linear-affine dynamics. 
\mods{The motivation for a DLR formulation is twofold.}
Firstly, as noted in~\cite{OnNumericalPrLiJi2008}, when evolving the~\enkf~with few particles, the Monte-Carlo error might still constitute a non-negligible bias in the dynamics. 
Reducing \mods{it}~\textit{a priori} requires \mods{increasing the number of particles}, which may not be possible owing to computational budget constraints.
The \dlr-\enkf~naturally leverages the low-rank structure, allowing to evolve more \mods{particles} at a reduced cost -- thereby aiming to reduce the potential spurious correlations that arise from the sample covariance.
\mods{
	Secondly, the possibility to naturally extend the \dlr-\enkf~to general non-linear dynamics is of interest, as it allows to evolve a low-rank empirical measure in settings where computing the mean and covariance via~\cref{eqn:kb-mean,eqn:kb-cov} is not valid.
	While distinct from the true filtering measure, the empirical measure might still display suitable signal-tracking abilities at reduced computational cost, making it an interesting tool for applications such as filtering of large-scale non-linear systems, parameter identification, etc. 
	It is moreover important to understand the (asymptotic) properties of \dlr-\enkf, and as a first step in that direction, we show a propagation of chaos result in the case of linear-affine dynamics \mods{and full observations}. 
}
%Finally, as~\cref{eqn:vkb} is a McKean-Vlasov equation, showing convergence of the discretisation is not straighforward; the \dlr-\enkf, which is an interacting particle system, can be understood as a ``semi-discretisation'' and aid in analysing the fully discretised system.  

In what follows, $\widehat{\Xcal}_t = \left[ \widehat{\Xcal}_t^{(1)}, \ldots, \widehat{\Xcal}_t^{(\rb{N})} \right] $ denotes an ensemble of $\rb{N}$ particles. 
Furthermore, we let
\begin{equation*}
	\E_{\rb{N}}[\widehat{\Xcal}_t] = \frac{1}{{\rb{N}}} \sum_{i=1}^{\rb{N}} \widehat{\Xcal}_t^{(i)}, 
	\quad
	\widehat{\Xcal}_t^{\star} = \widehat{\Xcal}_t - \E_{\rb{N}}[\widehat{\Xcal}_t],
	\quad
	\widehat{P}_t = \frac{{\rb{N}}}{{\rb{N}}-1} \E_{\rb{N}} [\widehat{\Xcal}^{\star}_t (\widehat{\Xcal}^{\star}_t)^{\top}]
\end{equation*}
\mods{denote} the \mods{sample} mean, \mods{sample} zero-mean and \mods{sample} covariance\mods{,} respectively.
Consider also $\mods{\hWcal_t  = \modr{(\hWcal_t^{(1)}, \ldots, \hWcal_t^{({\rb{N}})})}}$ (resp. $\hVcal_t$) \mods{the collection of} ${\rb{N}}$ independent copies of $\Wcal_t$ (resp. $\Vcal_t$).
The \enkf~consists in evolving
\begin{equation} 
	\label{eqn:enkf}
	\rd \widehat{\Xcal}_t^{(p)} = (A \Xcal_t^{(p)} + \modt{f}) \rd t 
	+ \Soh \rd \hWcal_t^{(p)} 
	+ \widehat{P}_t H^{\top} \Gamma^{-1} ( \rd \Zcal_t - H \widehat{\Xcal}_t^{(p)} \rd t - \Goh \rd \widehat{\Vcal}_t^{(p)}),
\end{equation} 
for $p=1, \ldots,{\rb{N}}$. 
\mods{Notice that the evolution of the particles is coupled due to the presence of the sample covariance $\hP_t$.}

\subsection{DLR-EnKF equations}

\mods{To derive} the \dlr-\enkf~system \mods{as a particle approximation of the \dlr-\kbp~\eqref{eqn:kb-dlr-U0}--\eqref{eqn:kb-dlr-Y} and~\eqref{eqn:dlr-vkb}}, we replace $\bY_t = \left[Y^1_t, \ldots, Y^R_t\right]$ by $\hbY_t =  \left[\widehat{Y}^1_t, \ldots, \widehat{Y}^R_t \right]$ where $\widehat{Y}^i_t = [\widehat{Y}^{i,(1)}_t, \ldots,\widehat{Y}^{i,({\rb{N}})}_t]^{\top}$.
Each particle $\hbY_t^{(p)} = \left[\widehat{Y}^{1,(p)}_t, \ldots, \widehat{Y}^{R,(p)}_t \right]$ thus has dimension $R$. 
We also enforce $\E_{\rb{N}}[\hbY_t^i] = 0$; then, \mods{the solution of the \dlr-\enkf~system will have the form} $\hXcal_t = \modt{\widehat{U}}_t^0 + \bU_t \hbY^{\top}_t$ where $\modt{\widehat{U}}_t^0$ is the sample mean, \modt{and} \mods{the sample covariance reads} $\widehat{P}_t = \frac{1}{{\rb{N}}-1} \bU_t \hbY_t^{\top} \hbY_t \bU_t^{\top}$.
\modt{ A relevant quantity in this section will be the ``sample reduced covariance'' $\hP_{\hbY_t} = \frac{1}{{\rb{N}}-1}\hbY_t^{\top} \hbY_t$.}

To set the initial condition \mods{$X_0 = U^0_0 + \bU_0 \bY_0^{\top}$ with $Y_0 \sim \Ncal(0, M_{\bY_0})$, we draw ${\rb{N}}$ i.i.d particles $\widehat{\bZ}_0^{(p)} \sim \Ncal(0, M_{\bY_0})$, $p = 1, \ldots {\rb{N}}$ and set $\widehat{U}^0_0 = U^0_0 + \bU_0 \E_{\rb{N}}[\widehat{\bZ}_0]$ and $\hbY_0 = \widehat{\bZ}_0 - \E_{\rb{N}}[\widehat{\bZ}_0]$ so that $\hXcal_0 = \widehat{U}_0 + \bU_0 \hbY_0$ has zero sample mean modes.}

We propose the following \dlr-\enkf~system:
\begin{align}
	\rd \modt{\widehat{U}}_t^0 &=  (\widehat{\tilde{A}} \modt{\widehat{U}}_t^0 + \modt{f}) \rd t 
			+ \widehat{P}_t H^{\top} \Gamma^{-1} \rd \Zcal_t 
			+ \modt{\Pi_{\bU_t}} \Soh \rd \E_{\rb{N}} [\hWcal_t] 
			\modt{-} \widehat{P}_t H^{\top} \Gmoh \rd \E_{\rb{N}} [\widehat{\Vcal}_t], \label{eqn:mean-mode-particles} \\
	\rd \bU_t &= \modt{\Pi^{\perp}_{\bU_t}} A \bU_t \rd t, \label{eqn:phys-modes-particles}\\
	\rd \hbY_t^{\top} &= \bU_t^{\top} \widehat{\tilde{A}} \bU_t \hbY_t^{\top} \rd t + 
	\bU_t^{\top} \Soh \rd \hWcal_t^{\star} - \bU_t^{\top} \widehat{P}_t H^{\top} \Gmoh \rd \widehat{\Vcal}^{\star}_t, \label{eqn:stoch-modes-particles}
\end{align}
where $\widehat{\tilde{A}} = A - \widehat{P}_t S$.
The terms in~\cref{eqn:stoch-modes-particles} are understood in a $p$-component-wise sense, i.e., for fixed $p = 1, \ldots, {\rb{N}}$, 
\begin{equation*}
	\rd \hbY_t^{(p)} = 
	\bU_t^{\top} \widehat{\tilde{A}} \bU_t \hbY_t^{(p)} \rd t + 
		\bU_t^{\top} \Soh \rd \hWcal_t^{\star,(p)}
		+ \bU_t^{\top} \widehat{P}_t H^{\top} \Gmoh \rd \widehat{\Vcal}_t^{\star,(p)}. \label{eqn:stoch-modes-particles-detail}
\end{equation*}
To ensure that the stochastic modes $\hbY_t$ have zero sample mean at all times, in~\cref{eqn:stoch-modes-particles} we have subtracted the sample mean of $\hWcal_t$ and $\hVcal_t$, which then has to be included in the dynamics of $\widehat{U}^0_t$ in~\cref{eqn:mean-mode-particles}. 
The following equation for the particles $\hXcal_t^{(p)}$, $p = 1, \ldots {\rb{N}}$ can be easily recovered from~\cref{eqn:mean-mode-particles,eqn:phys-modes-particles,eqn:stoch-modes-particles}.
\begin{equation} 
	\rd \widehat{\Xcal}_t^{(p)} = (A \widehat{\Xcal}_t^{(p)} + f) \rd t 
	+ \modt{\Pi_{\bU_t}} \Soh \rd \hWcal_t^{(p)} 
	+ \widehat{P}_t H^{\top} \Gamma^{-1} ( \rd \Zcal_t - H \widehat{\Xcal}_t^{(p)} \rd t - \Goh \rd \widehat{\Vcal}_t^{(p)}). \label{eq:dXhat}
\end{equation}
It represents a particle version of~\cref{eqn:dlr-vkb}.
\modq{Our results will rely on the following characterisation of the sample mean and sample (reduced) covariance.
	\begin{lemma} \label{lem:red-mean-cov}
	The sample mean and sample reduced covariance satisfy the mean and reduced Riccati equations up to a stochastic perturbation,  
	\begin{align} 
		\rd \modt{\widehat{U}}_t^0  &= (A \modt{\widehat{U}}_t^0  + \modt{f})\rd t + \widehat{P}_t H^{\top} \Gamma^{-1} (\rd \Zcal_t - H \modt{\widehat{U}}_t^0 \rd t)  +  \frac{1}{\sqrt{{\rb{N}}}} \rd \Mcal_t, \label{eqn:d-sample-mean} \\
		\md \hP_{\hbY_t} &= \left( A_{\bU_t} \hP_{\hbY_t} + \hP_{\hbY_t} A_{\bU_t}^{\top} +  \Sigma_{\bU_t} - \hP_{\hbY_t} S_{\bU_t} \hP_{\hbY_t} \right)\md t + \frac{1}{\sqrt{{\rb{N}}-1}} \md \Ncal_t, \label{eqn:d-sample-Gram}
	\end{align}
	where $\Mcal_t$ is the martingale
	\begin{equation*} %\label{eqn:sample-mean-martingale}
		\rd \Mcal_t =  \frac{1}{\sqrt{{\rb{N}}}} \left( \modt{\Pi_{\bU_t}} \Soh \left( \sum_{p=1}^{\rb{N}} \md \hWcal^{(p)}_t \right)
			+ \widehat{P}_t H^{\top} \Gmoh \left( \sum_{i=1}^{\rb{N}} \md  \widehat{\Vcal}^{(p)}_t \right) \right).
	\end{equation*}
	and where the martingale $\Ncal_t$ is specified in the proof. 
\end{lemma}
\begin{proof}
	The statement for the sample mean is immediate from~\cref{eqn:mean-mode-particles}. 

	Noting that, for suitable indices $m,m'$,
	\begin{align*}
		\modt{\frac{\md}{\md t}} \left[ \hWcal_t^{\star,(p)}(m), \hWcal^{\star,(p')}_t(m') \right]_t 
		= \modt{ \left(\delta_{p,p'}  - \frac{1}{{\rb{N}}} \right) \delta_{mm'} }
		= \modt{\frac{\md}{\md t}} \left[ \hVcal_t^{\star,(p)}(m), \hVcal^{\star,(p')}_t(m') \right]_t 
	\end{align*}
	we develop 
	\begin{multline*} %\label{eqn:enkf-cov-sub2}
		\rd (\hbY_t^{\top} \hbY_t)_{i,j} 
		= \sum_{p=1}^{\rb{N}} \rd \left( \hY_t^{i,(p)} \hY_t^{j,(p)}\right) 
		= \sum_{p=1}^{\rb{N}} \rd\hY_t^{i,(p)} \hY_t^{j,(p)} + \hY_t^{i,(p)} \rd \hY_t^{j,(p)}  + \rd \left[ \hY_t^{i,(p)}, \hY_t^{j,(p)} \right]_t
		\\
		= \sum_{p=1}^{{\rb{N}}}  \bigg(
		(U_t^i)^{\top} \widehat{\tilde{A}} \bU_t \hbY_t^{(p)} \hY_t^{j,(p)} \rd t + 
		(U_t^i)^{\top} \Soh \hY_t^{j,(p)} \rd \hWcal_t^{\star,(p)} 
		+ (U_t^{i})^{\top} \hY_t^{j,(p)} \widehat{P}_t H^{\top} \Gmoh \rd \hVcal_t^{\star,(p)} 
		\\
		+ ( U_t^{j})^{\top} \widehat{\tilde{A}} \bU_t \hbY_t^{(p)} \hY_t^{i,(p)} \rd t 
		+ (U_t^{j})^{\top} \Soh \hY_t^{i,(p)} \rd \hWcal_t^{\star,(p)} 
		+ ( U_t^{j})^{\top} \hY_t^{i,(p)} \widehat{P}_t H^{\top} \Gmoh \rd \hVcal_t^{\star,(p)} \bigg)
		\\
		+ ({\rb{N}}-1)(U_t^{i})^{\top}  \Sigma U_t^{j} \rd t
		+ ({\rb{N}}-1)(U_t^{i})^{\top} \widehat{P}_t S \widehat{P}_t U_t^{j} \rd t. 
	\end{multline*}
%	\begin{align*} %\label{eqn:enkf-cov-sub2}
%		\rd (\hbY_t^{\top} \hbY_t)_{i,j} 
%		&= \sum_{p=1}^{\rb{N}} \rd \left( \hY_t^{i,(p)} \hY_t^{j,(p)}\right)  \\
%		&= \sum_{p=1}^{\rb{N}} \rd\hY_t^{i,(p)} \hY_t^{j,(p)} + \hY_t^{i,(p)} \rd \hY_t^{j,(p)}  + \rd \left[ \hY_t^{i,(p)}, \hY_t^{j,(p)} \right]_t
%		\\
%		&= \sum_{p=1}^{{\rb{N}}}  \bigg(
%		(U_t^i)^{\top} \widehat{\tilde{A}} \bU_t \hbY_t^{(p)} \hY_t^{j,(p)} \rd t + 
%		(U_t^i)^{\top} \Soh \hY_t^{j,(p)} \rd \hWcal_t^{\star,(p)} 
%		\\ 
%		& \quad + (U_t^{i})^{\top} \hY_t^{j,(p)} \widehat{P}_t H^{\top} \Gmoh \rd \hVcal_t^{\star,(p)} 
%		\\
%		&\quad + ( U_t^{j})^{\top} \widehat{\tilde{A}} \bU_t \hbY_t^{(p)} \hY_t^{i,(p)} \rd t 
%		+ (U_t^{j})^{\top} \Soh \hY_t^{i,(p)} \rd \hWcal_t^{\star,(p)} 
%		\\
%		&\quad
%		+ ( U_t^{j})^{\top} \hY_t^{i,(p)} \widehat{P}_t H^{\top} \Gmoh \rd \hVcal_t^{\star,(p)} \bigg)
%		\\
%		&\quad + ({\rb{N}}-1)(U_t^{i})^{\top}  \Sigma U_t^{j} \rd t
%		\\
%		&\quad + ({\rb{N}}-1)(U_t^{i})^{\top} \widehat{P}_t S \widehat{P}_t U_t^{j} \rd t. 
%	\end{align*}
	Hence, dividing by $({\rb{N}}-1)$ 
	\begin{align*}
		\md \hP_{\hbY_t} &= \left( \widehat{\tilde{A}}_{\bU_t} \hP_{\hbY_t} + \hP_{\hbY_t} \widehat{\tilde{A}}_{\bU_t}^{\top} + \Sigma_{\bU_t} + \hP_{\hbY_t} S_{\bU_t} \hP_{\hbY_t} \right)\md t + \frac{1}{\sqrt{{\rb{N}}-1}} \md \Ncal_t \\
				 &= \left( A_{\bU_t} \hP_{\hbY_t} + \hP_{\hbY_t} A_{\bU_t}^{\top} +  \Sigma_{\bU_t} - \hP_{\hbY_t} S_{\bU_t} \hP_{\hbY_t} \right)\md t + \frac{1}{\sqrt{{\rb{N}}-1}} \md \Ncal_t,
	\end{align*}
	where the $\Mat{R \times R}$-valued, symmetric martingale verifies 
	\begin{multline} \label{eqn:sample-gram-martingale}
		\md (\Ncal_t)_{ij} = \frac{1}{\sqrt{{\rb{N}}-1}}\sum_{p=1}^{\rb{N}} (U_t^i)^{\top} \Soh \hY_t^{j,(p)} \rd \hWcal_t^{\star,(p)} 
					+ (U_t^{j})^{\top} \Soh \hY_t^{i,(p)} \rd \hWcal_t^{\star,(p)} 
					\\ 
					+ (U_t^{i})^{\top} \hY_t^{j,(p)} \widehat{P}_t H^{\top} \Gmoh \rd \hVcal_t^{\star,(p)} 
					+ ( U_t^{j})^{\top} \hY_t^{i,(p)} \widehat{P}_t H^{\top} \Gmoh \rd \hVcal_t^{\star,(p)}.
	\end{multline}
\end{proof}

In our analysis, it will be necessary to bound the fluctuations of the sample mean and sample covariance characterised in~\cref{lem:red-mean-cov}. 
We therefore conclude this section by stating the following lemma, taken from~\cite{OnTheStabilitDelMo2018}, which will be useful in what follows.
\begin{lemma}[\cite{OnTheStabilitDelMo2018}, Lemma 7.2] \label{lem:fost-lyap}
	Let $\modt{Q_t}$ be a \modt{$\Acal_t$}-adapted stochastic process on a measurable state space $(E, \Ecal)$, and let $\modt{G}$ \modt{be} a non-negative measurable function on $(E, \Ecal)$ such that
	\begin{equation*}
		\md \modt{G}(\modt{Q_t}) = \Lcal_t \modt{G}(\modt{Q_t}) \md t + \md \Mcal_t
	\end{equation*}
	where $\Mcal_t$ \modt{is} an \mods{$\Acal_t$-martingale} and $\Lcal_t \modt{G}(\modt{Q_t})$ is $\modt{\Acal_t}$-adapted.
	\begin{itemize}
		\item[(a)] Assume that 
			\begin{align*}
				\Lcal_t \modt{G} (\modt{Q_t}) &\leq 2 \gamma \sqrt{\modt{G}(\modt{Q_t})} + 3 \alpha \modt{G}(\modt{Q_t}) - \beta \modt{G} (\modt{Q_t})^2 + r,
				\\
				\md [\Mcal_t, \Mcal_t]_t &\leq \modt{G}(\modt{Q_t}) (\tau_0 + \tau_1 \modt{G}(\modt{Q_t}) + \tau_2 \modt{G}(\modt{Q_t})^2) \md t,
			\end{align*}
			for some parameters $\alpha < 0$, and $\gamma, \beta, r, \tau_0, \tau_1, \tau_2 \geq 0$.
			In this situation, we have the uniform moment estimates 
			\begin{equation*}
				1 \leq n < 1 + 2 \min\{\nicefrac{\beta}{\tau_2}, \nicefrac{|\alpha|}{\tau_1}\} \implies \sup_{t \geq 0 } \E[\modt{G}(\modt{Q_t})^n] < \infty
			\end{equation*}
			with the convention $\nicefrac{\beta}{0} = \infty = |\alpha|/0$ when $\tau_2 = 0$ or when $\tau_1 = 0$.
		\item[(b)] Assume that 
			\begin{align*}
				\Lcal_t \modt{G} (\modt{Q_t}) &\leq 2 \tau_t (\modt{Q_t}) \sqrt{\modt{G}(\modt{Q_t})} + 2 \alpha \modt{G} (\modt{Q_t}) + \beta_t (\modt{Q_t}), 
				\\
				\md [\Mcal_t, \Mcal_t]_t &\leq \modt{G}(\modt{Q_t}) \gamma_t (\modt{Q_t}) \md t,
			\end{align*}
			for some $\alpha < 0$ and some nonnegative functions $\tau_t, \beta_t, \gamma_t$ such that, \modt{for a given $n \geq 1$,}
			\begin{align*}
				\delta_{\tau,t}(n) &= \E[ \tau_t(\modt{Q_t})^n ]^{\frac{1}{n}} < \infty &
				\delta_{\beta,t} &= \E [\beta_t(\modt{Q_t})^n]^{\frac{1}{n}} < \infty &
				\delta_{\gamma,t} &= \E [\gamma_t(\modt{Q_t})^n]^{\frac{1}{n}} < \infty.
			\end{align*}
			Then, \modt{for that $n$,} it holds	
			\begin{multline*}
				\E[\modt{G}(\modt{Q_t})^n]^{\frac{1}{n}} \leq e^{\alpha t} \E[\modt{G}(\modt{Q_t})^n]^{\frac{1}{n}}
				+ \int_{0}^t e^{\alpha (t-s)} \left[ \frac{\delta_{\tau,s}(2n)^2}{|\alpha|} + \delta_{\beta,s}(n) + \frac{n-1}{2} \delta_{\gamma,s} \right] \md s.
			\end{multline*}
	\end{itemize}
\end{lemma}
}

%\begin{lemma}
%	Assuming $\E_P[\hbY_0] = \bm{0}_{1 \times R}$, then $\E_P[\hbY_t] = \bm{0}_{1 \times R}$.
%\end{lemma}
%\begin{proof}
%Since
%\begin{align*}
%	\E_P[\hbY_t^{\top}] - \E_P[\hbY_0^{\top}] 
%		&= \int_{0}^t \bU_t^{\top} \widehat{\tilde{A}} \bU_t \E_P[\hbY_t^{\top}] \md t,
%		%+ \int_0^{t} \bU_t^{\top} \Soh \md \E_P[\hWcal_t^{\star}] \\
%		%&\quad + \int_0^{t} \bU_t^{\top} \widehat{P}_t H^{\top} \Gmoh \md \E_P[\hVcal_t^{\star}]
%		%\\
%		%&= \int_{0}^t \bU_t^{\top} \widehat{\tilde{A}} \bU_t \E_P[\hbY_t] \md t
%\end{align*}
%it holds
%\begin{equation*}
%	\norm{\E_P[\hbY_t^{\top}]} \leq \int_0^t \norm{\bU_t \widehat{\tilde{A}} \bU_t} \norm{\E_P[\hbY_t^{\top}]} \md t
%\end{equation*}
%and the result follows by applying \modt{Gronwall's} inequality~\cite{NumericalApproQuarte1994}.
%\end{proof}
%\noindent \modt{This property ensures that the particles $\widehat{\Xcal}_t^{(p)} = \modt{\widehat{U}}_t^0 + \bU_t \hbY_t^{(p)}$ are naturally decomposed into $\modt{\widehat{U}}_t^0$ (sample mean) and its zero-sample mean stochastic fluctuation $\bU_t \hbY_t^{(p)}$.
%Furthermore, the ``stochastic excess'' added to the sample mean allows to recover the following consistent formulation:
%}

\subsection{Well-posedness for fully observed dissipative dynamics}

In this section, we prove the well-posedness of the \dlr-\enkf~\mods{system~\eqref{eqn:mean-mode-particles} -- \eqref{eqn:stoch-modes-particles}} under \modt{two assumptions. 
\mods{The first is}} the (strong) assumption of fully observed dynamics \modt{$H = \bm{I}_d$ and i.i.d Gaussian observation noise ($\Gamma = \rho^{-1} \bm{I}_d$)}, i.e., 
\begin{equation} \label{eqn:full-obs}
	S = \rho \bm{I}_d, \quad \rho > 0. \tag{A1}
\end{equation}
\modt{While usually not applicable in practice, the assumption of fully observed dynamics is not uncommon when analysing theoretical properties of the \enkf, see e.g.~\cite{LongTimeStabiDeWil2018,OnTheStabilitDelMo2018,WellPosedKelly2014}, and in the following sections our proofs rely several times on~\ref{eqn:full-obs}.
}
\mods{
	In our approach, proving the well-posedness of the \dlr-\enkf~system is contingent on establishing (uniform) upper moment bounds of the sample covariance $\widehat{P}_t$. 
	Such bounds are in general not readily available, and therefore the second assumption restricts the dynamics to the subclass of dissipative dynamics	
\begin{equation} \label{eqn:lAsymneg}
	\lambda_{\max}(A + A^{\top}) < 0. \tag{A2}
\end{equation}
Under this assumption and a technical condition \modr{on} ${\rb{N}}$ being sufficiently large, we obtain moment bounds on $\widehat{P}_t$, allowing us to establish well-posedness.
}

\mods{Under the assumption of fully observed dynamics~\ref{eqn:full-obs}, we obtain the following result.}
\begin{lemma} \label{lem:well-pos-Yhat}
Assuming~\ref{eqn:full-obs} holds true, then~\cref{eqn:stoch-modes-particles} is well-posed.
\end{lemma}
\begin{proof}
	\mods{We use here~\cite[Theorem II.3.5]{StochasticDiffMao2008}, which states that a $n$-dimensional SDE $\md \xi_t = a(t, \xi_t) \md t + b(t, \xi_t) \md \Wcal_t$ with $a$ and $b$ locally Lipschitz and satisfying a monotonicity condition 
		\begin{equation*}
			\xi_t^{\top} a(t, \xi_t) + \frac{1}{2} \norm{b(t,\xi_t)}^2_F \leq K (1 + \norm{\xi_t}^2)
		\end{equation*}
	has a unique strong solution in the sense of~\cite[Definition II.2.1]{StochasticDiffMao2008}, that is the SDE holds for every $t \geq 0$ almost-surely, $\xi_t$ is continuous almost-surely and $\Acal_t$-adapted, and $a(t, \xi_t) \in L^1([0, t], \R^n)$ and $b(t, \xi_t) \in L^2([0,t], \Mat{n\times n})$.}
Define $A_{\sym} = \frac{1}{2} (A + A^{\top})$ and the symmetric matrix $\mS \in \Mat{{\rb{N}} \times {\rb{N}}}$ by $\mS_{ij} = \delta_{ij} - \nicefrac{1}{{\rb{N}}}$, and rewrite $\hWcal^{\star}_t = \hWcal_t \mS$.
Vectorising~\cref{eqn:stoch-modes-particles} yields the non-linear $R{\rb{N}}$-dimensional SDE
\begin{multline} \label{eqn:vec-dY}
	\md \mvec(\hbY_t^{\top}) 
	= 
	(\bm{I}_{\rb{N}} \otimes \bU_t^{\top}\widehat{\tilde{A}} \bU_t) \mvec(\hbY_t^{\top}) \md t
	+ (\mS \otimes \bU_t^{\top} \Soh) \md \mvec(\hWcal_t) 
	\\
	+ (\mS \otimes \bU_t^{\top} \widehat{P}_t H^{\top} \Gmoh) \md \mvec(\hVcal_t),
\end{multline}
having used $\mvec(ABC) = (C^{\top} \otimes A) \mvec(B)$.
Left-multiplying the drift term by $\mvec(\hbY_t^{\top})^{\top}$ yields
\begin{align*}
	\mvec(\hbY^{\top}_t)^{\top} (\bm{I}_{\rb{N}} \otimes \bU_t^{\top}\widehat{\tilde{A}} \bU_t) \mvec(\hbY^{\top}_t) 
		&= \tr(\hbY_t \bU_t^{\top} \widehat{\tilde{A}} \bU_t \hbY_t^{\top}) 
	\\
		&= \tr(\hbY_t \bU_t^{\top}  A \bU_t \hbY_t^{\top}) 
		- \rho \tr(\hbY_t \bU_t^{\top} \widehat{P}_t \bU_t \hbY_t^{\top})
	\\
		&= \tr(\hbY_t \bU_t^{\top}  A_{\sym} \bU_t \hbY_t^{\top}) 
		- \rho \tr(\hbY_t \bU_t^{\top} \widehat{P}_t \bU_t \hbY_t^{\top})
		\\
		&\leq \lambda_{\max}(\bU_t^{\top} A_{\sym} \bU_t) \tr(\hbY_t^{\top} \hbY_t) - \rho \tr(\hbY_t \bU_t^{\top} \widehat{P}_t \bU_t \hbY_t^{\mods{\top}})
		\\
		&\leq \lambda_{\max}(A_{\sym}) \norm{\mvec(\hbY_t)}^2_{\mods{F}} - \frac{\rho}{{\rb{N}}-1} \tr(\hbY_t \hbY_t^{\top} \hbY_t \hbY_t^{\top}).
\end{align*}
On the other hand,
\begin{multline*}
	\tr( (\mS \otimes \Soh \bU_t ) (\mS \otimes \bU_t^{\top} \Soh) ) +  \tr((\mS \otimes \Gmoh H \widehat{P}_t \bU_t) (\mS \otimes \bU_t^{\top} \widehat{P}_t H^{\top} \Gmoh)) 
	\\
	= \tr( \mS^2 \otimes \Soh \bU_t \bU_t^{\top} \Soh ) + \tr(\mS^2 \otimes \Gmoh H \widehat{P}_t \bU_t \bU_t^{\top} \widehat{P}_t H^{\top} \Gmoh)
	\\
	= \tr(\mS^2) \tr(\bU_t^{\top} \Sigma \bU_t) + \frac{\rho}{({\rb{N}}-1)^2} \tr(\mS^2) \tr( \mods{\bU_t} \hbY_t^{\top} \hbY_t  \hbY_t^{\top} \hbY_t \mods{\bU_t^{\top}})
	\\
	\leq ({\rb{N}}-1) \tr(\Sigma) + \frac{\rho}{{\rb{N}}-1} \tr(\hbY_t  \hbY_t^{\top} \hbY_t \hbY_t^{\top} ),
\end{multline*}
having used \mods{$\mS^2 = \mS$ and $\tr(\mS) = {\rb{N}}-1$}.  
Hence, the sum of the two above developments is bounded by
\begin{equation*}
	\max\{\lambda_{\max}(A_{\sym}), ({\rb{N}}-1) \tr(\Sigma) \} (1 + \norm{\mvec(\hbY_t)}^2),
\end{equation*}
\modt{leading to a monotonicity condition,}
which, along \modt{with} the local Lipschizianity of the drift and diffusion of the vectorised equation, implies the well-posedness~\cite[Theorem II.3.5]{StochasticDiffMao2008}.

%\modt{To show the well-posedness of~\cref{eqn:mean-mode-particles}, again expanding $\md \Zcal_t$ via~\cref{eqn:truth-noisy-observations}, consider the system given by the augmented state system $Q_t = [\modt{\widehat{U}}_t^0, \Xcal_t^{\signal}, \mvec(\hbY_t^{\top})]$, governed by~\cref{eqn:mean-mode-particles,eqn:truth-model-dynamics,eqn:stoch-modes-particles}, and written in condensed form as
%	\begin{equation*}
%		\md [\modt{\widehat{U}}_t^0, \Xcal_t^{\signal}, \mvec(\hbY_t^{\top})] = (f_1, f_2, f_3) \md t + (g_1, g_2, g_3) \md [\mvec(\hWcal_t), \mvec(\hVcal_t), \widetilde{\Wcal}_t, \widetilde{\Vcal}_t].
%	\end{equation*}
%	One verifies that the coefficients in this system are locally Lipschitz, and we proceed again by recovering a monotonicity condition. 
%	Observe that $\mvec(\hbY_t^{\top})^{\top} f_3 + \norm{g_3}^2_F \leq C(1 + \norm{Q_t}^2)$, as it is the same term as in the first computation. 
%	Next, $(\Xcal^{\signal}_t)^{\top} f_2 + \norm{g_2}^2_F \leq C(1 + \norm{Q_t}^2)$, since the dynamics of $f_2$ are linear-affine and $g_2 = \Sigma^{\oh}$.
%	Finally, 
%	\begin{align*}
%		(\modt{\widehat{U}}_t^0)^{\top} f_1 
%		&=
%		(\modt{\widehat{U}}_t^0)^{\top} (A - \widehat{P}_t S) \modt{\widehat{U}}_t^0 + (\modt{\widehat{U}}_t^0)^{\top} f + (\modt{\widehat{U}}_t^0)^{\top} \widehat{P}_t S \Xcal_t^{\signal} \\
%		&\leq \frac{1}{2} (\lambda_{\max}(A_{\sym}) + 1) \norm{\modt{\widehat{U}}_t^0}^2 + \frac{1}{2} \norm{f}^2 + \rho | (\modt{\widehat{U}}_t^0)^{\top} \widehat{P}_t \Xcal_t^{\signal} |
%	\end{align*}
%}
\end{proof}

\modt{
	In the setting of full observations, establishing the well-posedness of the \dlr-\enkf~system~\eqref{eqn:mean-mode-particles} to~\eqref{eqn:stoch-modes-particles} therefore reduces to checking the well-posedness of the sample mean equation~\eqref{eqn:mean-mode-particles}, as the well-posedness of~\cref{eqn:phys-modes-particles} was established \modr{already} in~\cref{sssec:phys-modes}.
	This last step is not immediate, as the sample mean depends on the process $\widehat{P}_t$ and, \mods{as discussed beforehand, suitable (uniform) upper bounds on the moments are not known \textit{a priori} for general dynamics}.
	\mods{In what follows, we therefore restrict ourselves to dissipative dynamics \modr{satisfying assumption}~\ref{eqn:lAsymneg} and establish well-posedness in that setting.}
	%In this work, we therefore consider the subclass of dissipative dynamics	
%\begin{equation} \label{eqn:lAsymneg}
%	\lambda_{\max}(A + A^{\top}) < 0, \tag{A3}
%\end{equation}
%for which, under a technical condition of ${\rb{N}}$ being sufficiently large, we obtain moment bounds on $\widehat{P}_t$, allowing us to establish well-posedness.
	We begin \mods{with a few preparatory results. 
	The proofs of the following \rb{three statements} are housed in Appendix~\ref{app:b}.}
\mods{
\begin{lemma} \label{lem:beta-gamma bounds}
	Assuming~\ref{eqn:full-obs} and~\ref{eqn:lAsymneg}, then for $\modr{{\rb{N}} > 4R - 1}$ it holds
	\begin{equation}
		\sup_{t \geq 0} \E[(\tr(\hP_{\hbY_t}))^{n}] < \infty.
	\end{equation}
\end{lemma}
\begin{lemma} \label{lem:trMYt-sup-bounds}
	Assuming~\ref{eqn:full-obs} and~\ref{eqn:lAsymneg}, then for $\modr{{\rb{N}} > 4R -1}$ it holds
	\begin{equation*}
		\E[\sup_{0 \leq s \leq t}(\tr(\hP_{\hbY_s}))^{2}] < \infty\modr{, \quad \forall t > 0.}
	\end{equation*}
\end{lemma}
}

\begin{proposition} \label{prop:vec-Yt-bounds}
	Assuming~\ref{eqn:full-obs} and~\ref{eqn:lAsymneg}, \rb{then for $\modr{{\rb{N}} > 4R -1}$, it holds}
	\begin{equation*}
		\E[\sup_{0 \leq s \leq t} \norm{\mvec(\hbY_s)}^2] < \infty\modr{, \quad \forall t > 0.}
	\end{equation*}
	Additionally, the sample paths of $\hbY_s$ are bounded almost surely. 
\end{proposition}
\mods{With the previous result we can now show the well-posedness of the \dlr-\enkf.}
}

\modt{
\begin{theorem}  
	Assuming~\ref{eqn:full-obs} and~\ref{eqn:lAsymneg} hold true and that the number of particles verifies $\modr{{\rb{N}} > 4R - 1}$.
	Then the \dlr-\enkf~system~\eqref{eqn:mean-mode-particles} to~\eqref{eqn:stoch-modes-particles} is well-posed. 
\end{theorem}
\begin{proof}
For the well-posedness of~\cref{eqn:mean-mode-particles}, begin by expanding $\md \Zcal_t$ via~\cref{eqn:truth-noisy-observations} to obtain
	\begin{multline*}
	\md \widehat{U}_t^0 \! =
	(A - \widehat{P}_t S)\modt{\widehat{U}}_t^0 \md t 
	+ (\modt{f} + \widehat{P}_t S \Xcal_t^{\signal}) \md t 
	+ \modt{\Pi_{\bU_t}} \Soh \rd \E_{\rb{N}} [\hWcal_t] 
	+ \widehat{P}_t H^{\top} \Gmoh  \md (\widetilde{\Vcal}_t - \E_{\rb{N}} [\widehat{\Vcal}_t]).
	%= \tilde{a}(t, \widehat{U}_t^0, \widehat{P}_t) \md t + \tilde{b}(t, \widehat{U}_t^0, \widehat{P}_t) \md [\widehat{\Wcal}_t, \widehat{\Vcal}_t, \widetilde{\Vcal}_t].
	\end{multline*}
	%Note the explicit dependence on $\widehat{P}_t$.
	Combining with~\cref{eqn:stoch-modes-particles}, we consider the augmented state $[\widehat{U}_t^0, \hbY_t]$ system
	\begin{equation*}
		\md [\widehat{U}_t^0, \hbY_t] = a(t, \widehat{U}_t^{0}, \hbY_t) \md t + b(t, \widehat{U}_t^{0}, \hbY_t) \md [\widehat{\Wcal}_t, \widehat{\Vcal}_t, \widetilde{\Vcal}_t].
		%\quad \widehat{P}_t = \frac{1}{P-1} \bU_t \hbY_t^{\top} \hbY_t \bU_t^{\top},
	\end{equation*}
	\mods{for suitable drift $a$ and diffusion $b$.}
	%i.e., the sample covariance equation closes the SDE equation. 
	%\begin{equation*}
	%	\md \hbY_t^{\top} = \tilde{e}(t, \hbY_t, \widehat{P}_t) \md t + \tilde{f}(t, \hbY_t, \widehat{P}_t) \md [\widehat{\Wcal}_t^{\star}, \widehat{\Vcal}_t^{\star}].
	%\end{equation*}
	%Consider the truncation operator 
	%\begin{align*}
	%	\Rcal_n : \Mat[0]{d \times d} &\to \Mat[0]{d \times d}
	%	\\
	%	P = U \mathrm{diag}(\lambda_1, \ldots, \lambda_d) U^{\top} &\mapsto \Rcal_n(P) = U \mathrm{diag}(\max\{\lambda_1, n\}, \ldots, \max\{\lambda_d, n\}) U^{\top},
	%\end{align*}
	%where $\lambda_1 \geq \ldots \geq \lambda_d \geq 0$.
	%If $P$ has rank $R$, then $\lambda_{R+1} = \ldots = \lambda_d = 0$, and $\norm{\Rcal_n(P)}_F^2 \leq R n^2$.
	%%In a slight abuse of notation, for rank-$R$ matrices $P = U M U^{\top}$, with $M \in \Mat[0]{R \times R}$, we extend $\Rcal(P) = U \Rcal_n(M) U^{\top}$.
	%Furthermore, if $P$ is a $\Mat[0]{d \times d}$-valued random variable with $k$-th moment bound ($k \geq 1$), $\norm{\Rcal_n(P)}_F \leq \norm{\Rcal_{n+1}(P)}_F \leq \ldots \leq \norm{P}_F$ pointwise in $\omega$ and by the Lebesgue dominated convergence theorem, 
	%\begin{equation*}
	%	\lim_{n \rightarrow \infty} \E[\norm{\Rcal_n(P)}^k_F] = \E[\norm{P}^k_F].
	%\end{equation*}
	We consider the sequence of truncated problems
	\begin{align*}
		\md [\widehat{U}_t^{0,n}, \hbY_t^{n}] &= a_n(t, \widehat{U}_t^{0,n}, \hbY_t^{n}) \md t + b_n(t, \widehat{U}_t^{0,n}, \hbY_t^{n}) \md [\widehat{\Wcal}_t, \widehat{\Vcal}_t, \widetilde{\Vcal}_t]
	\end{align*}
	where
	\begin{equation*}
		a_n(t, \widehat{U}_t^{0,n}, \hbY_t^{n}) = 
		\left\{
			\begin{aligned}
				&a(t, \widehat{U}_t^{0,n}, \hbY^{n}_t) && \text{if } |(\widehat{U}_t^{0,n}, \hbY_t^n)| \leq n
				\\
				&a(t, n \widehat{U}_t^{0,n} / |(\widehat{U}_t^{0,n}, \hbY_t^n)|, n \hbY^{n}_t / |(\widehat{U}_t^{0,n}, \hbY_t^n)|) && \text{if } |(\widehat{U}_t^{0,n}, \hbY_t^n)| > n,
			\end{aligned}
			\right.
	\end{equation*}
	where $|(\widehat{U}_t^{0,n}, \hbY_t^n)|^2 = \norm{\widehat{U}_t^{0,n}}^2 + \norm{\mvec(\modr{\hbY_t^n})}^2$, and similarly for $b_n$. 
	The coefficients of the truncated problems satisfy a Lipschitz condition and a linear growth bound-type condition, ensuring the well-posedness of each problem on $[0,T]$. 
	Next, define $\tau_n = t \wedge \inf \{ 0 \leq s \leq t \,:\, |(\widehat{U}_t^{0,n}, \hbY_t^n)| \geq n \}$. 
	\mods{We focus on the stopped process $(\widehat{U}_{t \wedge \tau_n}^{n,0}, \hbY^{n,0}_{t \wedge \tau_n})$, and will show it is bounded almost-surely uniformly in $n$.}	
	We first deal with the stochastic modes, as they are independent of $\widehat{U}_{\mods{t \wedge \tau_n}}^{0,n}$. 
	%Define $\tilde{\tau}_n = t \wedge \inf \{ 0 \leq s \leq t \,:\, \norm{ \mvec(\hbY_s)} > n\}$, where $\hbY_t$ the solution to~\cref{eqn:stoch-modes-particles}.
	Since up to $\tau_n$, the equations for the truncated process is identical to~\cref{eqn:stoch-modes-particles}, of which the solution is known to be unique, it holds $\mods{\hbY^n_{s \wedge \tau_n}} = \hbY^n_s = \hbY_s$ a.s. for $s \in [0, \tau_n]$. 
	By~\Cref{prop:vec-Yt-bounds}, $\norm{\mvec(\hbY_t)}$ is a.s. bounded, and therefore in order to verify that $|(\widehat{U}^{0,n}_{\mods{t \wedge \tau_n}}, \hbY_{\mods{t \wedge \tau_n}}^{n})|$ is a.s. bounded \mods{uniformly in $n$}, it remains to verify that $\widehat{U}^{0,n}_{\mods{t \wedge \tau_n}}$ is a.s. bounded \mods{uniformly in $n$}. 
	\mods{Developing}
	\begin{align*}
		\norm{\widehat{U}_{\mods{t \wedge \tau_n}}^{0,n}}^2 - \norm{\widehat{U}_0^{0,n}}^2
		&=
	\int_0^{\mods{t \wedge \tau_n}}	2 \langle \widehat{U}_{\mods{u \wedge \tau_n}}^{0,n} , \md \widehat{U}_{\mods{u \wedge \tau_n}}^{0,n} \rangle 
	+ \sum_{i=1}^d \md [ \widehat{U}_{\mods{u \wedge \tau_n}}^{0,n}(i), \widehat{U}_{\mods{u \wedge \tau_n}}^{0,n}(i)]_{u} 
	\\
		&= \int_0^{\mods{t \wedge \tau_n}} \Big\{ \langle \widehat{U}_{\mods{u \wedge \tau_n}}^{0,n} , (A + A^{\top}) \widehat{U}_{\mods{u \wedge \tau_n}}^{0,n} \rangle - 2 \rho \langle \widehat{U}_{\mods{u \wedge \tau_n}}^{0,n}, \widehat{P}^n_{\modr{u \wedge \tau_n}} \widehat{U}_{\mods{u \wedge \tau_n}}^{0,n} \rangle 
		\\
		&\quad + 2 \langle \widehat{U}_{\mods{u \wedge \tau_n}}^{0,n} , f_{\modr{u}} \rangle + 2 \langle \widehat{U}_{\mods{u \wedge \tau_n}}^{0,n} , \widehat{P}^n_{\mods{u \wedge \tau_n}} \Xcal_{u}^{\signal} \rangle \Big\} \md u 
		\\
		&\quad + 2 \int_0^{\mods{t \wedge \tau_n}} (\widehat{U}^{0,n}_{\mods{u \wedge \tau_n}})^{\top} \mods{\Pi_{\bU_{u}}} \Sigma^{\oh} \md \E_{\rb{N}}[\widehat{\Wcal}_{u}] 
		\\
		&\quad + 2 \int_0^{\mods{t \wedge \tau_n}} (\widehat{U}^{0,n}_{\mods{u \wedge \tau_n}})^{\top} \widehat{P}^{\mods{n}}_{\mods{u \wedge \tau_n}} H^{\top} \Gmoh \md (\widetilde{\Vcal}_{u} - \E_{\rb{N}}[\widehat{\Vcal}_{u}]) 
		\\
		&\quad + \int_0^{\mods{t \wedge \tau_n}} \left( \tr(\mods{\Pi_{\bU_{u}}} \Sigma \mods{\Pi_{\bU_{u}}}) + \rho \tr((\widehat{P}^n_{\mods{u \wedge \tau_n}})^2) \right) \md u
		\\
		& \leq \int_0^{\mods{t \wedge \tau_n}} \modr{\frac{1}{4t}} \norm{\widehat{U}^{0,n}_{\mods{u \wedge \tau_n}}}^2 \md u 
		+  \int_0^{\mods{t}} \modr{8t} \left( \norm{f_{\modr{u}}}^2 + \norm{\mods{\widehat{P}_{u}}}_{\mods{F}}^2 \norm{\Xcal_{u}^{\signal}}^2 \right) \md u
		\\
		&\quad + \tr(\Sigma)t + \rho \int_0^{\mods{t}} \tr(\mods{\widehat{P}_{\mods{u}}})^2  \md u
		\\
		&\quad + 2 \left| \mods{\frac{1}{{\rb{N}}}\sum_{p=1}^{\rb{N}}} \int_{0}^{\mods{t \wedge \tau_n}} (\widehat{U}^{0,n}_{\mods{u \wedge \tau_n}})^{\top} \mods{\Pi_{\bU_{u}}} \Sigma^{\oh} \mods{\md \hWcal^{(p)}_{u}} \right| 
		\\
		&\quad + 2 \left| \mods{\frac{1}{{\rb{N}}}\sum_{p=1}^{\rb{N}}} \int_{0}^{\mods{t \wedge \tau_n}} (\widehat{U}^{0,n}_{\mods{u \wedge \tau_n}})^{\top} \widehat{P}^n_{\mods{u \wedge \tau_n}} H^{\top} \Gmoh \md (\widetilde{\Vcal}_{u} - \mods{{\hVcal}^{(p)}_{u}})\right|.
	\end{align*}
	\mods{having used \modr{the $\varepsilon$-Young inequality to obtain the first two terms of the upper bound,} the fact that $\widehat{P}^n_{s \wedge \tau_n} = \widehat{P}_{s}$ for $s \in [0,\tau_n]$, and extending the upper bound of the integrals depending only on $\widehat{P}_s$ to $t$ by positivity.}
	Taking the supremum over $[0,t]$ and then the expectation, the first term of the above expression is bounded by
	\begin{equation*}
		\E\left[\sup_{0 \leq s \leq t} \int_{0}^{s \wedge \tau_n} \frac{1}{4t} \norm{\widehat{U}^{0,n}_{u \wedge \tau_n}}^2 \md u \right] 
		\leq 
		\frac{1}{4} \E[ \sup_{0 \leq s \leq t} \norm{\widehat{U}^{0,n}_{s \wedge \tau_n}}^2].
	\end{equation*}
	Using this bound, and applying the Burkholder–Davis–Gundy inequality~\cite[Theorem I.7.3]{StochasticDiffMao2008} for the last two terms, yields
	\begin{multline*}
		\E[\sup_{0 \leq s \leq t} \norm{\widehat{U}^{0,n}_{{s \wedge \tau_n}}}^2] 
		\leq  {\frac{1}{4} \E[ \sup_{0 \leq s \leq {t}} \norm{\widehat{U}^{0,n}_{s \wedge \tau_n}}^2 ]} 
		+ {8t} \int_0^t \left(  \norm{f_{{u}}}^2 + \norm{\Xcal_u^{\signal}}^2 {\E[\norm{\widehat{P}_u}^2_F]} \right) {\md u} 
		\\
		+\E[\norm{\widehat{U}^{0,n}_0}^2] 
		+ \tr(\Sigma)t 
		+ \rho \int_{0}^t {\E[\tr(\widehat{P}_u)^2]} \md u
		+ {C_{1}} \E \left[ \left( \int_0^{{t \wedge \tau_n}} (\widehat{U}^{0,n}_{{u \wedge \tau_n}})^{\top} {\Pi_{\bU_u}} \Sigma {\Pi_{\bU_u}} \widehat{U}^{0,n}_{{u \wedge \tau_n}}) \md u \right)^{\oh} \right]
	\\
		+ {C_2} \E \left[ \left( \int_{0}^{{t \wedge \tau_n}}(\widehat{U}^{0,n}_{{s \wedge \tau_n}})^{\top} {(\widehat{P}^n_{{s \wedge \tau_n}})^2} \widehat{U}^{0,n}_{{s \wedge \tau_n}}  \md s \right)^{\oh} \right],
	\end{multline*}
%	\begin{align*}
%		\E[\sup_{0 \leq s \leq t} \norm{\widehat{U}^{0,n}_{\mods{s \wedge \tau_n}}}^2] 
%		&\leq  \modr{\frac{1}{4} \E[ \sup_{0 \leq s \leq \modr{t}} \norm{\widehat{U}^{0,n}_{s \wedge \tau_n}}^2 ]} 
%		+ \modr{8t} \int_0^t \left(  \norm{f_{\modr{u}}}^2 + \norm{\Xcal_u^{\signal}}^2 \mods{\E[\norm{\widehat{P}_u}^2_F]} \right) \mods{\md u} 
%		\\
%		&+\E[\norm{\widehat{U}^{0,n}_0}^2] 
%		+ \tr(\Sigma)t 
%		+ \rho \int_{0}^t \mods{\E[\tr(\widehat{P}_u)^2]} \md u
%		\\
%		&+ \mods{C_{1}} \E \left[ \left( \int_0^{\mods{t \wedge \tau_n}} (\widehat{U}^{0,n}_{\mods{u \wedge \tau_n}})^{\top} \mods{\Pi_{\bU_u}} \Sigma \mods{\Pi_{\bU_u}} \widehat{U}^{0,n}_{\mods{u \wedge \tau_n}}) \md u \right)^{\oh} \right]
%	\\
%		&+ \mods{C_2} \E \left[ \left( \int_{0}^{\mods{t \wedge \tau_n}}(\widehat{U}^{0,n}_{\mods{s \wedge \tau_n}})^{\top} \mods{(\widehat{P}^n_{\mods{s \wedge \tau_n}})^2} \widehat{U}^{0,n}_{\mods{s \wedge \tau_n}}  \md s \right)^{\oh} \right],
%	\end{align*}
	where $C_1 = \frac{2 \sqrt{32}}{\modr{\sqrt{{\rb{N}}}}}$, and $C_2 = \rho C_1 \modr{\sqrt{{\rb{N}}+1}}$, both depending on ${\rb{N}}$, but admitting the uniform upper bound $C_1,C_2 \leq C \coloneqq 4 \sqrt{32} \max\{\rho,1\}$ for any ${\rb{N}} \geq 1$.
	%where $C = \mods{2 \sqrt{32}}(1 + \frac{1}{{\rb{N}}}) \max\{1, \rho\} \leq 8 \mods{\sqrt{2}} \max\{1, \rho\}$.
	To bound the term in the last line, \mods{we again leverage the fact that $\widehat{P}^n_{s \wedge \tau_n}$ coincides with $\widehat{P}_s$ until the stopping time, and then extend the integral bound to $t$; this yields the first inequality below, which is then developed as} 
	\begin{align*}
		C\E \left[ \left( \int_{0}^{\mods{t \wedge \tau_n}} (\widehat{U}^{0,n}_{\mods{s \wedge \tau_n}})^{\top} (\widehat{P}^n_{\mods{s \wedge \tau_n}})^2 \widehat{U}^{0,n}_{\mods{s \wedge \tau_n}}  \md s \right)^{\oh} \right] 
		&\leq C\E \left[ \left( \int_{0}^{\mods{t}} \sup_{0 \leq s' \leq t} \lambda_{\max}(\mods{\widehat{P}_{\mods{s'}}})^2 \norm{\widehat{U}^{0,n}_{\mods{s \wedge \tau_n}}}^2  \md s \right)^{\oh} \right]
		\\
		&\leq C\E\left[ \left( \sup_{0 \leq s \leq t} \tr(\mods{\widehat{P}_{s}})^2  \sup_{0 \leq s \leq t} \norm{\widehat{U}^{0,n}_{\mods{s \wedge \tau_n}}}^2 t \right)^{\oh} \right]
		\\
		&\leq \E \left[ C t^{\oh} \sup_{0 \leq s \leq t} \tr(\mods{\widehat{P}_{s}})  \sup_{0 \leq s \leq t} \norm{\widehat{U}^{0,n}_{\mods{s \wedge \tau_n}}} \right]
		\\
		&\leq  C^2 t \E \left[ \sup_{0 \leq s \leq t} \tr(\mods{\widehat{P}_{s}})^2 \right] + \frac{1}{4} \E \left[ \sup_{0 \leq s \leq t} \norm{\widehat{U}^{0,n}_{\mods{s \wedge \tau_n}}}^2 \right],
	\end{align*}
	having used the $\varepsilon$-Young inequality in the last line.
	Proceeding similarly for the first term,
	\begin{equation*}
		\mods{C} \E \left[ \left( \int_0^{\mods{t \wedge \tau_n}} (\widehat{U}^{0,n}_{\mods{u \wedge \tau_n}})^{\top} \mods{\Pi_{\bU_u}} \Sigma \mods{\Pi_{\bU_u}} \widehat{U}^{0,n}_{\mods{u \wedge \tau_n}}) \md u \right)^{\oh} \right]
		\leq 
		C^2 \tr(\Sigma) t + \frac{1}{4} \E \left[ \sup_{0 \leq s \leq t} \norm{\widehat{U}^{0,n}_{\mods{s \wedge \tau_n}}}^2 \right]
	\end{equation*}
	Combining the last 3 bounds and rearranging 
	\begin{multline*}
		\E[\sup_{0 \leq s \leq t} \norm{\widehat{U}^{0,n}_{\mods{s \wedge \tau_n}}}^2] 
		\leq  
		K \Big( \E[\norm{\widehat{U}^{0,n}_0}^2] 
			+ \int_0^t \norm{f_{\modr{u}}}^2 \md u 
		+ \left( \sup_{0 \leq s \leq t} \E[\norm{\widehat{P}_s}^2] \right) \int_{0}^t \norm{\Xcal_u^{\signal}}^2 \md u 
		\\
		+ \tr(\Sigma)t 
		+ \rho \int_{0}^t \E[\tr(\widehat{P}_u)^2] \md u 
		+ C^2 \tr(\Sigma) t
		+ C^2 t \E\left[ \sup_{0 \leq s \leq t} \tr(\widehat{P}_s)^2 \right]
		\Big).
	\end{multline*}
	By~\cref{lem:beta-gamma bounds} and~\cref{lem:trMYt-sup-bounds}, all the terms in the rhs involving $\widehat{P}_s$ are bounded.
	\modr{This} shows that $\E[\sup_{0 \leq s \leq t} \norm{\widehat{U}^{0,n}_{\mods{s \wedge \tau_n}}}^2] < \mods{M}$ \mods{uniformly in $n$}, and following a \mods{Borel-Cantelli argument} as in~\Cref{prop:vec-Yt-bounds}, we obtain that $\norm{\widehat{U}^{0,n}_{\mods{s \wedge \tau_n}}}$ is \mods{a.s.} bounded \mods{uniformly in $n$}. 
	This implies that $\tau_n \rightarrow \modr{\infty}$ as $n \rightarrow \infty$.
	% OLD PROOF
	%\begin{multline*}
	%	\frac{1}{2} \E\norm{\widehat{U}_t^{0,n}}^2 \leq \E\norm{\widehat{U}_t^{0,n}}^2 + \tr(\Sigma) t + t \sup_{0 \leq s \leq t} \E \norm{\widehat{P}_s}^2_F
	%	\\
	%	+ t^2 \left( \max_{0 \leq r \leq t} \norm{\Xcal_r^{\signal}} \cdot \sup_{0 \leq s \leq t} \E[\norm{\widehat{P}_s}^2_F]^{\oh}  + \norm{f} \right)^2.
	%\end{multline*}
	%As 
	%\begin{equation*}
	%	\norm{\widehat{P}_t}^2_F = \tr(\widehat{P}_t^2) = \tr(\hP_{\hbY_t}^2) = \sum_{i=1}^R \lambda_i(\hP_{\hbY_t})^2 \leq \left(\sum_{i=1}^R \lambda_i(\hP_{\hbY_t})\right)^2 \leq (\tr(\hP_{\hbY_t}))^2,
	%\end{equation*}
	%taking expectations, under the condition $P > R+1$, $\E[(\tr(\hP_{\hbY_t}))^2]$ is bounded by~\cref{lem:beta-gamma bounds}, and therefore the second moments of $\widehat{U}_t^{0,n}$ are uniformly bounded.
	%	
	For $n,m > 0$, let $\sigma = \min\{\tau_n, \tau_m\}$.
	Therefore,	
	\begin{equation*}
		\bbP(\norm{\widehat{U}_{\mods{t \wedge \tau_n}}^{0,n} - \widehat{U}_{\mods{t \wedge \tau_m}}^{0,m}} > 0) \leq \bbP(\sigma < t) \leq \bbP(\tau_n < t) + \bbP(\tau_m < t) \rightarrow 0 \quad \text{for } n,m \rightarrow \infty,
	\end{equation*}
	which implies $\widehat{U}_{\mods{t \wedge \tau_n}}^{0,n} \xrightarrow{n \to \infty} \widehat{U}_t^{0}$ in probability. 
	Note that, for $s \leq \sigma$, $\widehat{U}_{\mods{s \wedge \tau_n}}^{0,n} = \widehat{U}_{\mods{s \wedge \tau_m}}^{0,m}$ almost surely.
	Therefore, for $s \leq \tau_n$,
	\begin{equation*}
		\widehat{U}_{\mods{s \wedge \tau_n}}^{0,n} = \widehat{U}_{\mods{s \wedge \tau_n}}^{0,n+1} = \ldots = \widehat{U}_s^0 \quad \text{a.s.}, 
	\end{equation*}
	Furthermore, $\widehat{U}_t$ is a solution to~\cref{eqn:mean-mode-particles} in the sense of~\cite{StatisticsOfLiptser2001}, as for all $0 \leq s \leq t$,
	\begin{multline*}
		\bbP \left( \norm{ \widehat{U}^{0}_s - \widehat{U}^0_0 - \int_0^s a(s, \widehat{U}_s, \hbY_s) \md s - \int_0^s b(s, \widehat{U}_s, \hbY_s) \md [\widehat{\Wcal}_t, \widehat{\Vcal}_t, \widetilde{\Vcal}_t] } > 0\right) 
		\\
		\leq \bbP(s > \tau_n) \xrightarrow{n \to \infty} 0.
	\end{multline*}
	The uniqueness follows from a stopping argument. 
	Consider two distinct solutions $(\widehat{U}_t^{0}, \hbY_t)$ and $((\widehat{U}_t^{0})', \hbY_t')$.
	Since the solution to~\cref{eqn:stoch-modes-particles} is known to be unique, $\hbY_t = \hbY_t'$ almost surely, and so we focus on the sample means. 
	Define $\tau_n'$ the stopping time for the second solution analogously to $\tau_n$.
	It holds that for $0 \leq s \leq \min\{\tau_n, \tau_n'\}$, both solutions verify the same equations as $(\widehat{U}_s^{0,n}, \hbY_s^{n})$, and therefore by uniqueness \mods{of the truncated SDE} $(\widehat{U}_s^{0}, \hbY_s) = (\widehat{U}_s^{0,n}, \hbY_s^n) = ((\widehat{U}_s^{0})', \hbY_s')$ almost surely. 
	Note that $\tau_n' \rightarrow \infty$ as $n \rightarrow \infty$ since $(\widehat{U}_s)'$ is a.s. bounded too; the result follows by letting $n \to \infty$. 
\end{proof}
}

\subsection{Propagation of chaos for fully observed dissipative dynamics} \label{ssec:poc}

\modq{
We now turn to the propagation of chaos result; this section largely follows the structure of the proof in~\cite{OnTheStabilitDelMo2018}. 
A key component of the analysis consists in bounding the fluctuations of the sample mean and sample covariance characterised in~\cref{lem:red-mean-cov}, to which end we will make use of~\cref{lem:fost-lyap}. 
Notably, our framework allows us to bound the fluctuations in the reduced space (for the reduced covariance). 
As a first result, we obtain}
\begin{lemma} \label{lem:dEt-bound}
	\mods{The} norm of the error $\modt{G}_t = \norm{\hP_{\hbY_t} - M_{\bY_t}}^2_F$ satisfies
	\begin{align*} %\label{eqn:dEt-bound}
		\md \modt{G}_t &\leq 2 \lambda_{\max}(A + A^{\top}) \modt{G}_t \md t
		+ \frac{2}{\sqrt{{\rb{N}}-1}} \md \mathcal{E}_t \nonumber
				 \\& \quad+ \frac{2}{{\rb{N}}-1} \left ( \tr\left( (\Sigma_{\bU_t} + \hP_{\hbY_t} S_{\bU_t} \hP_{\hbY_t}) \hP_{\hbY_t}\right) + \tr(\Sigma_{\bU_t} + \hP_{\hbY_t} S_{\bU_t} \hP_{\hbY_t} ) \tr(\hP_{\hbY_t})\right) \md t,
	\end{align*}
	where $\md \mathcal{E}_t = \sum_{\modt{i,j = 1}}^R \mods{(E_t)_{ij}} \md (\Ncal_t)_{ij}$, with \mods{$E_t = \hP_{\hbY_t} - M_{\bY_t}$ and $\Ecal_t$ has} quadratic variation
	\begin{equation} \label{eqn:Ecal-quadr-var}
		\md [\Ecal_t, \Ecal_t ] = 4\tr\left(E_t \hP_{\hbY_t} E_t ( \Sigma_{\bU_t} + \hP_{\hbY_t} S_{\bU_t} \hP_{\hbY_t} ) \right) \modt{\md t}.
	\end{equation}
\end{lemma}
\begin{proof}
	\modt{See~\mods{Appendix} \ref{app:b}.}
\end{proof}
We now prove uniform-in-time convergence of the sample reduced covariance to the true one. 
\begin{proposition} \label{prop:unif-bounds-cov}
	Assume~\ref{eqn:full-obs} and~\ref{eqn:lAsymneg}.
	Then, for $\modr{4R + 1 \leq 2(3n - 1)R + 1 \leq {\rb{N}}}$, it holds
	\begin{equation*}
		\sup_{t \geq 0} \E[\norm{\hP_{\hbY_t}  - M_{\bY_t}}^n_F]^{\frac{1}{n}} < \frac{c_n}{\sqrt{{\rb{N}}}},
	\end{equation*}
	\modt{for some $c_n > 0$ independent of ${\rb{N}}$.}
\end{proposition}
\begin{proof}
	For $t=0$, recall that $\widehat{\bY}_0 = \widehat{\bZ} - \E_{\rb{N}} [\widehat{\bZ}]$, where $\bZ^{(p)} \sim \Ncal(0, M_{\bY_0})$. 
	By~\cite[Lemma 10]{MultilevelEnseCherno2021}, when $n \geq 2$, $\E[\norm{\hP_{\hbY_0} - M_{\bY_0}}_F^{n}]^{\frac{1}{n}} \leq c_n / \sqrt{{\rb{N}}}$, where $c_n$ depends on some higher-order moment of $\bZ$, which are finite since \modt{$\bZ$ is} Gaussian.  For \mods{$n \in [1, 2)$}, it is immediate $\E[\norm{\hP_{\hbY_0} - M_{\bY_0}}^{\mods{n}}_F] \leq \E[\norm{\hP_{\hbY_0} - M_{\bY_0}}_F^{2}]^{\frac{\mods{n}}{2}} \mods{\leq c_{2}^n} / \sqrt{{\rb{N}}}$.
	
	Define $\modt{G}(\hP_{\hbY_t}, M_{\bY_t}) \equiv \modt{G}_t = \norm{\hP_{\hbY_t} - M_{\bY_t}}^2_F$.
	For $t > 0$, the assumptions of~\cref{lem:fost-lyap} \mods{(b)} are satisfied with $\alpha = \frac{2}{3}\lambda_{\max}(A + A^{\top})$, \mods{and} bounding
	\begin{multline*}
		\frac{2}{{\rb{N}}-1} \left ( \tr\left( (\Sigma_{\bU_t} + \hP_{\hbY_t} S_{\bU_t} \hP_{\hbY_t}) \hP_{\hbY_t}\right) + \tr(\Sigma_{\bU_t} + \hP_{\hbY_t} S_{\bU_t} \hP_{\hbY_t} ) \tr(\hP_{\hbY_t})\right) \\
		\leq 
		\frac{4}{{\rb{N}}-1} \left( \tr(\Sigma_{\bU_t}) \tr(\hP_{\hbY_t}) + \rho  \tr(\hP_{\hbY_t})^3 \right) 
		\mods{\eqcolon}
		\beta_t(\hP_{\hbY_t}, M_{\bY_t}).
	\end{multline*}
	\mods{
		For the martingale part, the quadratic variation \mods{of $\Ecal_t$} in~\cref{eqn:Ecal-quadr-var} is upper-bounded by
\begin{equation*} 
	\md [\Ecal_t, \Ecal_t]_t  \leq 4\tr(\hP_{\hbY_t}) \left ( \tr( \Sigma_{\bU_t}) + \rho \tr ( \hP_{\hbY_t} )^2 \right) \norm{E_t}^2_F \modt{\md t},
\end{equation*}
and therefore $\Mcal_t \mods{\coloneqq} \frac{2}{\sqrt{{\rb{N}}-1}} \Ecal_t$ has quadratic variation bounded by
}
	%\begin{equation*}
	%	\md [\Mcal_t, \Mcal_t]_t 
	%	\leq 
	%	\frac{16}{P-1} \tr(\hP_{\hbY_t}) \left ( \tr( \Sigma_{\bU_t}) + \rho \tr ( \hP_{\hbY_t} )^2 \right) \modt{G}_t \modt{\md t}
	%	= 
	%	\modt{G}_t \gamma_t\modt{\md t}
	%\end{equation*}
	\begin{equation*}
		\md [\Mcal_t, \Mcal_t]_t 
		\leq 
		\modt{G}_t \gamma_t\modt{\md t} \quad
		\text{ \mods{with} } \quad
		\mods{\gamma_t = \frac{16}{{\rb{N}}-1} \tr(\hP_{\hbY_t}) \left ( \tr( \Sigma_{\bU_t}) + \rho \tr ( \hP_{\hbY_t} )^2 \right)}.
	\end{equation*}
	\mods{By Hölder's inequality}	
	\begin{equation*}
		\beta_t^n \leq \frac{\mods{8^n}}{\modt{2}({\rb{N}}-1)^n} \left( \tr(\Sigma_{\bU_t})^n \tr(\hP_{\hbY_t})^n + \rho^n \tr(\hP_{\hbY_t})^{3n} \right) ,
	\end{equation*}
	and hence by~\cref{lem:beta-gamma bounds}, for $\modr{2(3n -1)R + 1 < {\rb{N}}}$, it holds $\sup_{t \geq 0} \E[\beta_t^{n}] = \modt{\frac{c^{\beta}_n}{{\rb{N}}^n}} < \infty$. 
	Using exactly the same argument, we also have $\sup_{t \geq 0} \E[\gamma_t^n] = \modt{\frac{c^{\gamma}_n}{{\rb{N}}^n}} < \infty$.
	Now applying~\cref{lem:fost-lyap}, \mods{the $L^n$-norm of $G_t$} is bounded by
	\begin{equation*} 
		\modt{\E[\modt{G}_t^n]^{\frac{1}{n}}} \leq 
		e^{\alpha t} \E[\modt{G^n_0}]^{\frac{1}{n}}
		+ 
		\int_{0}^t e^{\alpha (t-s)} \left[ \E[\beta_s^n]^{\frac{1}{n}} +  \frac{n-1}{2} \E[\gamma_s^n]^{\frac{1}{n}} \right] \md s 
		\leq 
		\frac{c^2_{2n}}{{\rb{N}}}
	\end{equation*}
	which yields the result. 
	%\mods{for $n$ even as $\E[{G}_t^n] = \E[\norm{\hP_{\hbY_t} - M_{\bY_t}}^{2n}_F]$.}
	%For $n$ odd, by Hölder's inequality,
	%\begin{equation*}
	%	\E[\norm{\hP_{\hbY_t} - M_{\bY_t}}^{2n+1}_{\modt{F}}] \leq \left( \E[\modt{G}_t^{n+1}] \right)^{\nicefrac{1}{2}} \left( \E[\modt{G}_t^{n}] \right)^{\nicefrac{1}{2}} \leq  \left( \frac{c_{2n+1}}{\sqrt{P}} \right)^{\modt{2n+1}}
	%\end{equation*}
	%with $c_{2n+1} = (c_{2n+2}^{n+\modt{1}} c_{2n}^{n})^{\nicefrac{1}{2n+1}}$.
\end{proof}
\modt{The constraint on the number of particles ${\rb{N}}$ to get $n$-th order bounds scales with the reduced dimension $R$, and not the state dimension $d$. 
	It arises as a sufficient condition to ensure the boundedness of $\E[\tr(\hP_{\hbY_t})^{3n}]$ in~\cref{lem:beta-gamma bounds}, in essence guaranteeing the stability of (some moments of) the stochastic fluctuations.
	Specifically, the proof involves the reduced covariance dimension ($R$) rather than the full covariance dimension ($d$) as in~\cite{OnTheStabilitDelMo2018}.}

	\rb{As a last preparatory statement, the following moment bound of the \dlr-\kbp\ (whose proof is housed in Appendix \ref{app:b}) will be useful. }
		\begin{lemma} \label{lem:bounded-moments-xdlr}
			Assume~\ref{eqn:full-obs} and~\ref{eqn:lAsymneg} \modt{hold}. 
			Let $\Xcal_t$ \modt{be} the solution to~\cref{eqn:dlr-vkb} with $\Xcal_0$ such that $\E[\norm{\Xcal_0}^n] < \infty$ \mods{for $n \geq 1$}.
			Then, $\Xcal_t$ has all moments bounded uniformly in time, i.e., 
			\begin{equation*}
				\sup_{t \geq 0} \E[\norm{\Xcal_t}^n] < \infty.
			\end{equation*}
		\end{lemma}
		%\begin{proof}
		%	See Appendix~\ref{app:b}.
		%\end{proof}

	We now give the propagation of chaos result. Let $\hXcal_t^{(1)} = \modt{\widehat{U}}_t^0 + \sum_{i=1}^R \bU_t^i \hbY_t^{i,(1)}$ \modt{be} \mods{the first particle of} the solution to~\cref{eqn:mean-mode-particles,eqn:phys-modes-particles,eqn:stoch-modes-particles}, \mods{whose dynamics is driven by the Brownian motions $\{\hWcal_t^{(1)}, \hVcal_t^{(1)}\}$.} 
%	with $\hXcal_0^{(1)} = U^0_0 + \sum_{i=1}^R \bU^i_0 \hbY_0^{i,(1)}$ with Browian motions $\{\hWcal_t^{(1)}, \hVcal_t^{(1)}\}$.
Let $\Xcal_t^{(1)} = U_t^0 + \sum_{i=1}^R \bU_t^i \bY_t^{i}$ \modt{be} the solution to~\cref{eqn:kb-dlr-U0,eqn:kb-dlr-U,eqn:kb-dlr-Y}, with the same Brownian motions, and with $\Xcal_0^{(1)} = \hXcal^{(1)}_0$.
\begin{theorem}[Propagation of Chaos] \label{prop:prop-chaos}
	Under the same assumptions as in\modt{~\Cref{prop:unif-bounds-cov}}, \modr{under the condition} $\modr{4R+1 \leq 2(3n-1)R + 1 < {\rb{N}}}$, it holds
	\begin{equation*}
		\sup_{t \geq 0} \E [\norm{\hXcal_t^{(1)} - \Xcal_t^{(1)}}^n]^{\frac{1}{n}} < \frac{c_n}{\sqrt{{\rb{N}}}}.
	\end{equation*}
\end{theorem}
\begin{proof}
	Denote $\mathcal{D}_t = \hXcal_t^{(1)} - \Xcal_t^{(1)}$, and $\Delta P_t = \widehat{P}_t - P_t$.
	By~\cref{eq:dXhat} and~\cref{eqn:dlr-vkb}, it holds
	\begin{align*}
		\md  \Dcal_t &= A \Dcal_t \md t+ \Delta P_t S \Xcal_t^{\signal} \md t
		- \widehat{P}_t S \hXcal_t^{(1)} \md t + P_t S \Xcal^{(1)} \md t
		+ \Delta P_t H^{\top} \Gmoh \md (\widetilde{\Vcal}_t - \hVcal_t^{(1)} )			    
		\\	
			     &= A \Dcal_t \md t + \Delta P_t (S \Xcal_t^{\signal}
			     - S \Xcal_t^{(1)} ) \md t - \widehat{P}_t S \Dcal_t \md t
		+ \Delta P_t H^{\top} \Gmoh \md (\widetilde{\Vcal}_t - \hVcal_t^{(1)}) 
,
	\end{align*}
	whence it holds
	\begin{multline*}
		\md \norm{\Dcal_t}^2 = 2\langle \Dcal_t, \md \Dcal_t \rangle + \sum_{i=1}^{\modt{d}} \md [\Dcal_t(i), \Dcal_t(i)]_t
				     = \Dcal_t^{\top} (A + A^{\top}) \Dcal_t + 2\Dcal_t^{\top} \Delta P_t  \modt{(S}\Xcal_t^{\signal}  - S \Xcal_t^{(1)}) \md t \\
				     - \Dcal_t^{\top} \widehat{P}_t S \Dcal_t \md t 
				     - \Dcal_t^{\top} S \widehat{P}_t \Dcal_t \md t + 2 \tr(\Delta P_t S \Delta P_t) \md t + \md \Mcal_t ,
	\end{multline*}
%	\begin{align*}
%		\md \norm{\Dcal_t}^2 &= 2\langle \Dcal_t, \md \Dcal_t \rangle + \sum_{i=1}^{\modt{d}} \md [\Dcal_t(i), \Dcal_t(i)]_t
%		\\
%				     &= \Dcal_t^{\top} (A + A^{\top}) \Dcal_t + 2\Dcal_t^{\top} \Delta P_t  \modt{(S}\Xcal_t^{\signal}  - S \Xcal_t^{(1)}) \md t - \Dcal_t^{\top} \widehat{P}_t S \Dcal_t \md t 
%		\\
%				     &\quad \modt{- \Dcal_t^{\top} S \widehat{P}_t \Dcal_t \md t} + 2 \tr(\Delta P_t S \Delta P_t) \md t + \md \Mcal_t ,
%	\end{align*}
	with $\md \Mcal_t = 2\Dcal_t^{\top} \Delta P_t H^{\top} \Gmoh \md (\widetilde{\Vcal}_t - \hVcal_t^{(1)})$.
	Note that
	\begin{multline*}
		\Dcal_t^{\top} (A + A^{\top}) \Dcal_t 
		+ 2 \Dcal_t^{\top} \Delta P_t (S \Xcal_t^{\signal} -  S \Xcal_t^{(1)})  - \Dcal_t^{\top} \modt{ ( \widehat{P}_t S + S \widehat{P}_t) } \Dcal_t  + 2\tr(\Delta P_t S \Delta P_t)  
		\\
		\leq \lambda_{\max}(A + A^{\top}) \norm{\Dcal_t}^2 + 2 \norm{\Dcal_t} \norm{\Delta P_t} ( \norm{S \Xcal_t^{(1)}} + \norm{S \Xcal_t^{\signal}})  + 2 \rho \norm{\Delta P_t}^2_F
		\\
		\leq \frac{1}{2} \lambda_{\max}(A + A^{\top}) \norm{\Dcal_t}^2  + 2\left( \frac{ ( \norm{S \Xcal_t^{(1)}} + \norm{S \Xcal_t^{\signal}})^2}{|\lambda_{\max}(A + A^{\top})|} + \rho \right) \norm{\Delta P_t}^2_F
	\end{multline*}
	having used~\ref{eqn:full-obs}, $- \rho \Dcal_t^{\top} \widehat{P}_t \Dcal_t \leq 0$, and the $\varepsilon$-Young inequality in the last step.
	Furthermore, 
	%\begin{align*}
	%	\modt{\frac{\md}{\md t}} [\Mcal_t, \Mcal_t]_t &= \modt{8}\rho \Dcal_t^{\top} \Delta P_t^2 \Dcal_t \\
	%				 &\leq \modt{8} \rho \norm{\Delta P_t}^2_F \norm{\Dcal_t}^2, 
	%\end{align*}
	\begin{equation*}
		\frac{\md}{\md t} [\Mcal_t, \Mcal_t]_t = \modt{8}\rho \Dcal_t^{\top} \Delta P_t^2 \Dcal_t \leq 8 \rho \norm{\Delta P_t}^2_F \norm{\Dcal_t}^2, 
	\end{equation*}
	hence the assumptions of~\cref{lem:fost-lyap} \mods{(b)} are satisfied with 
	\begin{align*}
		%\tau_t &= \norm{\Delta P_t}_F (\norm{S \hXcal_t^{(1)}} + \norm{S \Xcal_t^{\signal}}) & 
		\alpha &= \frac{\lambda_{\max}(A + A^{\top})}{6} 
		       & \beta_t &= 2\left( \frac{ ( \norm{S \Xcal_t^{(1)}} + \norm{S \Xcal_t^{\signal}})^2}{|\lambda_{\max}(A + A^{\top})|} + \rho \right) \norm{\Delta P_t}^2_F
& \gamma_t &= \norm{\Delta P_t}^2_F.
			\end{align*}
			Since $\norm{\Delta P_t}^2_F = \norm{\hP_{\hbY_t} - M_{\bY_t}}^2_F$, 
			by\modt{~\Cref{prop:unif-bounds-cov}}, $\sup_{t \geq 0}\E[\gamma_t^n]^{\nicefrac{1}{n}} \leq c^{\gamma}_n/\sqrt{{\rb{N}}} < \infty$, \modt{for any $n$ satisfying the constraint in the statement. For that $n$, it also holds}
			\begin{equation*}
				\E[\beta_t^n]^{\nicefrac{1}{n}} \leq  2 \E\left[\left( \frac{ ( \norm{S \Xcal_t^{(1)}} + \norm{S \Xcal_t^{\signal}})^2}{|\lambda_{\max}(A + A^{\top})|} + \rho \right)^{2n}\right]^{\nicefrac{1}{2n}} \E \left[ \norm{\Delta {\rb{N}}_t}^{2n}_F \right]^{\nicefrac{1}{2n}} \leq \frac{c^{\beta}_{2n}}{\sqrt{{\rb{N}}}},
			\end{equation*}
			again using\modt{~\Cref{prop:unif-bounds-cov}}, and bounding the other term using~\cref{lem:bounded-moments-xdlr}.
			This yields the result.
\end{proof}

\section{Numerical Experiments} \label{sec:numerical-experiments}

This section numerically investigates the properties of the \dlr-\kbp\ and \dlr-\enkf\ frameworks established above.
Of particular interest are 
\begin{enumerate}[label=(\roman*)]
	\item the approximation/recovery properties of the \dlr-\kbp~with respect to the Full Order Model (\fom)-\kbp,
	\item the computational efficiency and accuracy of \dlr-\kbp,
	\item verifying the propagation of chaos property,
	\item assessing the computational gains of \dlr-\enkf~over \fom-\enkf, and
	\item for \dlr-\enkf, benchmarking the effect of rank and number of particles to showcase the interest of sampling a higher number of particles while maintaining a moderate rank.
\end{enumerate}
Points (i) - (iv) are explored in the first example with a toy model, while (v) is investigated in the second example \modt{of} an air pollution model.

\subsection{Linear advection}

The first example is inspired by a (stochastically perturbed) finite-difference-type discretisation of a one-dimensional advection-reaction PDE with forcing term, and can be viewed as a continuous analog of the linear advection in~\cite{ADetermSakov2008}.
The noiseless forward model is the solution over $[0,T]$ with $T = 1$ of
\begin{equation*} 
	\partial_t u = - \partial_x u - 0.1 u + 0.03,
\end{equation*}
on $D = [0,L]$ where $L=10$\mods{,} and periodic boundary conditions.
Introducing a finite-differences discretisation with $d = 100$~equispaced points $\{x_i = \frac{(i-1)L}{d}, i = 1, \ldots, d\}$ and adding a noise source yields an SDE like~\cref{eqn:truth-model-dynamics}, where $A$ is the upwind-finite difference operator associated to $- \partial_x - 0.1$, $\modt{f} = 0.03 \texttt{*ones}(d)$, and the noise source $\Sigma = \sigma \bm{I}_d$ for some $\sigma > 0$.
%This technically corresponds to a finite-difference discretisation of an SPDE with space-time white noise~\cite{IntroductionToLord2014}, but we omit the discussion of the SPDE well-posedness and focus on its discretisation, a (well-posed) finite-dimensional SDE.
We furthermore consider the case of full observations, hence $H = \bm{I}_d$ and $k = d$; finally, let $\Gamma = \gamma \bm{I}_d$ \modt{with} $\gamma = 2$. 
This setting allows to showcase the validity of the theoretical predictions of this work. 
The initial condition is a Gaussian, generated by $u_0 = \sin(\frac{2 \pi x}{L}) + \sum_{k=1}^{R_0} \frac{1}{i} \modt{\sin(\frac{2 \pi x k}{L})} \xi_k = U_0^0 + \bU_0 \bY_0^{\top}$ and $\xi_k \overset{i.i.d}{\sim} \Ncal(0,1)$, \modt{with} initial rank $R_{0} = 25$.
When not specified otherwise, an Euler-Maruyama scheme is used for the time-discretisation of SDEs and an explicit Euler scheme for ODEs, with $\D t = 10^{-4}$.

Firstly, we focus on the approximation properties of \dlr-\kbp; as stated in~\cref{ssub:error-approx}, \modt{these} strongly depend \ra{on the size of the orthogonal defect \rf{\eqref{eqn:mea-sigma}}}. 
As that property is hard to naturally enforce in the SDE setting, we instead tune $\sigma$ to values in~$[0, 10^{-3}, 10^{-1}, 0.5]$ and compare the means and covariances of the \fom-\kbp~to those of \dlr-\kbp~for varying $\sigma$. 
For a fixed rank $R = 15$, \cref{fig:mean-covs-errs} displays the evolution of errors between the \dlr-\kbp~and \fom-\kbp~means\mods{,} resp. \modt{covariances}. 
The plots show that the approximation properties of \dlr-\kbp~are negatively impacted by an increasing $\sigma$, but remain an interesting approximation tool for small $\sigma$. 

\begin{figure}[ht]
\centering
\begin{subfigure}{.5\textwidth}
  \centering
  \includegraphics[width=\linewidth]{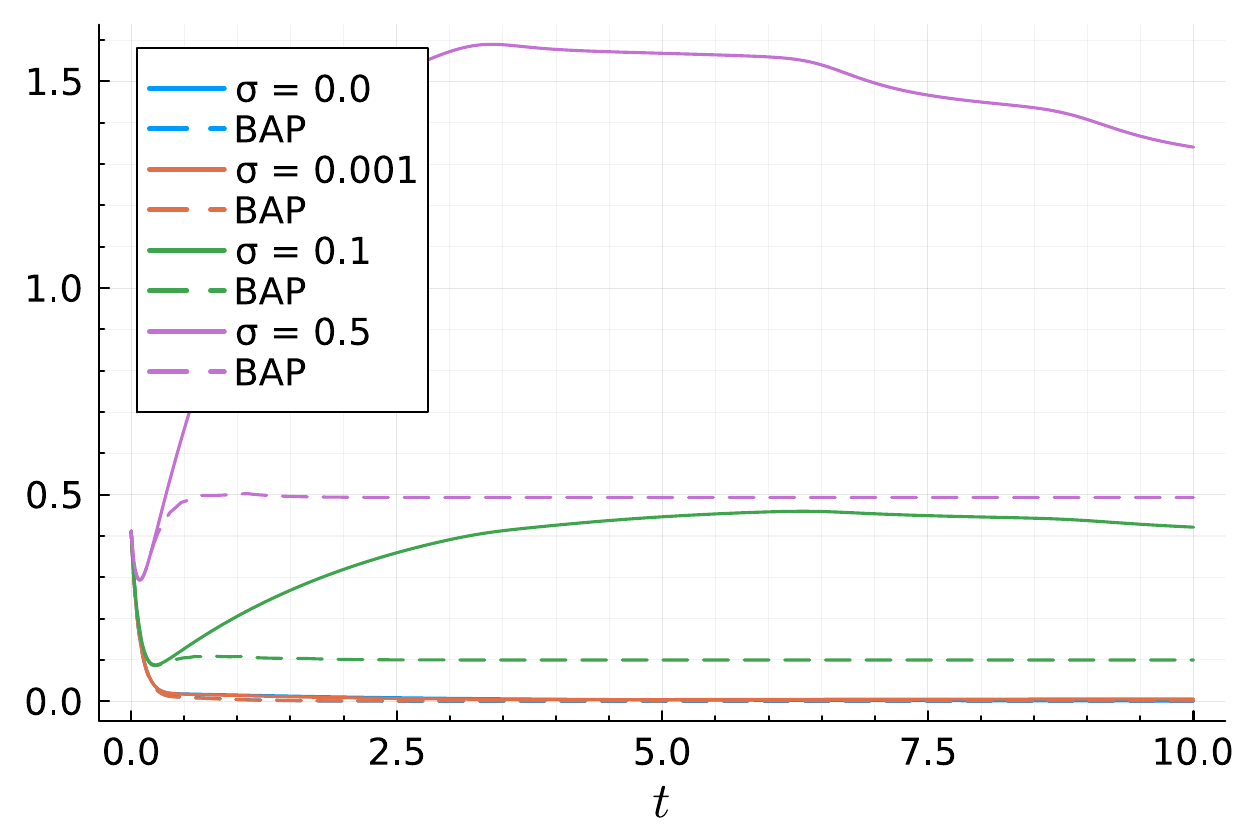}
  \caption{Errors $\norm{P_t^{\dlr} - P_t^{\fom}}_F$}
  \label{fig:covariance-errs}
\end{subfigure}%
\begin{subfigure}{.5\textwidth}
  \centering
  \includegraphics[width=\linewidth]{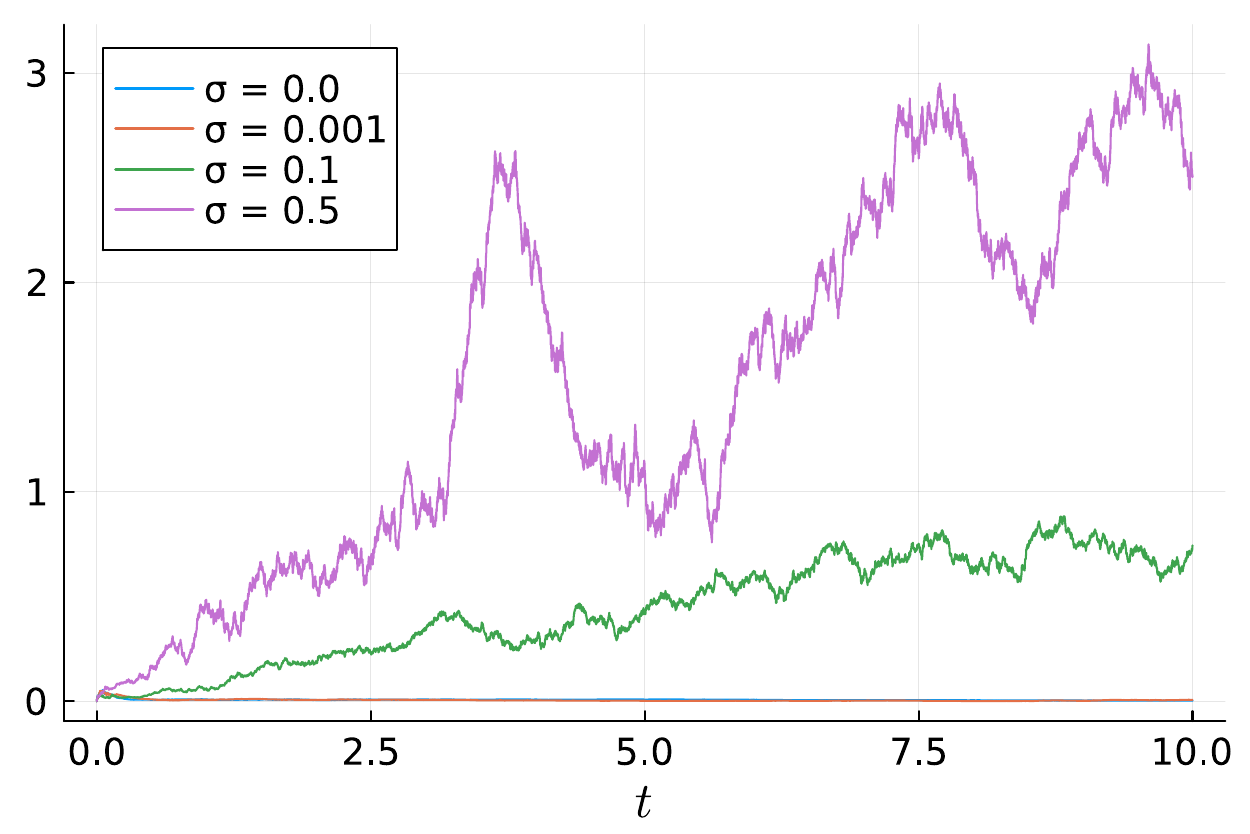}
  \caption{Errors $\norm{m_t^{\dlr} - m_t^{\fom}}$}
  \label{fig:mean-errs}
\end{subfigure}
\caption{Mean and covariance errors between \fom-\kbp~and \dlr-\kbp.
	The \dlr-\kbp~rank is $R = 15$.
	BAP = Best Approximation Possible = $\norm{ P_t^{\fom} - \Tcal_R (P_t^{\fom})}_F $.}
\label{fig:mean-covs-errs}
\end{figure}
Informed by \modt{these} first results, we hereafter fix $\sigma = 10^{-3}$ for the remaining computational experiments.
\modt{The performance of \dlr-\kbp~is assessed by comparing its mean and covariance with \modt{those of the} \fom-\kbp}
%\begin{align*}
%	\mathrm{CM}_t &= \sqrt{\norm{P_t^{\dlr} - P_t^{\fom}}^2 + \norm{m_t^{\dlr} - m_t^{\fom}}^2},
%\end{align*}
as well as directly measuring how it tracks the true signal via the (integrated) RMSE 
\begin{align*}
	\mathrm{RMSE}(m_t, P_t) &= \sqrt{ \norm{ m_{t} - \Xcal_{t}^{\signal}}^2 + \tr(P_t) },
				&  
	\mathrm{iRMSE}(m_t, P_t) &= \frac{\D t}{T} \sum_{i=1}^N \mathrm{RMSE}(m_{t_i}, P_{t_i}).
\end{align*}
Note that this definition of the RMSE also takes into account the ``uncertainty'' of the error via the covariance. 
\modt{The former metrics are shown in the two plots of~\cref{fig:kbp-v-dlr-kbp}.  
Both the errors in the mean and covariance of \dlr-\kbp~with respect to their \kbp~counterparts decrease with increasing rank.}
\modt{In turn, the plots in~\cref{fig:rmse-plots} display the signal-tracking abilities of \dlr-\kbp~for varying ranks.
\Cref{fig:rmses-ranks} showcases how, for a very small rank $R=2$, \dlr-\kbp~displays a poorer accuracy than the \fom~for the RMSE metric, but this discrepancy vanishes when using a higher rank.} 
This plot also verifies that using the (\dlr-)\kbp~process to track the signal is beneficial when compared to simply evolving the mean and covariance of the model without observations, which results in a \mods{much} higher RMSE \modt{(the dashed black line in~\cref{fig:rmses-ranks})}; \modt{for reference, we also include the norm of the signal (magenta dashed line in~\cref{fig:rmses-ranks}), confirming that the (\dlr-)\kbp~solution yield better approximations of the signal.
The eventual decrease of the signal norm and the plotted errors in time is in line with the dissipative nature of the dynamics.}
This is also verified using the iRMSE metric in~\cref{fig:ranks-accuracy-vs-compt}; it is apparent that the error in the mean rapidly reaches the same levels as those of the~\fom~for increasing rank, while retaining a computational advantage \modt{(the errors and simulation times were averaged over $100$ runs)}.

\begin{figure}[ht]
\centering
\begin{subfigure}{.49\textwidth}
  \centering
  \includegraphics[width=\linewidth]{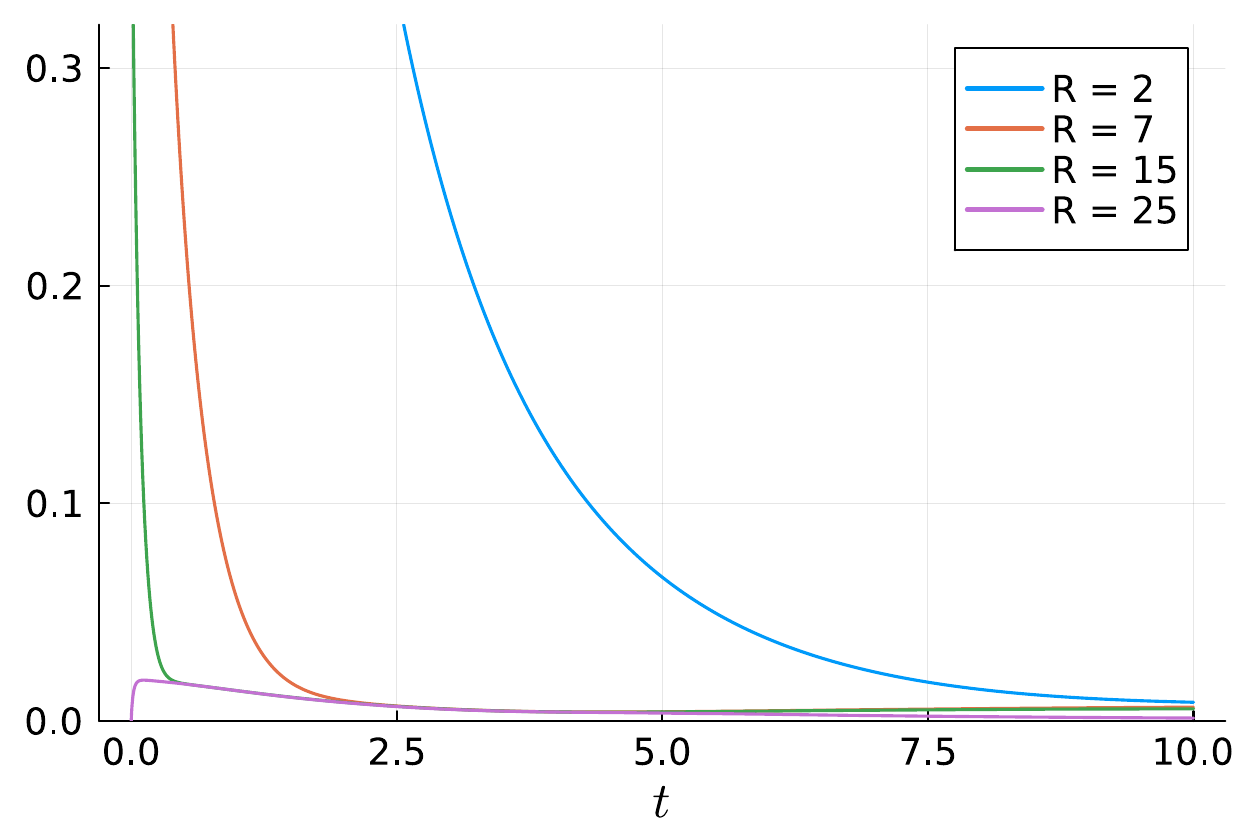}
  \caption{Errors $\norm{P_t^{\dlr} - P_t^{\fom}}_F$}
  \label{fig:cov-errs-varR}
\end{subfigure}
\begin{subfigure}{.49\linewidth}
  \centering
  \includegraphics[width=\linewidth]{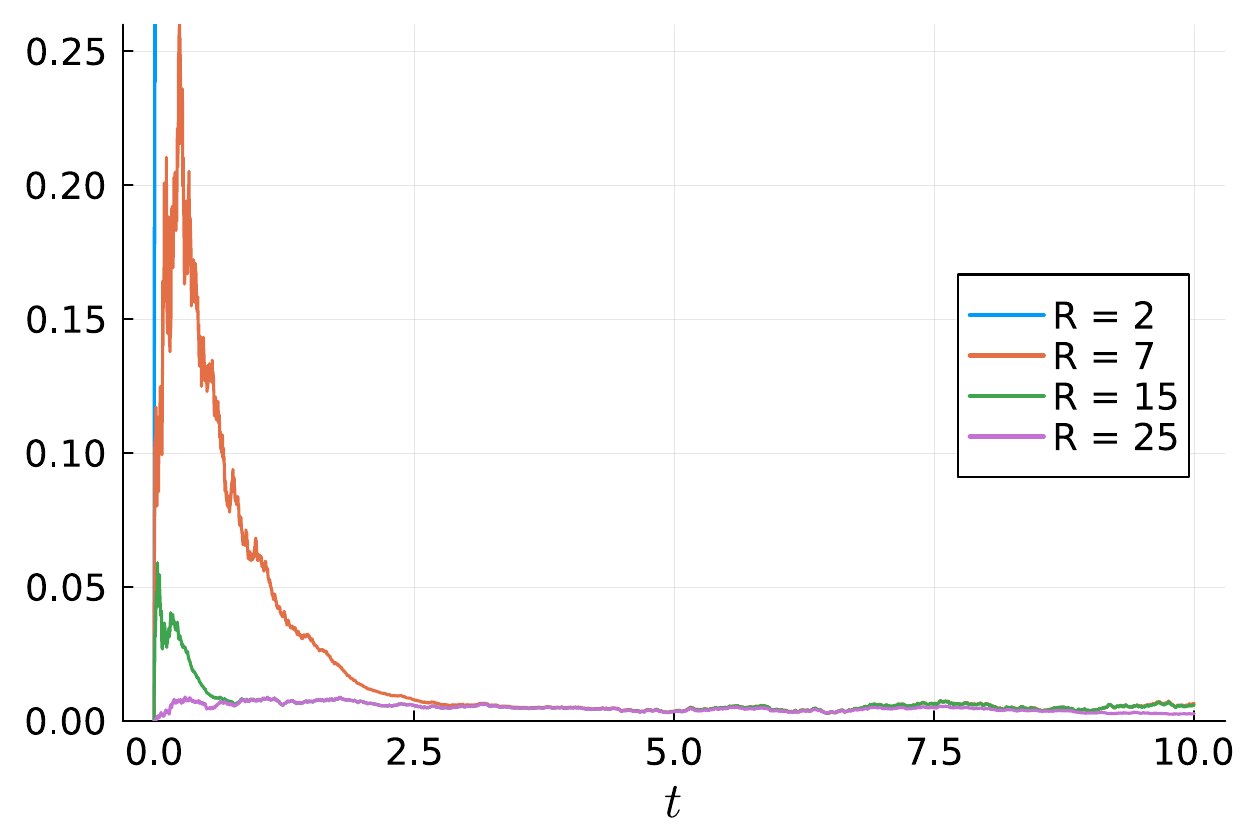}
  \caption{Errors $\norm{m_t^{\dlr} - m_t^{\fom}}$}
  \label{fig:mean-errs-varR}
\end{subfigure}%
\caption{\dlr-\kbp~approximation properties of \kbp~for varying $R$.}
\label{fig:kbp-v-dlr-kbp}
\end{figure}

% Varying ranks RMSE metrics 
\begin{figure}[ht]
\centering
\begin{subfigure}{.5\textwidth}
  \centering
  \includegraphics[width=\linewidth]{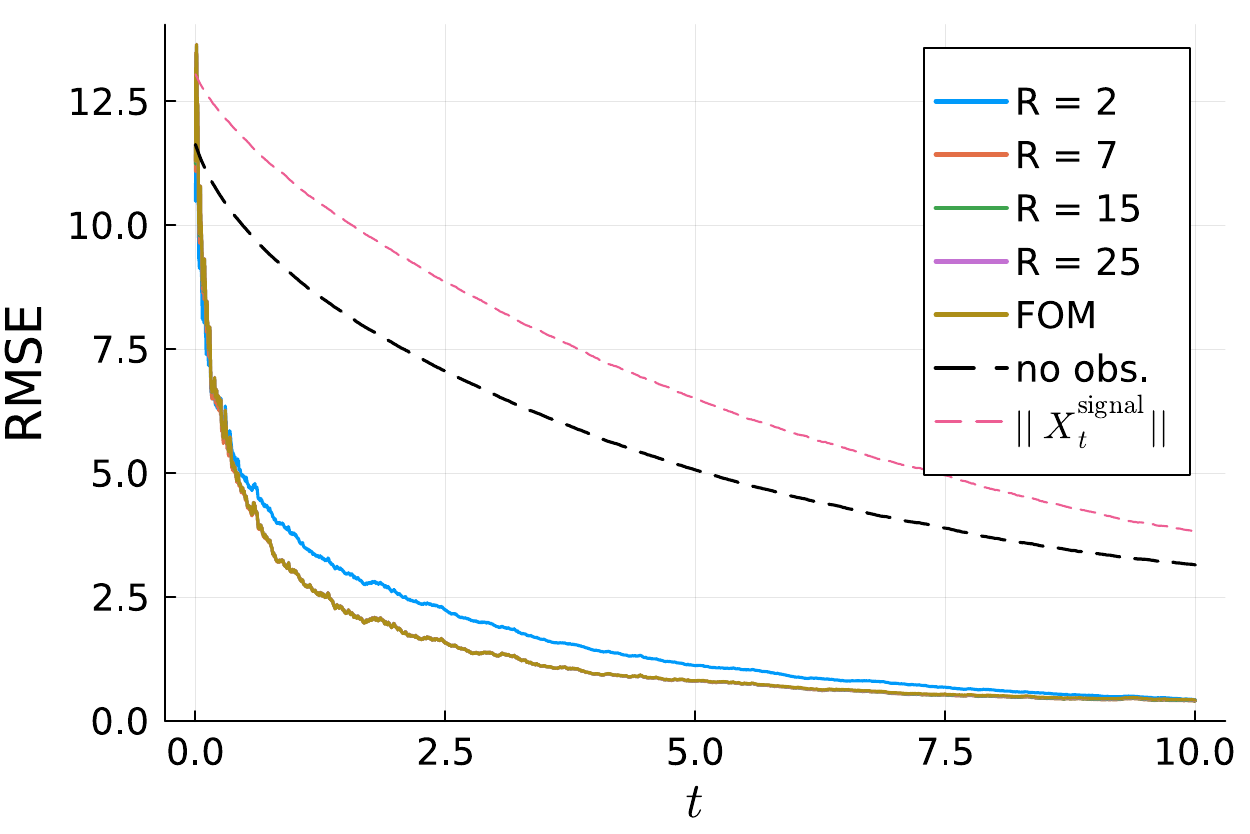}
  \caption{(\dlr)-\kbp~RMSEs}
  \label{fig:rmses-ranks}
\end{subfigure}%
\begin{subfigure}{.5\textwidth}
  \centering
  \includegraphics[width=\linewidth]{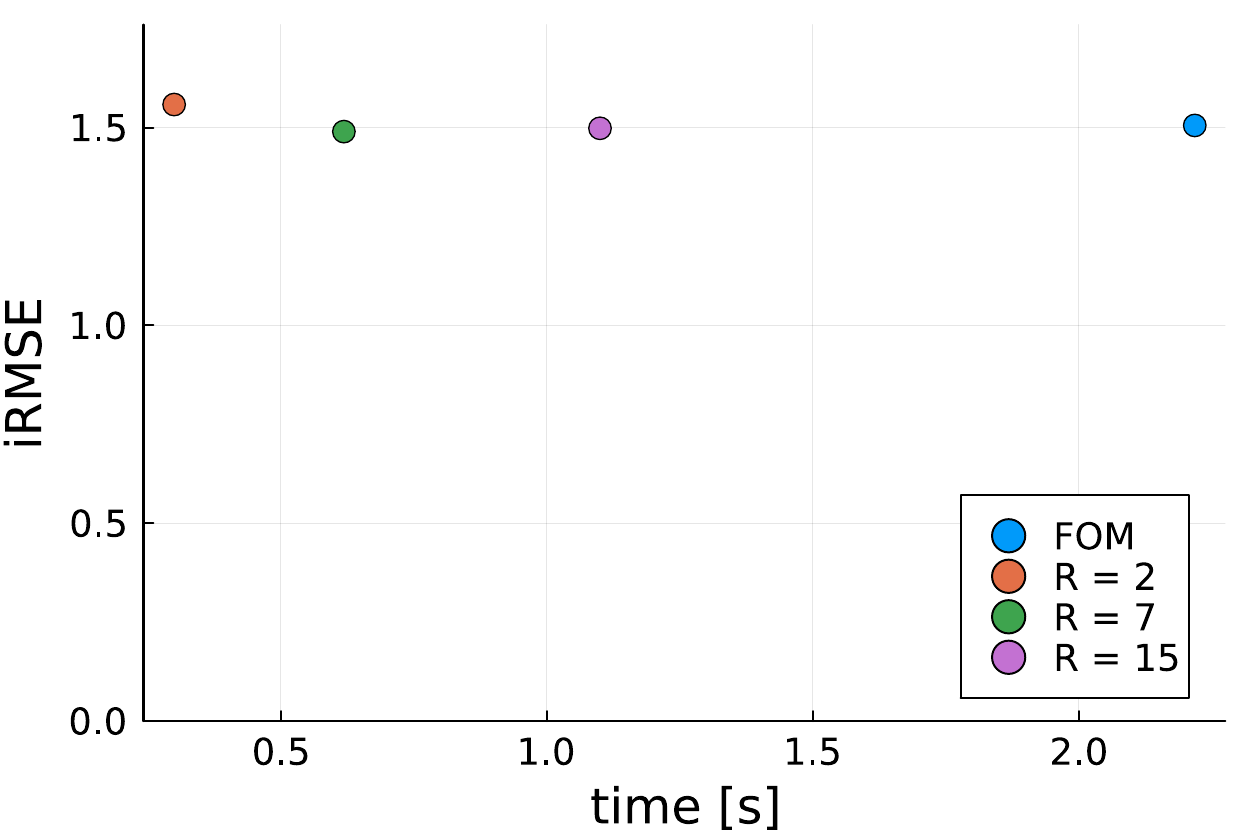}
  \caption{Integrated RMSEs v. computing cost [s]}
%  (averaged over $100$ runs).
  \label{fig:ranks-accuracy-vs-compt}
\end{subfigure}%
\caption{Signal tracking abilities of \dlr-\kbp~for varying $R$ via RMSE metrics.}
\label{fig:rmse-plots}
\end{figure}

Finally, the propagation of chaos result from~\cref{ssec:poc} is verified in~\cref{fig:poc-all}.
To this end, \rb{$N$ i.i.d copies $\Xcal_0^{(i)}$, $i=1, \ldots, N$,} are sampled from the original distribution (now with $R = R_{0} = 7$) and decomposed as $\widehat{\Xcal}_0 = \E_{\rb{N}} [\widehat{\Xcal}_0] + \bU_0 \widehat{\bY}_0^{\top}$. 
The \dlr-\enkf~particle system is simulated up to $T = 1$ and at that time we compare the sample reduced covariance matrix, sample mean \modt{(left)} and a chosen particle to their ``true'' \dlr-\kbp~counterparts \modt{(right)}. 
Those simulations are repeated $15$ times to approximate the expectations in{\modt{~\Cref{prop:unif-bounds-cov,prop:prop-chaos}} with $n = 2$. 
As can be seen in~\cref{fig:poc-gram+mean,fig:poc-particles}, the error has the expected $\mathcal{O}({\rb{N}}^{-\nicefrac{1}{2}})$ decay. 
\modt{Although not directly predicted by \mods{our} theory, we also verify that the error of the mean has the same error decay in~\cref{fig:poc-gram+mean}.}

\begin{figure}[ht]
\centering
\begin{subfigure}{.5\textwidth}
  \centering
  \includegraphics[width=\linewidth]{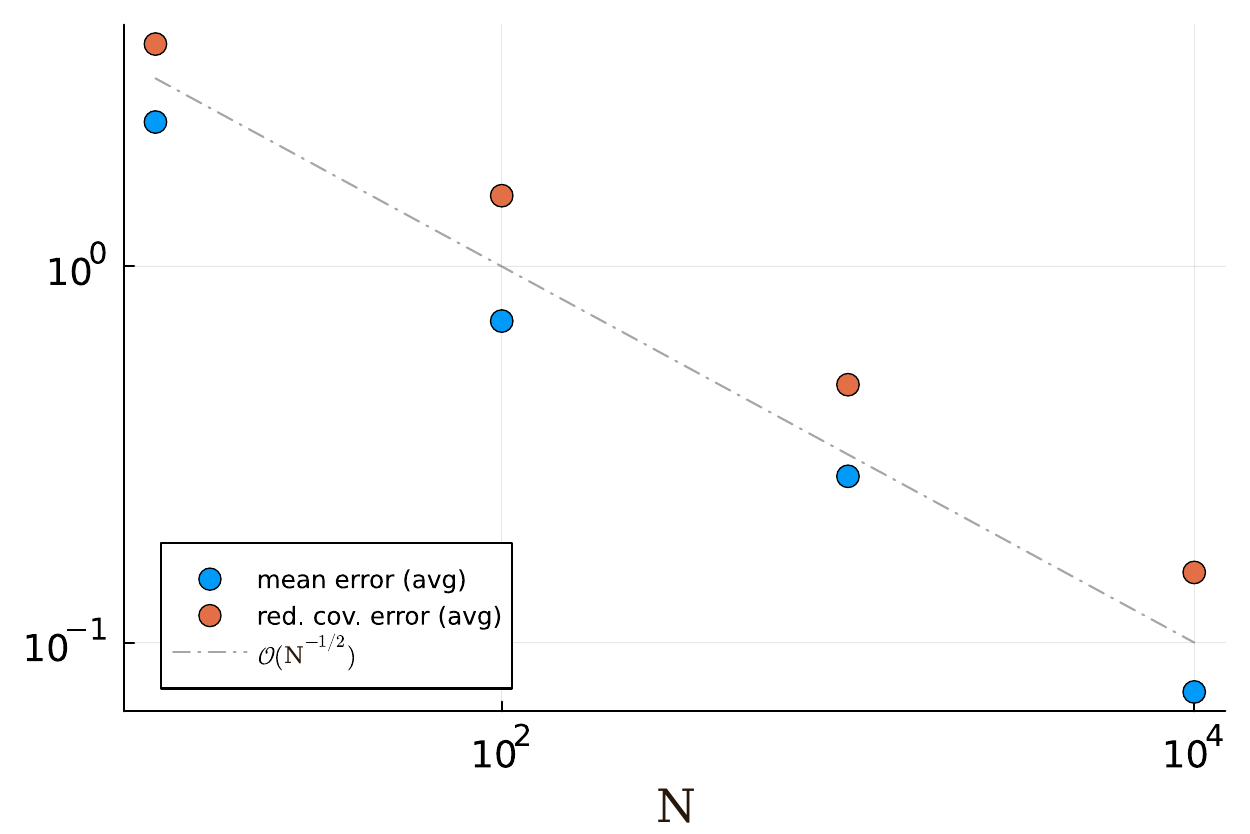}
  \caption{Error reduced covariance + mean}
  \label{fig:poc-gram+mean}
\end{subfigure}%
\begin{subfigure}{.5\textwidth}
  \centering
  \includegraphics[width=\linewidth]{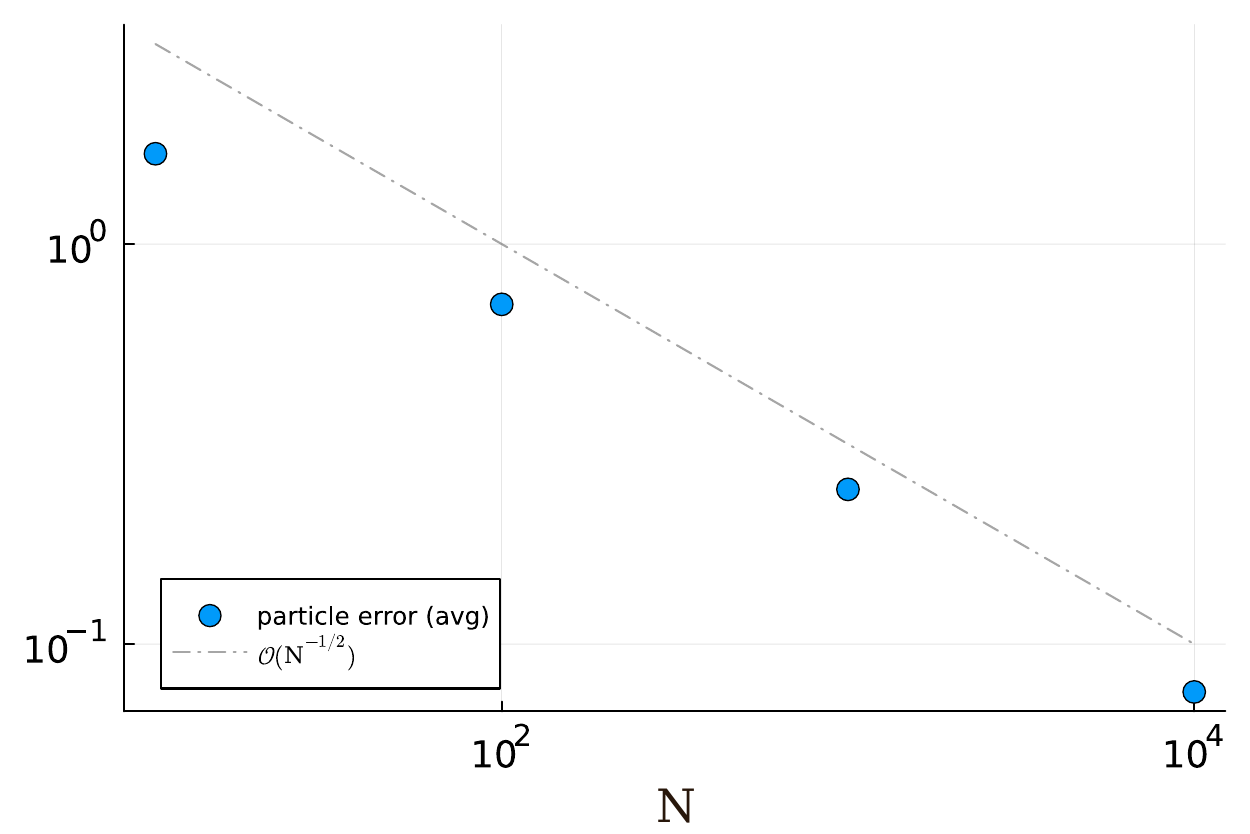}
  \caption{Error particle (propagation of chaos)}
  \label{fig:poc-particles}
\end{subfigure}
\caption{\dlr-\enkf~approximation error of \dlr-\kbp~with ${\rb{N}}$ particles}
\label{fig:poc-all}
\end{figure}

\subsection{Air pollution model}

This section investigates the performance of \dlr-\enkf~with respect to \fom-\enkf, specifically the relation between the number of particles, RMSE and rank. 
The model is given by the Finite Element discretisation of a $2D$ stochastic advection-diffusion model over the time-interval $[0,1]$ and either full or partial observations.
It can be viewed as a simple model for the passive transport of a pollutant. 
On the domain $D = [0,1]^2$, let $V = H^1_{\mathrm{per}}(D) = \{u(x_1,x_2) \in H^1(D) \st u(0,x_2) = u(1,x_2), x_2 \in [0,1]\}$ and $H = L^2(D)$. 
We consider the signal 
\begin{equation*}
	d \mods{u_t^{\signal}} = (a \Delta \mods{u_t^{\signal}} + \bm{b} \cdot \nabla \mods{u_t^{\signal}}) \md t + \sigma^{\nicefrac{1}{2}} \md \widetilde{W}_t,
\end{equation*}
with $0$-Neumann boundary conditions on $\{(x_1,0), (x_1,1) \st x_1 \in [0,1]\}$ and periodic boundary conditions \modt{on} $(0,x_2) = (1,x_2)$ for $x_2 \in [0,1]$ (\modt{$\mods{u_t^{\signal}}(0\modr{,} x_2) = \mods{u_t^{\signal}}(1, x_2)$ and $\partial_{x_1} \mods{u_t^{\signal}}(0, x_2) = \partial_{x_1} \mods{u_t^{\signal}}(1, x_2)$}), 
the coefficients are $a = 10^{-1}, \bm{b} = (1, 0)^{\top}, \sigma = 10^{-5}$ and $\widetilde{W}_t$ is a finite-dimensional Q-Wiener process that will be specified below.
The initial condition has distribution 
\begin{equation*}
	u_0(x) = \exp(- (x_1 - 0.5)^2 - (x_2 - 0.5)^2 ) + \sum_{i=1}^{R_{0}} \frac{1}{i^2} \sin(i \pi x_1) \cos(i \pi x_2) \xi_i, \quad \xi_i \overset{iid}{\sim} \Ncal(0,1)
\end{equation*}
where $R_{0} = 12$. 
This SPDE admits a solution in the sense of~\cite[Theorem 4.2.4]{StochasticPartLiuW2015}, i.e., a continuous $H$-valued process whose equivalence class admits a $V$-valued version -- see~\cite[Definition 4.2.1]{IntroductionToLord2014}. 

For the discretisation, let $\mathcal{T}_h$ \mods{be} a \modt{quadrangular} mesh of $D$ with nodes at $21$ equispaced points along both dimensions.
The Galerkin approximation space is $V_h = \mathbb{P}_1^C(\mathcal{T}_h)$ \modt{the space of piecewise continuous bilinear polynomials over the mesh $\Tcal_h$} with $\mathrm{dim}(V_h) = 420 = d$, with \modt{Lagrangian} basis $\{\phi_i\}_{i=1}^{d}$.
For simplicity, the Q-Wiener process is $\widetilde{W}(t,x) = \sum_{i=1}^{d} \widetilde{\Wcal}^i_t \phi_i(x)$, where $\widetilde{\Wcal}^i_t$ are independent standard $1$-dimensional Brownian motions.
The corresponding assembled system is then 
\begin{equation*}
	M \md \Xcal_t^{\signal} = A \Xcal_t^{\signal} \md t + \sigma^{\nicefrac{1}{2}} M \md \widetilde{\Wcal}_t,
\end{equation*}
where $\Xcal_t^{\signal}$ \modt{is the vector of nodal} degrees of freedom, $A$ contains the assembled diffusion and advection terms, $M$ is the mass matrix and $\widetilde{\Wcal}_t = [\widetilde{\Wcal}_t^1, \ldots, \widetilde{\Wcal}_t^d]^{\top}$.
A semi-implicit Euler-Maruyama scheme~\cite{IntroductionToLord2014} is used for time-stepping, i.e. $(M - \Delta t A)\Xcal_{n+1}^{\signal} = \sigma^{\nicefrac{1}{2}} M \D \widetilde{\Wcal}_{n}$, where $\Delta t = 10^{-2}$.
%Then, consider the (discretised) stochastic heat equation in weak form 
%\begin{multline*}
%	(\md \Xcal^{\signal}_{t,h}, v_h)_{L^2(D)} = (a \nabla \Xcal^{\signal}_{t,h}, \nabla v_h)_{L^2(D)}  \md t + (\bm{b} \cdot \nabla \Xcal^{\signal}_{t,h}, v_h)_{L^2(D)}  \md t 
%	\\
%	+ (c \md \Wcal_{h}(t), v_h)_{L^2(D)}  \quad \forall v_h \in V_h.
%\end{multline*}
Two (bounded linear) observation operators are considered, 
\begin{align*}
	h^{\mathrm{full}} : V \to V &&& h^{\mathrm{part}} : V \to \R^k,
	\\
	h^{\mathrm{full}}(\mods{u}) = \mods{u} &&&  h^{\mathrm{part}}(\mods{u})(\ell) = \int_{D} \mods{u} \chi_{\ell}, \quad 1 \leq \ell \leq k
\end{align*}
where $k = 25$ and $\chi_{\ell}$ is the indicator function of a prescribed square on the domain (see~\cref{fig:indicators}).
%\begin{align*}
%	\Hcal^{\mathrm{full}}(\Xcal_{t,h}) &= \Xcal_t \in \R^d, && \Hcal^{\mathrm{part}} : V_h \to \R^k, 
%	\\
%					   &&&  \Hcal^{\mathrm{part}}(\Xcal_{t,h})(\ell) = \int_{D} \Xcal_{t,h} \chi_{\ell}, \quad 1 \leq \ell \leq k = 25
%\end{align*}
%and where $\chi_{\ell}$ is the indicator function of a prescribed square on the domain (see~\cref{fig:indicators}).
 \begin{figure}[ht]
	 \centering
	 \includegraphics[width=.25\textwidth]{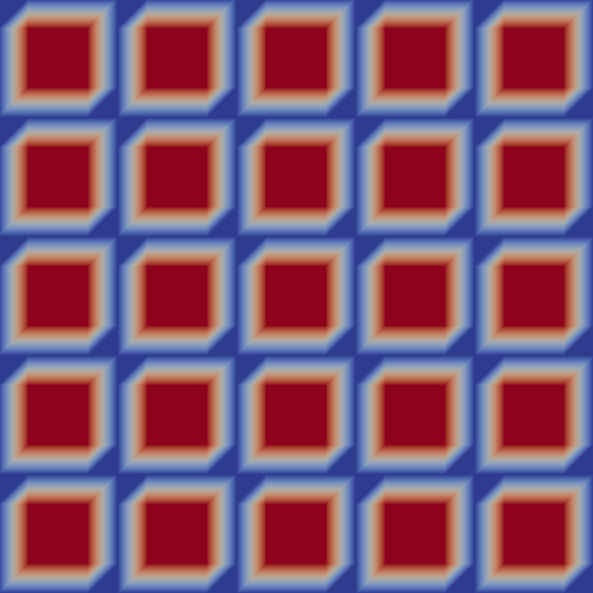}
	 \caption{Each indicator function $\chi_\ell$ is associated to one of the red squares.}
	 \label{fig:indicators}
\end{figure}
Denoting the observation (Hilbert) space by $K$ and letting $\Gamma = \gamma I \in \Lcal(K,K)$ with $\gamma = 10^{-2}$, the observation process is given by
\begin{equation*}
	\md z_t = h^{\star}(\mods{u_t^{\\signal}}) \md t + \gamma^{\nicefrac{1}{2}} \md \widetilde{V}_t^{\star},
\end{equation*}
where if $\star = \mathrm{full}$, $\widetilde{V}_t$ is a Q-Brownian motion on the finite element basis (like $\widetilde{W}_t$), and if $\star = \mathrm{part}$, then the observation process is a finite-dimensional SDE and $\widetilde{V}_t = \widetilde{\Vcal}_t$ a $25$-dimensional Brownian motion.

Note that for either observation operator $h : H \to K$, its transpose $h^{\top} : K \to H$ verifies $\langle h x, y \rangle_K = \langle x, h^{\top} y\rangle_H$ for $x \in H$ and $y \in K$.

For ${\rb{N}}$ particles $ \{ \mods{u}^{(p)}_t \}_{p=1}^{\rb{N}}$, denote the sample mean $\bar{\mods{u}}_t = {\rb{N}}^{-1} \sum_{p=1}^{\rb{N}} \mods{u}_t^{(p)}$ and the sample covariance $\rb{\widehat{P}_t}(\rb{u_t})$ by $\rb{\widehat{P}_t(u_t)} v = ({\rb{N}}-1)^{-1}\sum_{i=1}^{\rb{N}} e_t^{(p)} \langle  e_t^{(p)}, v \rangle_H$ for $v \in H$, where $e_t^{(p)} = \mods{u}_t^{(p)} - \bar{\mods{u}}_t$.
Similarly to~\cite{WellPosedKelly2014}, we consider the Ensemble Kalman system with ${\rb{N}}$ particles
\begin{equation*}
	d \mods{u}_t^{(p)} = (a \Delta \mods{u}_t^{(p)} + \bm{b} \cdot \nabla \mods{u}_t^{(p)}) \md t + \sigma^{\nicefrac{1}{2}} \md \widehat{W}_t^{(p)} + \frac{1}{\gamma} \widehat{P}(\mods{u}) h^{\top}  \left[ \md z_t - h(\mods{u}_t^{(p)}) \md t - \gamma^{\nicefrac{1}{2}} \md \widehat{V}_t^{(p)} \right]
\end{equation*}
The FEM discretisation in the full observations case therefore yields the following SDE in the degrees of freedom
\begin{equation*}
	M \md \Xcal^{(p)}_t = A \Xcal_t^{(p)} \md t + \sigma^{\nicefrac{1}{2}} M \md \widehat{\Wcal}^{(p)}_t + \frac{1}{\gamma} M \widehat{P}_t M \left[ \md \Zcal_t - (\Xcal_t^{(p)} \md t + \gamma^{\nicefrac{1}{2}}  \md \widehat{\Vcal}^{(p)}_t) \right],
\end{equation*}
whereas in the case of partial observations,
\begin{equation*}
	M \md \Xcal^{(p)}_t = A \Xcal_t^{(p)} \md t + \sigma^{\nicefrac{1}{2}} M \md \widehat{\Wcal}^{(p)}_t + \frac{1}{\gamma} M \widehat{P}_t H^{\top}  \left[ \md \Zcal_t - (H \Xcal_t^{(p)} \md t + \gamma^{\nicefrac{1}{2}}  \md \widehat{\Vcal}^{(p)}_t) \right],
\end{equation*}
where $H \Xcal_t = h(\mods{u}_{t})$ for $\mods{u}_{t} \in V_h$ and $\Xcal_t$ are the degrees of freedom of $\mods{u}_{t}$.
%To obtain the second system, note that $\widehat{P}(x) \left( h^{\top}y \right) = ({\rb{N}}-1)^{-1}\sum_{i=1}^{\rb{N}} e_t^{(p)} \langle  h(e_t^{(p)}), y \rangle_{\R^k} = ({\rb{N}}-1)^{-1}\sum_{i=1}^{\rb{N}} e_t^{(p)} \mathrm{dof}(e_t^{(p)}) H^{\top} y$.
The FEM discretisations are again solved with the semi-implicit scheme $(M - \D t A) \Xcal_{n+1} = \mathrm{rhs}$.

For the time-integration of \dlr-\enkf, we propose the following BUG-like~\cite{ARankAdaptiveCeruti2022,AnUnconventionCeruti2022} time-integration scheme. 
It approximates the physical space at time $t_{n+1}$, creates an augmented physical basis, and projects the original dynamics onto that augmented space. 
This yields an update equation for the (augmented) stochastic modes, which are then truncated back to rank $R$.
Note that while this scheme shares strong similarities with schemes of the BUG class, it is \emph{not} properly speaking a BUG scheme as it does not compute the update of the second base separately from the update of the singular values. 
Specifically because of that, the scheme is well-suited to the DO framework we \modt{have used in this work}.
Below we detail the equations for the case of partial observations, those for full observations being analogous.
\begin{enumerate}
        \item Mean update
                \begin{multline*}
			(M - \D t A) U_{n+1}^0 =  M U_n^0 
			+ \sigma^{\nicefrac{1}{2}} M \Pi_{\bU_n}  \E_{\rb{N}} [\D \hWcal_n]  
                        \\
			+ M \widehat{P}_t H^{\top} \Gamma^{-1}  
			\left[ \D \Zcal_n 
                        - H U^0_n \D t
                        + \Goh \E_{\rb{N}} [ \D \widehat{\Vcal}_n] \right]
                \end{multline*}
	%\item Orthonormalise $\overline{\bY}_n R_1 = \texttt{qr}(\hbY_n)$ such that $\E_{\rb{N}}[\overline{\bY}_n^i \overline{\bY}_n^j] = \delta_{ij}$.
	\item Physical modes update. 
                \begin{equation*}
			(M - \D t A)\widetilde{\bU}_{n+1} = M \bU_n - M \widehat{P}_t S \bU_n \D t  %+ M \Pi_{\bU_n} \Soh  \E_P [\D \hWcal_n^{\star} {\overline{\bY}}_n]  \\
					%+      M \widehat{P}_t H^{\top} \Gmoh \E_P [ \D \widehat{\Vcal}^{\star}_n {\overline{\bY}}_n] 
                \end{equation*}
	\item Orthonormalise $\overline{\bU}_{n+1} = \texttt{qr}([\bU_n, \widetilde{\bU}_{n+1}])$ such that $\overline{\bU}_{n+1}^{\top} M \overline{\bU}_{n+1} = \bm{I}_{2R}$.
        \item Stochastic mode update. 
		Define $\widetilde{\bY}_{n}^{\top} = \overline{\bU}_{n+1}^{\top} M \bU_n \hbY_n^{\top}$ and solve
                \begin{multline*}
			(I - \D t \overline{\bU}_{n+1}^{\top} A \overline{\bU}_{n+1})\widetilde{\bY}_{n+1}^{\top} = 
                        \widetilde{\bY}_{n}^{\top} 
			- 
                        \overline{\bU}_{n+1}^{\top} M \widehat{P}_n S \Xcal_n^{\star} \D t 
                        + \overline{\bU}_{n+1}^{\top} \Soh \D\Wcal_n^{\star}
                        \\
                        - \overline{\bU}_{n+1}^{\top} M \widehat{P}_{n} H^{\top} \Gmoh \D \Vcal_n^{\star}.
                \end{multline*}
        \item Truncation
                \begin{align*}
                        \mathcal{T}_R(\widetilde{\bY}_{n+1}) = U_R S_R V_R^{\top} && \bU_{n+1} = \overline{\bU}_{n+1} U_R && \hbY_{n+1} = V_R S_R.
                \end{align*}
		(where $\E_{\rb{N}}[V_R^{\top} V_R] = \bm{I}_{2R}$.)
\end{enumerate}

To verify \mods{that} the integrator is consistent, we consider the noiseless forward model case $\sigma = 0$ (and full observations). 
In that setting, the \fom-\enkf~evolves a rank-$R_{0}$ solution at all times. 
We set ${\rb{N}} = 425$, hence ${\rb{N}} > \mathrm{dim}(V_h)$, making it possible (in the general case) for \enkf(${\rb{N}}$) to have a full-rank sample covariance. 
Given an initial condition, we simulate \fom-\enkf(${\rb{N}}$) and, truncating that initial condition for different $R$'s, we simulate the corresponding \dlr-\enkf($R,{\rb{N}}$).
It is verified in \cref{fig:dlr-enkf-l2errs} that the integrator has the expected approximation properties. 
\modt{Using the norm $\trnorm{\Xcal_t}_H \mods{\coloneq} \norm{\mods{u}_t}_H$, we consider the error metric}
\begin{equation*}
	\modt{\trnorm{\hXcal_t^{\fom-\enkf} - \hXcal_t^{\dlr-\enkf}}^2_{\ell^2_{\rb{N}}(H)} = \frac{1}{{\rb{N}}} \sum_{p=1}^{\rb{N}} \trnorm{ (\hXcal_t^{\fom-\enkf})^{(p)} - (\hXcal_t^{\dlr-\enkf})^{(p)}}^2_{H}}\mods{.}
\end{equation*}
\mods{The} errors of \dlr-\enkf~ are compared to the ``best \modt{rank} $R$ approx'' 
\begin{equation*}
	\modt{\trnorm{\hXcal_t^{\fom-\enkf} - \E_{\rb{N}}[\hXcal_t^{\fom-\enkf}] - \Tcal_R[(\hXcal_t^{\fom-\enkf})^{\star}]}_{\modt{\ell^2_{\rb{N}}(H)}}}
\end{equation*}
(this is a slight abuse of denomination, as it technically corresponds to the best rank-$R$ approximation of the zero-sample mean part of the \fom-\enkf~and is suboptimal with respect to the ``true best rank-$R$ approximation'').
Note that \modt{it} is expected that the \dlr-\enkf~ensembles eventually deviate from the best $R$ approx, as the dynamics of the mean are dependent on the sample covariance, which is never exactly recovered for $R < R_{0}$ by \dlr-\enkf.
It is also verified that for $R = R_0$, \dlr-\enkf~recovers \fom-\enkf~(not shown in the figure).

\begin{figure}[ht]
	\centering
	\includegraphics[width=.5\linewidth]{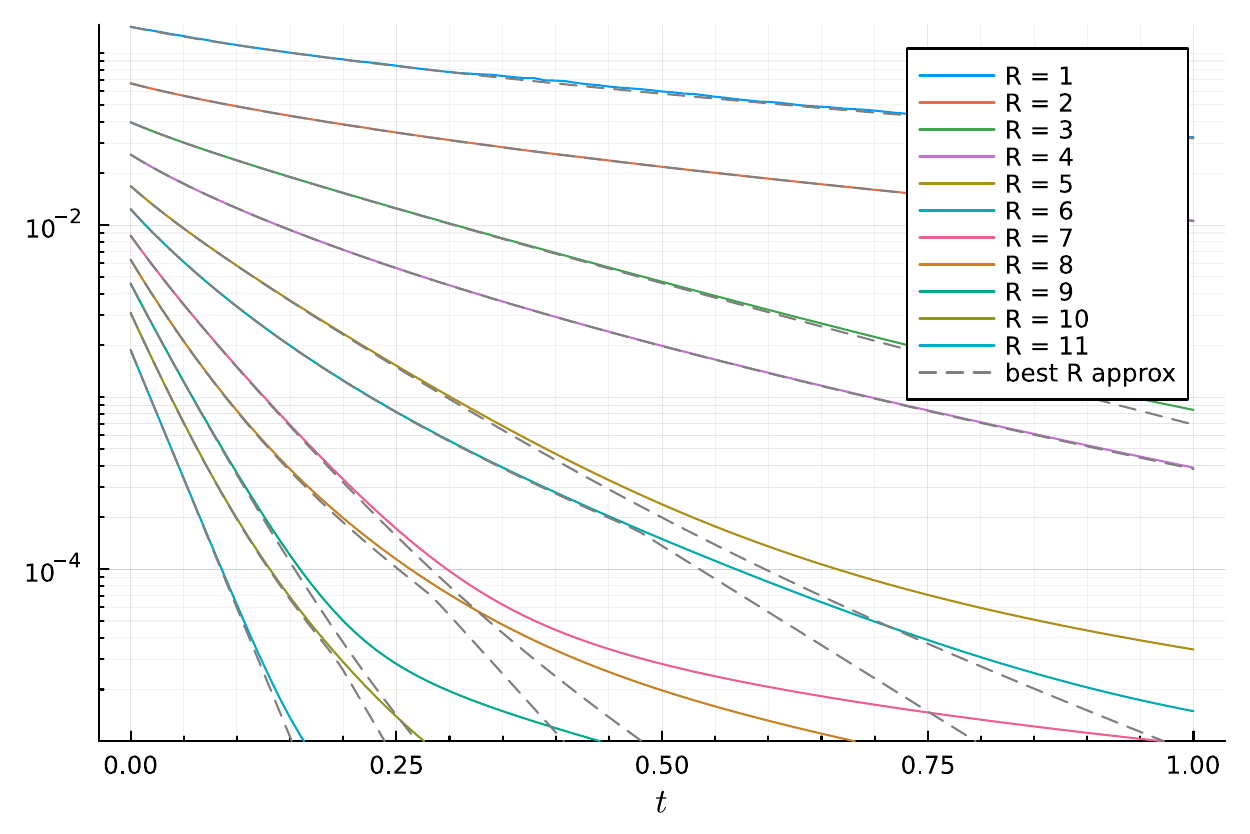}
	\caption{Errors $\trnorm{\hXcal_t^{\fom-\enkf} - \hXcal_t^{\dlr-\enkf}}_{\ell^2_{\rb{N}}(H)}$.}
	\label{fig:dlr-enkf-l2errs}
\end{figure}

We now verify the performance of \dlr-\enkf~in the context of signal-tracking via the RMSE metric.
In the case of an ensemble $\hXcal_t$, we use the metric
\begin{equation*}
	\mathrm{RMSE}(\hXcal_t, \Xcal_t^{\signal})^2 = \frac{1}{{\rb{N}}}\sum_{p=1}^{\rb{N}} \modt{\trnorm{\hXcal_t^{(p)} - \Xcal_t^{\signal}}^2_H}.
\end{equation*}
For both observation operators, we generate a true signal and the corresponding observations. 
We simulate the \enkf~with ${\rb{N}}=10$ or ${\rb{N}}'=425$ particles (for the same reasons as previously).
Finally, truncating the initial condition of \enkf(${\rb{N}}'$) to rank $R = 10$, we set it as the initial condition of \dlr-\enkf($R,{\rb{N}}'$).
As can be seen in~\cref{fig:avgd-rmses}, in both regimes of fully or partially observed dynamics, the RMSE statistics of \fom-\enkf(${\rb{N}}'$) and \dlr-\enkf($R,{\rb{N}}'$) remain close over the whole time-interval, and in particular the standard deviations are significantly smaller than those of \fom-\enkf(${\rb{N}}$).
We emphasise that those numerical results were obtained in the case \mods{of} small model noise $\sigma = 10^{-5}$, which is meaningful as it guarantees \mods{that} the ensemble retains a suitable low-rank structure at all times. 
In other numerical \modt{experiments} (not displayed here), when considering larger model noises, the low-rank structure is lost and \dlr-\enkf~could display unstable behaviour (the RMSE initally closely matching that of \fom-\enkf~but eventually diverging).

\begin{figure}[ht]
\centering
\begin{subfigure}{.5\textwidth}
  \centering
  \includegraphics[width=\linewidth]{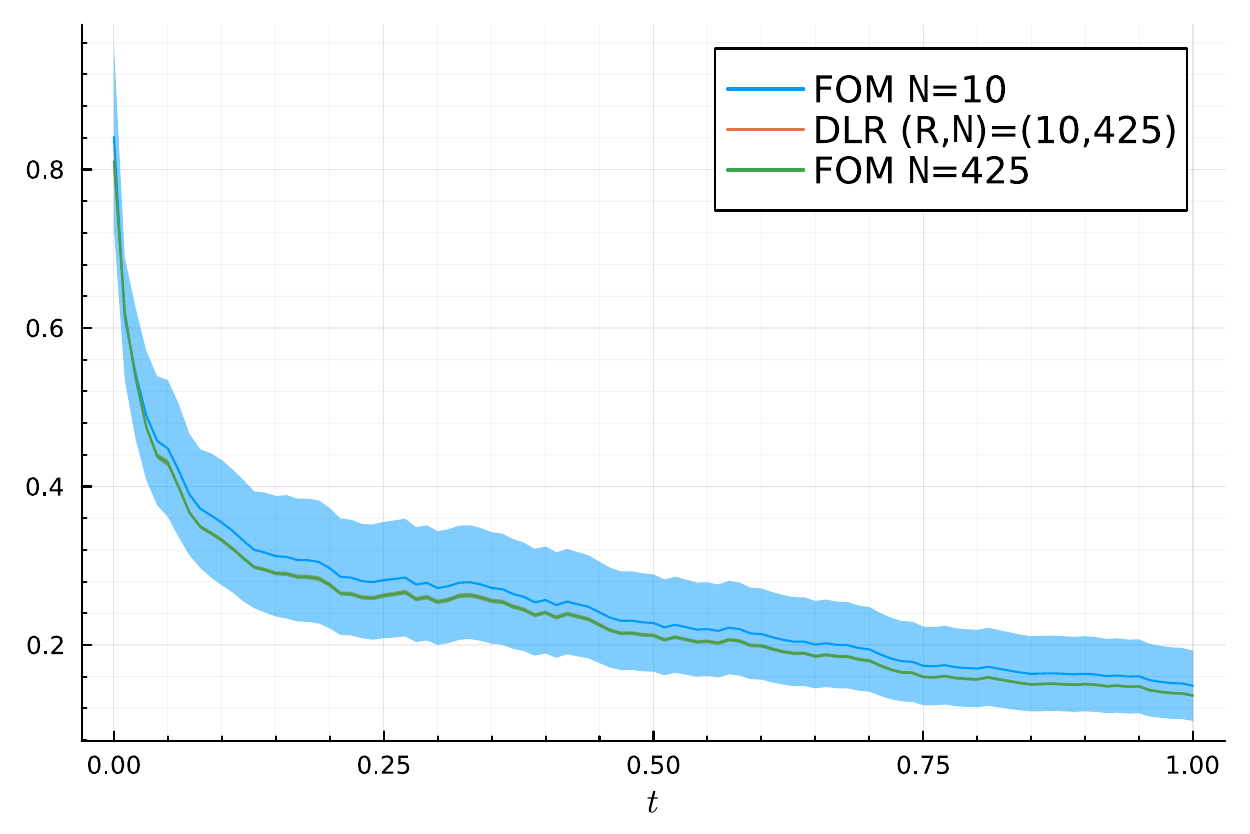}
  \caption{Full observations}
  \label{fig:avg-rmses-full}
\end{subfigure}%
\begin{subfigure}{.5\textwidth}
  \centering
  \includegraphics[width=\linewidth]{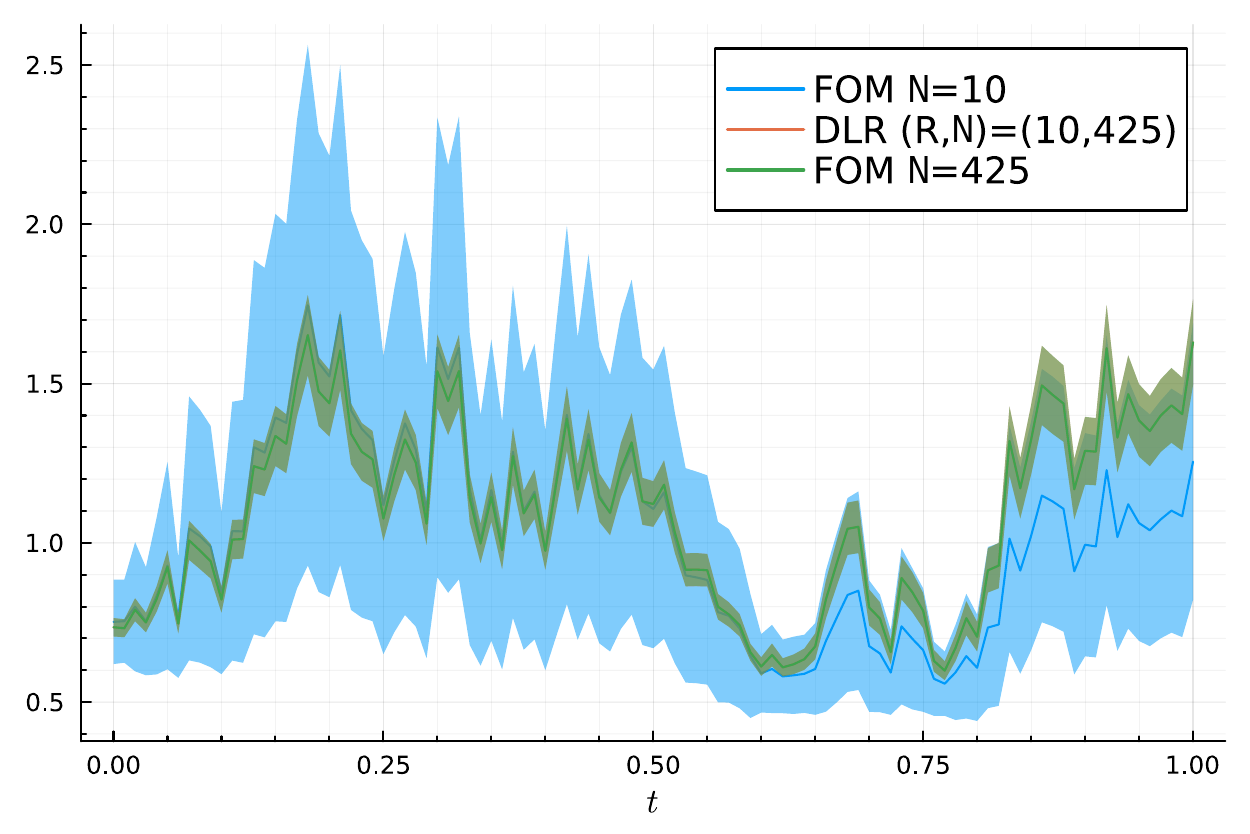}
  \caption{Partial observations}
  \label{fig:avg-rmses-part}
\end{subfigure}
\caption{ Sample mean RMSE $\pm$ standard deviation (over $10$ runs) of \enkf(${\rb{N}}=10$), \enkf(${\rb{N}}=425$) and \dlr-\enkf(${\rb{N}}=425$) with $R = 10$}
\label{fig:avgd-rmses}
\end{figure}

\section{Conclusion}

\modt{
In this work, we introduced a framework that merges Kalman–Bucy filters with the Dynamical Low-Rank approximation technique. 
Succinctly, the approximation evolves the filtering distribution on a low-dimensional, time-varying subspace, thereby reducing computational cost. 
After extending the DLR formulation of SDEs developed in~\cite{DynamicalLowZoccolan2023} to processes driven by an additional \mods{innovation} \Ito~process, we applied it to the Kalman–Bucy process with linear-affine dynamics and derived the \dlr-\kbp~system. 
We proved that the resulting \dlr-\kbp~is a Gaussian process and explicitly characterized its mean and (reduced) covariance, leading to reduced Kalman–Bucy equations. 
While these reduced equations are not new -- having been introduced in~\cite{LowRankApproximatedTsuzukiB2024,ComparisonOfEstYamada2021,OnANewLowRYamada2021} as \textit{ad hoc} model-order reduction schemes -- our analysis derives them rigorously from the low-rank structure of the underlying process. 
Moreover, we quantified the discrepancy between the \dlr-\kbp~and the full \kbp~in terms of the model noise size $\Sigma$, highlighting the necessity of assuming small or vanishing model noise, an assumption consistent with common data assimilation settings. 
Building on this analysis, we proposed an ensemble-type variant of the \dlr-\kbp, the \dlr-\enkf, and investigated its theoretical properties. 
In particular, we established well-posedness for linear-affine dynamics with full observations and, under the same assumptions, demonstrated a propagation-of-chaos property linking its mean-field limit to the \dlr-\kbp.
}

Beyond these contributions, the framework also suggests several directions for future research. 
\rb{A key direction is given by the suitable handling of \enkf\ filters for nonlinear systems, which offers several challenges.
	Such problems can display multi-modal and/or non-Gaussian filtering densities, and devising approaches to suitably handle such posteriors within a DLR context is of interest; in particular, we mention the Gaussian Mixture Model approach proposed in \cite{DataAssimilatiPt2Sonder2013,DataAssimilatiSonder2013}.
	Furthermore,} to reduce the potentially expensive evaluation cost of nonlinear terms, it is necessary to combine the \dlr-\enkf~with hyper-reduction techniques such as the Discrete Empirical Interpolation Method (DEIM)~\cite{NonlinearModelChatur2010}. 
%A relevant example of this situation arises in parameter identification within the filtering context~\cite{StateAndParamNusken2019}, where a non-linear system based on the \kbp~equations induce dynamics on the parameters, (hopefully) driving them to their true value. 
Additionally, the design of robust time-integration methods for \dlr-\enkf~systems constitutes a further important step in making these techniques practical for realistic data assimilation problems.
This is complicated by the fact that explicit Euler–Maruyama-type integrators can be numerically unstable for Kalman–Bucy processes (see~\cite{AStronglyCBlomker2018} and references therein), and furthermore dynamical low-rank systems often become stiff when the subspace is over-approximated, requiring bespoke integrators to ensure stability. 
\rb{Our recent work \cite{DynamicalLTrigo2026} in fact offers answers to several of the above issues, and builds on the present article to show that parameter identification tasks can be successfully carried out by extensions of the \dlr-\enkf\ to nonlinear systems.}
Finally, \rb{combining filter stabilisation techniques such as localisation or variance inflation with the \dlr\ methodology is of high interest, although evidently challenging as they operate on different levels. 
	Indeed, localisation leverages pointwise operations to stabilise the filter, whereas the \dlr-\enkf\ (which already aims to increase stability by enabling larger sample sizes at equivalent cost, thus reducing stochastic fluctuations in the covariance) operates on a mode-based decompositions. 
\rf{The combination of these approaches could be particularly relevant}} in intermediate regimes with a moderate number of particles.

\appendix

\section{Consistency of the DO equations} \label{app:a}

\ra{This Appendix is dedicated to proving~\Cref{th:meansepDO}.}

\subsection{Preparatory statements}

\begin{lemma} \label{lem:cond-expec}
	Let $\rf{\Xcal_t^{\true}}$ be the solution to~\cref{eqn:general-SDE}. 
	Assume that 
	\begin{enumerate}
		\item \modt{$a(t, \omega), b(t, \omega)$ and $c(t, \omega)$ are adapted and \rf{a.s.} continuous,} 
		\item \modt{$\E \int_0^t (\norm{a_s}^2 + \norm{b_s}_F^2 + \norm{c_s}^2 (\norm{\gamma_s}_F^2 + \norm{\Hcal_s}^2 + 1)) \md s < \infty$,}
	\item $a(t, \modr{\cdot})$, $c(t, \modr{\cdot}) \in L^2(\Omega)$ for \modt{almost} all $t \geq 0$, 
	\item  $\E [ \rf{a_t} \st \Acal_{\Zcal_t}]$ and $\E[ \rf{c_t} \st \Acal_{\Zcal_t}]$ have continuous sample paths almost surely, \modt{and }
	\item \modt{$\E \int_0^t  \norm{\E [c_s | \Acal_{\Zcal_s}]}^2 (\norm{\gamma_s}_F^2 + \norm{\Hcal_s}^2) \md s < \infty$,}
		%\item for $f(t, \rf{\Xcal_t^{\true}}) \in \{c(t, \rf{\Xcal_t^{\true}}), \E[c \st \Acal_{\Zcal_t}]\}$, $f(t, \rf{\Xcal_t^{\true}})$ is mean-square continuous, and the mapping $(s,s',u,u') \mapsto \E[\langle f_s H\Xcal^{\signal}_u, f_{s'} H \Xcal^{\signal}_{u'} \rangle]$ is continuous.
	\end{enumerate}
	\rf{where $a_t$, $b_t$, $c_t$ denote $a$, $b$, $c$ evaluated at $(t, \Xcal_t^{\true}, \Law(\Xcal_t^{\true}))$.}
	Then, 
	\begin{equation} \label{eqn:mean-separation}
		\rd \E[\rf{\Xcal_t^{\true}} \st \Acal_{\Zcal_t}] = 
		\E[ \rf{a_t} \st \Acal_{\Zcal_t}] \rd t + \E[ \rf{c_t} \st \Acal_{\Zcal_t} ] \rd \Zcal_t,
	\end{equation}
	and both \modt{$\E[ \rf{a_t} \st \Acal_{\Zcal_t}]$ and $\E[ \rf{c_t} \st \Acal_{\Zcal_t}]$} are $\Acal_{\Zcal_t}$-progressively measurable. 
\end{lemma}
\begin{proof}
	Since~$\E [ a(t, \modt{\omega}) \st \Acal_{\Zcal_t}]$ and $\E[ c(t, \modt{\omega}) \st \Acal_{\Zcal_t}]$ are $\Acal_{\Zcal_t}$-measurable and have continuous sample paths almost surely by assumption, progressive measurability follows by~\cite[Proposition 1.13]{BrownianMKaratzas1991}. 
	In integral form,
	\begin{align}
		\E[\rf{\Xcal_t^{\true}} \st \Acal_{\Zcal_t}] &= \E \left[ \int_{0}^t a(s, \modt{\omega}) \md s \st \Acal_{\Zcal_t} \right] + \E \left[ \int_{0}^t b(s, \modt{\omega}) \md \overline{\Wcal}_s \st \Acal_{\Zcal_t} \right] +  \E \left[ \int_{0}^t c(s, \modt{\omega}) \md \Zcal_s \st \Acal_{\Zcal_t} \right]  \nonumber
		\\
						&=
						\int_{0}^t  \E \left[ a(s, \modt{\omega}) \st \Acal_{\Zcal_t} \right] \md s + \modt{\E \left[ \int_{0}^t b(s, \modt{\omega}) \md \overline{\Wcal}_s \st \Acal_{\Zcal_t} \right]} + \E \left[ \int_{0}^t c(s, \modt{\omega}) \md \Zcal_s \st \Acal_{\Zcal_t} \right], \label{eq:cond-exp-dec}
	\end{align}
	\modt{Let us decompose}
	\begin{equation*}
		\E \left[ a(s, \modt{\omega}) \st \Acal_{\Zcal_t} \right]  = \E \left[ \modt{a_s} - \E[\modt{a_s}] \st \Acal_{\Zcal_t}  \right] + \E \left[ \modt{a_s} \right] = \E [\modt{a^0_s} \st \Acal_{\Zcal_t}] + \E[\modt{a_s}],
	\end{equation*}
	where $\E [\E[\modt{a^0_s} \st \Acal_{\Zcal_t}]] = \E[\modt{a^0_s}] = 0$ by the law of total expectation.
	Then, defining \modt{$A_{\Zcal_s} = L^2_0(\Acal_{\Zcal_s})$}, $A_{\overline{\Wcal}_s} = L^2_0(\Acal_{\overline{\Wcal}_s})$ the subspaces of \rf{zero-mean $\Acal_{\Zcal_t}$- resp. $\Acal_{\overline{\Wcal}_s}$-measurable random variables}, the conditional expectation $\E[\modt{a^0_s} \st \Acal_{\Zcal_s}] = P_{A_{\Zcal_s}}(\modt{a^0_s})$ \modt{is} the orthogonal $L^2$-projection \modt{of $a_s^0$} onto $\mods{A}_{\Zcal_s}$ (well-defined since $\modt{a_s},\modt{a^0_s} \in L^2(\Omega)$). 
	Since $\modt{a^0_s}$ \modt{is measurable with respect to} $\Acal_{\overline{\Wcal}_s} \vee \Acal_{\Zcal_s}$ (the $\sigma$-algebra generated by both $\sigma$-algebras) and those two $\sigma$-algebras are independent, 
	by~\cite[Theorem 8.13]{FoundationsOfModKallenberg2021}, $\modt{a^0_s} \in A_{\Zcal_s} \oplus A_{\overline{\Wcal}_s}$.
	In particular, it admits the decomposition
	\begin{equation*}
		\modt{a^0_s} = P_{A_{\Zcal_s}}(\modt{a^0_s}) + \left(\modt{a^0_s} -  P_{A_{\Zcal_s}}(\modt{a^0_s})  \right)
	\end{equation*}
	where the first term is $\Acal_{\Zcal_s}$-measurable and the second one is $\Acal_{\overline{\Wcal}_s}$-measurable.
	\modt{Taking conditional expectation w.r.t $\Acal_{\Zcal_t}$ we have}, for the first term, 
	\begin{equation*}
		\E \left[ P_{A_{\Zcal_s}}(\modt{a^0_s}) \st \Acal_{\Zcal_t} \right] = 
	\E \left[ \E \left[ \modt{a^0_s} \st \Acal_{\Zcal_s} \right] \st \Acal_{\Zcal_t} \right] = \E \left[ \modt{a^0_s} \st \Acal_{\Zcal_s} \right] %= \E\left[ a(s, \Xcal_s) \st \Acal_{\Zcal_s} \right]
	\end{equation*}
	since $\Acal_{\Zcal_s} \subset \Acal_{\Zcal_t}$ for $s \leq t$. 
	For the second term,	
	\begin{equation*}
		\E [ \mods{a^0_s} -  P_{A_{\Zcal_s}}(\mods{a^0_s})  \st \Acal_{\Zcal_t} ] = \E[\rf{a_s^0} - P_{A_{\Zcal_s}}(\rf{a_s^0}) ] = 0,
	\end{equation*}
since \modt{$a^0_s - P_{\Acal_{\Zcal_s}} a^0_s$} is $\Acal_{\overline{\Wcal}_s}$-measurable hence independent of $\Acal_{\Zcal_t}$.
	Combining \modt{both terms we have}
	\begin{equation*}
		\E [\mods{a^0_s} \st \Acal_{\Zcal_t}] + \E[\rf{a_s}] = \E [\mods{a^0_s} \st \Acal_{\Zcal_s}] + \E[\rf{a_s}]  =  \E [\rf{a_s} \st \Acal_{\Zcal_s}],
	\end{equation*}
	yielding \modt{the first term in the r.h.s. \rf{of \cref{eqn:mean-separation}}}.

	\mods{\mods{For the third term in the r.h.s.\ of~\cref{eq:cond-exp-dec}, we begin by showing that a sequence of stochastic integrals of simple functions suitably approximate $\int_0^t c(s,\omega) \md \Zcal_s$. 
		Firstly, we assume that $c(s,\omega)$ is uniformly bounded.
		Then,} consider the sequence of simple functions
		\begin{equation} \label{eqn:simple-functions}
			c^n(t, \omega) = c_0^n(\omega) \chi_{0}(t) + \sum_{k=0}^{2^n - 1} c_k^n(\omega) \chi_{(s_k, s_{k+1}]}(t),
		\end{equation}	
		where for $s_k = k t / 2^n$, $c^n_k = c(s_k, \omega)$ is $\Acal_{s_k}$-measurable.
		Then, following the proof in~\cite[Lemma 3.2.4]{BrownianMKaratzas1991}, it holds 
		\begin{equation*}
			\int_{0}^t c^n \md \Zcal_s = \sum_{k=0}^{2^n - 1} c^n_k(\omega) (Z_{s_{k+1}} - Z_{s_k}) \to \int_{0}^t c(s, \omega) \md \Zcal_s \quad \text{in } L^2.
		\end{equation*}
		Indeed, 
		\mods{
		\begin{multline}
			\E \left[ \norm{ \int_0^t ( c^n(s, \omega) - c(s, \omega) ) \md \Zcal_s }^2 \right]
			\\
			=
			\E \left[  \norm{ \int_0^t ( c^n(s, \omega) - c(s, \omega) ) \Hcal_s \md s +
			\int_0^t ( c^n(s, \omega) - c(s, \omega) ) \gamma_s \md \widetilde{\Vcal}_s}^2 \right] 
				\\
			\leq 
			2 \E \left[  \norm{ \int_0^t ( c^n(s, \omega) - c(s, \omega) ) \Hcal_s \md s }^2 \right] +
			2 \E \left[ \norm{ \int_0^t ( c^n(s, \omega) - c(s, \omega) ) \gamma_s \md \widetilde{\Vcal}_s}^2 \right] 
			\\
			\leq 
		2 t \E \left[  \int_0^t  \norm{ c^n(s, \omega) - c(s, \omega) }^2 \norm{ \Hcal_s }^2 \md s  \right] +
		2 \E \left[  \int_0^t \norm{ c^n(s, \omega) - c(s, \omega) }^2  \norm{ \gamma_s }^2_F \md s \right] 
			%\\
			%\E \left[ \norm{\sum_{k=0}^{2^n-1} \int_{s_k}^{s_{k+1}} (c^n_k(\omega) - c(s, \omega)) \Hcal_s \md s + \int_{s_k}^{s_{k+1}} (c^n_k(\omega) - c(s, \omega)) \gamma_s \md \widetilde{\Vcal}_s}^2 \right]
			%\\
			%\leq
			%2 \sum_{k=0}^{2^n-1} \E \left[ \norm{  \int_{s_k}^{s_{k+1}} (c^n_k(\omega) - c(s, \omega)) \Hcal_s \md s }^2 
			%+ \norm{\int_{s_k}^{s_{k+1}} (c^n_k(\omega) - c(s, \omega)) \gamma_s \md \widetilde{\Vcal}_s}^2 \right]
			\\
			\leq 2 \max\{t, 1\} \E \left[  \int_{0}^{t} \norm{c^n(s, \omega) - c(s, \omega)}^2 (\norm{\Hcal_s}^2 +  \norm{\gamma_s}^2_F) \md s \right] \xrightarrow{n \to \infty} 0, \label{eqn:l2-bound}
			%\\
			%\leq 2 \max\{t, 1\} \sum_{k=0}^{2^n-1} \E \left[  \int_{s_k}^{s_{k+1}} \norm{c^n_k(\omega) - c(s, \omega)}^2 (\norm{\Hcal_s}^2 +  \norm{\gamma_s}^2_F) \md s \right] \xrightarrow{n \to \infty} 0, \label{eqn:l2-bound}
		\end{multline}
		having used a Cauchy-Schwarz bound for the term with $\Hcal_s$ and the~\Ito~isometry for the integral with respect to $\widetilde{\Vcal}_t$ in the second inequality. 
		}
		%\begin{multline}
		%	\E \left[ \norm{\sum_{k=0}^{2^n - 1} c^n_k(\omega) (Z_{s_{k+1}} - Z_{s_k}) -  \int_{0}^t c(s, \omega) \md \Zcal_s }^2 \right]
		%	\\
		%	=
		%	\E \left[ \norm{\sum_{k=0}^{2^n-1} \int_{s_k}^{s_{k+1}} (c^n_k(\omega) - c(s, \omega)) \Hcal_s \md s + \int_{s_k}^{s_{k+1}} (c^n_k(\omega) - c(s, \omega)) \gamma_s \md \widetilde{\Vcal}_s}^2 \right]
		%	\\
		%	\leq
		%	2 \sum_{k=0}^{2^n-1} \E \left[ \norm{  \int_{s_k}^{s_{k+1}} (c^n_k(\omega) - c(s, \omega)) \Hcal_s \md s }^2 
		%	+ \norm{\int_{s_k}^{s_{k+1}} (c^n_k(\omega) - c(s, \omega)) \gamma_s \md \widetilde{\Vcal}_s}^2 \right]
		%	\\
		%	\leq 2 \sum_{k=0}^{2^n-1} \E \left[  \int_{s_k}^{s_{k+1}} \norm{c^n_k(\omega) - c(s, \omega)}^2 (\norm{\Hcal_s}^2 +  \norm{\gamma_s}^2_F) \md s \right] \xrightarrow{n \to \infty} 0, \label{eqn:l2-bound}
		%\end{multline}
		Since by assumption $c^n(s, \omega)$ is uniformly bounded, we treat $\rho_s = \norm{\Hcal_s}^2 + \norm{\gamma_s}^2_F$ as a (pathwise a.s. continuous) measure; thanks to the second assumption in the lemma, the dominated convergence theorem yields
		\begin{equation*}
			\E \left[ \int_0^t \norm{c(s, \omega)- c^{n}(s, \omega)}^2 \rho_s \md s \right] \xrightarrow{n \to \infty} 0.
		\end{equation*}
		\mods{
			If $c(s, \omega)$ is not uniformly bounded, we consider the sequence of bounded continuous processes
		}	
		\begin{equation*}
			\tilde{c}^\mods{m}(s, \omega) = c(s, \omega) \left( \chi_{\norm{c(s, \omega)} \leq \mods{m}} + \chi_{\norm{c(s, \omega)} > \mods{m}} \frac{\mods{m}}{\norm{c(s, \omega)}} \right),
		\end{equation*}
		which, again thanks to the second assumption in the lemma, verifies
		\begin{equation*}
			\E \left[ \int_0^t \norm{ c(s, \omega)- \tilde{c}^{\mods{m}}(s, \omega)}^2 \rho_s \md s \right] \xrightarrow{\mods{m} \to \infty} 0.
		\end{equation*}		
		\mods{We then apply the simple functions convergence argument to the uniformly bounded $\tilde{c}^m$ instead, and obtain convergence. 
		This proves the claim $\int_0^t c^n \md \Zcal_s \to \int_0^t c \md \Zcal_s$ in $L^2$. 
	}
	%Conversely, for a fixed $k$, the functions $\check{c}^{m,n}$ converge to the simple functions $c^k$ defined in~\cref{eqn:simple-functions}, i.e., 
	%	\begin{equation*}
	%		\E \left[ \int_0^t \norm{ {c}^n(s, \omega) - \check{c}^{m,n}(s, \omega)}^2 \rho_s \md s \right] \xrightarrow{m \to \infty} 0,
	%	\end{equation*}
	%	in particular using the (third) assumption of the lemma that $c(t, \cdot) \in L^2(\Omega)$ for almost all $t\geq 0$.
	%	The result follows by combining the above three inequalities.
		From there,}
	\begin{align*}
		\E \left[ \lim_{n \rightarrow + \infty} \sum_{\modt{k=0}}^{\modt{2^n - 1}} \modt{c^n_k(\omega)} (Z_{s_{k+1}} - Z_{s_k}) \st \Acal_{\Zcal_t} \right] 
	 &\overset{(a)}{=} 
	 \lim_{n \rightarrow + \infty} \sum_{\modt{k=0}}^{\modt{2^n - 1}} \E \left[ \modt{c^n_k(\omega)} (Z_{s_{k+1}} - Z_{s_k}) \st \Acal_{\Zcal_t} \right] 
	 \\
	 &\overset{(b)}{=} 
	 \lim_{n \rightarrow + \infty} \sum_{k=0}^{\modt{2^n -1}} \E \left[ \modt{c^n_k} \st \Acal_{\Zcal_t} \right] (Z_{s_{k+1}} - Z_{s_k}) 
	 \\
	 &\overset{(c)}{=} 
	 \lim_{n \rightarrow + \infty} \sum_{k=0}^{\modt{2^n -1}} \E \left[ c(s_k, \modt{\omega})  \st \Acal_{\Zcal_{s_k}} \right] (Z_{s_{k+1}} - Z_{s_k})
	 \\
	 &\overset{(d)}{=} 
	 \int_0^t \E [ c(s, \Xcal_s) \st \Acal_{\Zcal_s}] \md \Zcal_s
	\end{align*}
	where in (a) the limit and conditional expectation commute because the conditional expectation is a continuous (projection) operator,
	(b) follows from $\Zcal_{s_{k+1}} - \Zcal_{s_k}$ being $\Acal_{\Zcal_t}$-measurable (by the pull-out property) and \modt{for (c) we use the same argument as for the conditional expectation of $a$}.
	Finally, (d) holds in $L^2$ by applying the same arguments and proof as for the first stochastic integral.	
	\rf{This yields the second term \rf{in} the r.h.s.\ of \cref{eqn:mean-separation}.}
	
	\modt{It remains to verify that the conditional stochastic integral w.r.t $\overline{\Wcal}_t$ vanishes.
		By the same arguments as for the previous stochastic integral, there exists a sequence of simple functions $b^n(t, \omega)$ such that $\int_0^t b^n(s, \omega) \md \overline{\Wcal}_s \to \int_0^t b(s, \omega) \md \overline{\Wcal}_s$ in $L^2$.
		Then,
		\begin{align*}
			\E \left[ \int_0^t b^n(s, \omega) \md \overline{\Wcal}_s \st \Acal_{\Zcal_t} \right] 
	&= \sum_{k=0}^{2^n - 1} \E \left[ b^n_k (\overline{\Wcal}_{s_{k+1}} -  \overline{\Wcal}_{s_{k}}) \st \Acal_{\Zcal_t} \right]
	\\
	&= \sum_{k=0}^{2^n - 1} \E \left[ \left(b^n_k - \E[b^n_k  \st \Acal_{\Zcal_t}] \right) (\overline{\Wcal}_{s_{k+1}} -  \overline{\Wcal}_{s_{k}}) \st \Acal_{\Zcal_t} \right]
	\\
	&\quad + \E \left[ \E[b^n_k  \st \Acal_{\Zcal_t}] (\overline{\Wcal}_{s_{k+1}} -  \overline{\Wcal}_{s_{k}}) \st \Acal_{\Zcal_t} \right]
	\\
	&\overset{(a)}{=} \sum_{k=0}^{2^n - 1} \E \left[ \left(b^n_k - \E[b^n_k  \st \Acal_{\Zcal_t}] \right) (\overline{\Wcal}_{s_{k+1}} -  \overline{\Wcal}_{s_{k}}) \st \Acal_{\Zcal_t} \right]
	\\
	&\quad + \E[b^n_k  \st \Acal_{\Zcal_t}] \E \left[ \overline{\Wcal}_{s_{k+1}} -  \overline{\Wcal}_{s_{k}} \st \Acal_{\Zcal_t} \right]
	\\
	&\overset{(b)}{=} \sum_{k=0}^{2^n - 1} \E \left[ \left(b^n_k - \E[b^n_k  \st \Acal_{\Zcal_t}] \right) (\overline{\Wcal}_{s_{k+1}} -  \overline{\Wcal}_{s_{k}}) \st \Acal_{\Zcal_t} \right]
	\\
	&\quad + \E[b^n_k  \st \Acal_{\Zcal_t}] \E \left[ \overline{\Wcal}_{s_{k+1}} -  \overline{\Wcal}_{s_{k}}  \right]
	\\
	&\overset{(c)}{=} \sum_{k=0}^{2^n - 1} \E \left[ \left(b^n_k - \E[b^n_k  \st \Acal_{\Zcal_{t}}] \right) (\overline{\Wcal}_{s_{k+1}} -  \overline{\Wcal}_{s_{k}}) \right] 
	\\
	&\overset{(d)}{=} \sum_{k=0}^{2^n - 1} \E \left[ \left(b^n_k - \E[b^n_k  \st \Acal_{\Zcal_{s_k}}] \right) (\overline{\Wcal}_{s_{k+1}} -  \overline{\Wcal}_{s_{k}}) \right] = 0,
		\end{align*}
		having used the pull-out property in (a), the independence of $\overline{\Wcal}_t$ and $\Acal_{\Zcal_t}$ in (b), the fact that $(b^n_k - \E[b^n_k \st \Acal_{\Zcal_t}])$ and $\overline{\Wcal}_{s_{k+1}} - \overline{\Wcal}_{s_{k}}$ are independent of $\Acal_{\Zcal_t}$ (the former again by~\cite[Theorem 8.13]{FoundationsOfModKallenberg2021}) for (c) and finally for (d) the property $\E[b^n_k  \st \Acal_{\Zcal_{t}}] = \E[b^n_k  \st \Acal_{\Zcal_{s_k}}]$. 
		The remaining term vanishes as, by the discussion above, $b^n_k - \E[b^n_k  \st \Acal_{\Zcal_{s_k}}]$ is $\Acal_{\mods{\overline{\Wcal}_{s_k}}}$-measurable, and hence the increment $\overline{\Wcal}_{s_{k+1}} - \overline{\Wcal}_{s_k}$ is independent of it.
		This concludes the proof.
}
\end{proof}

\begin{lemma} \label{lem:matchcoeff}
	\mods{Suppose \rf{that for $\Xcal_t^{\true}$ satisfying \cref{eqn:general-SDE}} there exists adapted continuous processes $a'$, $b'$, $c'$ such that}	
	\begin{equation*}
		\mods{	 \md \rf{\Xcal_t^{\true}} = a_t' \md t + b_t' \md \overline{\Wcal}_t + c_t' \md \Zcal_t. }
	\end{equation*}
	Then $a_t = a_t'$, $b_t = b_t'$, $c_t=c_t'$.
\end{lemma}
\begin{proof}
	Since $\Zcal_t$ is an \Ito~process, rewrite
	\begin{equation*}
	\tilde{a}_t \md t + b_t \md \overline{\Wcal}_t + \modt{\tilde{c}_t} \md \widetilde{\Vcal}_t = \md \rf{\Xcal_t^{\true}} = \tilde{a}_t' \md t + b_t' \md \overline{\Wcal}_t + \modt{\tilde{c}_t'} \md \widetilde{\Vcal}_t,
	\end{equation*}
	\modt{with $\tilde{a}_t = a_t + \Hcal_t$, $\tilde{c}_t = c_t \gamma_t$, and similarly for $\tilde{a}_t'$ and $\tilde{c}'_t$.}
	Then
	\begin{equation*}
	(\tilde{a}_t - \tilde{a}_t') \md t =  (b_t' - b_t) \md \overline{\Wcal}_t + \modt{(\tilde{c}_t' - \tilde{c}_t)} \md \widetilde{\Vcal}_t.
	\end{equation*}
	As the process on the left has bounded variations and the one on the right does not, it must be that $\tilde{a}_t' = \tilde{a}_t$ almost surely and $(b_t' - b_t) \md \overline{\Wcal}_t + (c_t' - c_t) \md \widetilde{\Vcal}_t = 0$ almost surely. 
	Let $\Ncal_t$ \modt{be} defined by $ (b_t - b_t') \md \overline{\Wcal}_t = \md \Ncal_t = (c_t' - c_t) \md \widetilde{\Vcal}_t$ almost surely.
	In particular, its quadratic (co-)variation $\md [\modt{\Ncal_t^i}, \modt{\Ncal_t^j}]_t$ is
	\modt{
	\begin{align*}
		\md [\Ncal_t^i, \Ncal_t^i] 
		= \sum_{\ell=1}^k (\tilde{c}_t' -  \tilde{c}_t)^2_{i \ell} \rf{\md t}
		&= \sum_{\ell = 1}^{\bar{d}} (b_t' - b_t)^2_{i \ell} \rf{\md t}
		\\
		&= \md [(b_t - b_t')_{ii} \md \overline{\Wcal}_t, (c_t' - c_t)_{ii} \md \widetilde{\Vcal}_t] = 0,
	\end{align*}
	owing to the independence of $\overline{\Wcal}_t$ and $\widetilde{\Vcal}_t$.}
This implies $\modt{(b_t')_{ik} = (b_t)_{ik}}$ for all indices and similarly for $c_t$ and $c'_t$.
\end{proof}

\subsection{Proof of Theorem 3.1} \label{ssec:dlr-da-eqs}

\ra{We start by noting that Conditions~\ref{item:lrdecomp}--\ref{item:coeffconstraint} are straightforwardly satisfied by \cref{eqn:zeromean-mean,eqn:zeromean-phys,eqn:zeromean-stoch,eqn:correction-term-general} in \cref{th:meansepDO}. 
	\rf{It remains to prove the consistency argument in the second part of the statement.
	Suppose $\Xcal_t^{\true}$ satisfies} $\Xcal_t^{\true} = U_t^0 + \sum_{i=1}^R U_t^i Y_t^i$ for all $t \geq 0$ with dynamics dictated by
	\begin{align*} 
		\md U_t^0 &= \alpha_t^0 \md t + \beta_t^0 \md \Zcal_t, & 
		\md U_t^i &= \alpha_t^i \md t + \beta_t^i \md \Zcal_t,
			  & \md Y_t^i &= \gamma_t^i \md t + \delta_t^i \md \overline{\Wcal}_t + \varepsilon_t^i \md \Zcal_t,
	\end{align*}
	and satisfying the orthogonality \rf{condition $\md \bU_t^{\top} \bU_t = 0$, equivalent to}
	\begin{align*}
		(U^i_t)^{\top} \alpha_t^j &= 0 & (U^i_t)^{\top} \beta_t^j &= 0 & \rf{1 \leq i,j \leq R.}
	\end{align*}
	\rf{We need to show that} the evolution equations of $(U_t^0, U_t^i, Y_t^i)$ match \cref{eqn:zeromean-mean,eqn:zeromean-phys,eqn:zeromean-stoch,eqn:correction-term-general}. 

	We begin {by deriving the equation for} $U_t^0$. Taking conditional expectations in~\cref{eqn:general-SDE}, and since by assumption the coefficients verify the conditions of \cref{lem:cond-expec}, applying the lemma yields
}
\begin{equation*}
	\rd U_t^0
	=
	\rd \E[\Xcal_t^{\ra{\true}} \st \Acal_{\Zcal_t}] 
	=
	\E[ \rf{a_t} \st \Acal_{\Zcal_t}] \rd t + \E[ \rf{c_t} \st \Acal_{\Zcal_t} ] \rd \Zcal_t,
\end{equation*}
which identifies $\alpha_t^0$ and $\beta_t^0$ \rf{as the terms in \eqref{eqn:zeromean-mean}.}
The \rf{(conditionally)} zero-mean process $\Xcal_t^{\star} = \Xcal_t^{\ra{\true}} - \E [\Xcal_t^{\ra{\true}} \st \Acal_{\Zcal_t}]$ then verifies
\begin{align*}
	\rd \Xcal_t^{\star} 
		%=  a \md t + b \md \Wcal_t + c \md \Zcal_t - \E[a \st \Acal_{\Zcal_t}] \rd t - \E[c \rd \Zcal_t \st \Acal_{\Zcal_t}  ]  
		    &= \rd \left(\sum_{i=1}^R U_t^i Y^i_t \right)
		    %\\
		    = 
		    \rf{a_t^{\star}} \md t + \rf{b_t} \md \overline{\Wcal}_t + \rf{c_t^{\star}} \md \Zcal_t.
\end{align*}
By the \Ito~formula,
\begin{align*}
	\rd \left( \sum_{i=1}^R U_t^i Y_t^i \right) &= \sum_{i=1}^R \rd U_t^i Y_t^i + U_t^i \rd  Y_t^i + \rd [U_t^i, Y_t^i]_t
	\\
						    &= \sum_{i=1}^R  Y_t^i \alpha_t^i \rd t +  Y_t^i \beta_t^i \rd \Zcal_t  + U_t^i \gamma_t^i \rd t +  U_t^i \delta_t^i\rd \overline{\Wcal}_t + U_t^i \varepsilon_t^i \rd \Zcal_t + \rf{G^i_t} \md t 	\\ 
						    &= \rf{a^{\star}_t} \rd t + \rf{b_t} \rd \overline{\Wcal}_t + \rf{c^{\star}_t} \rd \Zcal_t,
\end{align*}
where 
\begin{equation} \label{eqn:correction-term}
	\rf{G^i_t} \md t = \sum_{s,s^\prime \mods{=1}}^k (\beta_t^i)_{:,s} (\varepsilon_t^i)_{s^{\prime}} \rd [\Zcal_t^{\modt{s}}, \Zcal_t^{\modt{s'}}]_t.
\end{equation}
For convenience, \rf{denote $G_t = \sum_{i=1}^R G^i_t$}.
We seek to identify the coefficients $\alpha_t^i, \beta_t^i, \gamma_t^i, \delta_t^i, \varepsilon_t^i$, for $i=1, \ldots R$.
\rb{By \cref{lem:matchcoeff} we can match coefficients, and we immediately have}
\begin{equation*}
\sum_{i=1}^R U_t^i \delta_t^i = b \iff \bU_t^{\top} \bU_t \bdelta_t = \bU_t^{\top} b \iff
\bdelta_t = M_{\bU_t}^{-1} \bU_t^{\top} b,
\end{equation*} 
where $\bdelta_t = (\delta_t^1, \ldots, \delta_t^R)^{\top}$. 
Matching terms, the other two relations are
\begin{equation} \label{eqn:do-relations}
	\sum_{i=1}^R  Y_t^i \alpha_t^i + U_t^i \gamma_t^i = \rf{a_t^{\star} - G_t}
  \quad\mathrm{and}\quad
  \sum_{i=1}^R Y_t^i \beta_t^i + U_t^i \varepsilon_t^i = \rf{c_t^{\star}}.
\end{equation}
Two observations are in order. If $\beta_t^i$ and $\varepsilon_t^i$ are determined from the second equation, 
$\rf{G_t}$ is fully determined. 
\mods{Secondly, the system~\eqref{eqn:do-relations} in the unknowns $\alpha^i_t, \beta^i_t, \gamma^i_t, \varepsilon^i_t$ is under-determined.}
%Secondly, even in that situation, both systems are redundant. 
%Indeed, for fixed $\beta_t^i$ and $\varepsilon_t^i$, the constraints on both equations have the same structure
%$\sum_{i=1}^R Y_t^i \modt{e}^i + U^i_t d^i = \mathrm{r.h.s.}$, which admits infinitely many solutions.
Imposing $\rf{\md \bU_t}$ orthogonal to $\bU_t$ removes that redundancy -- this is the condition in~\cref{eqn:mod-do}.
Applying $\modt{\Pi_{\bU_t}} = \bU_t M_{\bU_t}^{-1} \bU_t^{\top}$ resp. $\modt{\Pi^{\perp}_{\bU_t}} = I - \modt{\Pi_{\bU_t}}$ to both relations shows that
\begin{align*}
	&\sum_{i=1}^R Y_t^i \modt{\Pi_{\bU_t}} \beta_t^i + U_t^i \varepsilon_t^i = \modt{\Pi_{\bU_t}} \rf{c_t^{\star}} &
	&\sum_{i=1}^R \modt{\Pi^{\perp}_{\bU_t}}	\alpha_t^i Y^i_t	= \modt{\Pi^{\perp}_{\bU_t}} \rf{a_t^{\star}} - \sum_{i=1}^R \sum_{s,s^\prime \modt{=1}}^k \modt{\Pi^{\perp}_{\bU_t}} (\beta_t^i)_{:,s} (\varepsilon_t^i)_{s^{\prime}} \rd [\Zcal_t^{\modt{s}}, \Zcal_t^{\modt{s'}}]_t 
	\\
	&\sum_{i=1}^R Y_t^i \modt{\Pi^{\perp}_{\bU_t}} \beta_t^i  = \modt{\Pi^{\perp}_{\bU_t}} \rf{c^{\star}_t} &
							&\sum_{i=1}^R \modt{\Pi_{\bU_t}}	\alpha_t^i Y^i_t \! + \! U^i_t \gamma_t^i 
							= \Pi_{\bU_t} \rf{a_t^{\star}}\!\!-\!\!\sum_{i=1}^R \sum_{s,s^\prime \modt{ = 1}}^k \modt{\Pi_{\bU_t}} (\beta_t^i)_{:,s} (\varepsilon_t^i)_{s^{\prime}} \rd [\Zcal_t^{\modt{s}}, \Zcal_t^{\modt{s^{\prime}}}]_t 
\end{align*}
Combining with~\cref{eqn:mod-do} (noting it also implies $\modt{\Pi_{\bU_t}^{\perp}} \alpha_t^i = \alpha_t^i$ and $\modt{\Pi_{\bU_t}^{\perp}} \beta_t^i = \beta_t^i$), \mods{we obtain}
\begin{align*}
	\modt{(i)    \quad}	\sum_{i=1}^R U_t^i \varepsilon_t^i &= \modt{\Pi_{\bU_t}} \rf{c_t^{\star}} &
	\modt{(ii)  \quad} \sum_{i=1}^R \alpha_t^i Y^i_t &= \modt{\Pi_{\bU_t}^{\perp}} \rf{a_t^{\star} - G_t} \\
	\modt{(iii) \quad} \sum_{i=1}^R Y_t^i \beta_t^i  &= \modt{\Pi^{\perp}_{\bU_t}} \rf{c^{\star}_t} &
	\modt{(iv)  \quad} \sum_{i=1}^R  U^i_t \gamma_t^i  &= \modt{\Pi_{\bU_t}} \rf{a_t^{\star}}.
\end{align*}
Performing the same development as for $\bdelta_t$ yields
\begin{align*}
	\bgamma_t &= M^{-1}_{\bU_t} \bU_t^{\top} \rf{a_t^{\star}} & \modt{\bm{\varepsilon}_t} &= \mods{M_{\bU_t}^{-1} \bU_t^{\top} \rf{c_t^{\star}}}
\end{align*}
For equality \modt{$(iii)$}, multiplying on both sides by $Y^j_t$ and taking the conditional expectation
\begin{equation*}
	\sum_{i=1}^R \beta_t^i (M_{\bY_t})_{ij} = \E[\modt{\Pi^{\perp}_{\bU_t}} \rf{c_t^{\star}} Y_t^j \st \Acal_{\Zcal_t}], \quad j = 1, \ldots, R \iff \beta^i_t = \modt{\Pi^{\perp}_{\bU_t}} \E[ \rf{c_t^{\star}} (\bY_t M_{\bY_t}^{-1})_i
 \st \Acal_{\Zcal_t}] 
\end{equation*}
leveraging the property $\E[ \beta_t^i \st \Acal_{\Zcal_t} ] = \beta_t^i$ since $\beta_t^i$ is $\Acal_{\Zcal_t}$-measurable and furthermore it holds $\E[\modt{\Pi^{\perp}_{\bU_t}}f \st \Acal_{\Zcal_t}] = \modt{\Pi^{\perp}_{\bU_t}} \E[f \st \Acal_{\Zcal_t}]$ because $\bU_t$ is $\Acal_{\Zcal_t}$-measurable.
Applying the same method to the remaining equation yields (also using the fact that the $\alpha_t^i$ are $\Acal_{\Zcal_t}$-measurable)
\begin{align}
	\modt{\bm{\alpha}_t^i =  \modt{\Pi^{\perp}_{\bU_t}}\E \left[ ( \rf{a_t^{\star} - G_t}) (\bY_t M_{\bY_t}^{-1}) \st \Acal_{\Zcal_t} \right].} 
\end{align}
\modt{With that last step, we have identified all involved coefficients, and combining those results yields~\cref{eqn:zeromean-mean,eqn:zeromean-phys,eqn:zeromean-stoch}. 
\ra{This concludes the proof of \cref{th:meansepDO}.}}

\section{Additional proofs} \label{app:b}

\subsection{Proofs of Section 4}

\mods{The following lemma recalls the useful fact that inhomogeneous Ornstein-Uhlenbeck processes are Gaussian processes (with a short proof for the sake of completeness).}
\begin{lemma} \label{lem:in-ou-gauss}
        An inhomogeneous Ornstein-Uhlenbeck process 
	\begin{equation*}
		\rd \Xcal_t = \left(A(t) \Xcal_t + \modt{f}(t) \right) \rd t + \Sigma(t) \rd  \Wcal_t
        \end{equation*}
	with Gaussian initial conditions is a Gaussian process.
\end{lemma}
\begin{proof} 
	Denote $\Phi(t)$ the fundamental solution to $\dot{X}(t) = A(t) X(t)$, i.e., $\dot{\Phi}(t) = A(t) \Phi(t)$, $\Phi(t)$ is invertible and $\Phi(0) = \bm{I}_d$.
	By~\cite{StochasticDiffMao2008}, the solution to the process is given by
	\begin{equation*}
		\Xcal_t = \Phi(t) \left( \Xcal_0 + \int_{0}^t \Phi^{-1}(s) b(s) \md s + \int_{0}^t \Phi^{-1} \Sigma(s) \md \Wcal_s \right). 
	\end{equation*}
	As a sum of Gaussian terms, $\Xcal_t$ is Gaussian too.	
\end{proof}

\begin{proof}[Proof of~\cref{lem:dlr-kbp-wp}]

	This proof follows the approach taken in~\cite[Lemma 5.2]{OnTheStabilitDelMo2018}. 
	We denote by $\varphi_t(P_0) \equiv \varphi_t(\bU_0, M_{\mods{\bY_0}}) = \bU_t \varphi_t(M_{\mods{\bY_0}}) \bU_t^{\top}$ the solution flow to the Riccati equation, which is well-defined by the discussion in~\cref{ssec:red-ricc}.
	Consider an auxiliary system of non-homogeneous diffusions given by \modt{
\begin{align}
	\md \tilde{U}_t^0 &= (A \tilde{U}_t^0 + \modt{f}) \rd t + \varphi_t(P_0) H^{\top} \Gamma^{-1} (\rd \Zcal_t - H \tilde{U}^0_t \rd t) 
	\label{eqn:nhd-mean} 
	\\
	\md \tilde{\bU}_t &= P^{\perp}_{\tilde{\bU}_t} A \tilde{\bU}_t \md t, 
	\label{eqn:nhd-physical-modes} 
	\\
	\md \tilde{\bY}_t^{\top} &= \tilde{\bU}_t^{\top} \tilde{A} \tilde{\bU}_t \tilde{\bY}_t^{\top} \md t + \tilde{\bU}_t^{\top} (\Soh \rd \Wcal_t - \varphi_t(P_0) H^{\top} \Gmoh \md \Vcal_t), 
	\label{eqn:nhd-stochastic-modes}
\end{align}
where $\tilde{A} = A - \varphi_t(P_0) S$, and denote \modt{$\widetilde{\Xcal}_t = \tilde{U}^0_t + \tilde{\bU}_t \tilde{\bY}_t$}.}
We \modt{begin by showing} that this system is well-posed. 
\Cref{eqn:nhd-physical-modes} is well-posed by the discussion in~\cref{sssec:phys-modes}. 
	Owing to the fact that $\varphi_t(P_0)$ is well-defined,~\cref{eqn:nhd-stochastic-modes} is a non-homogeneous, linear SDE in $\tilde{\bY}_t$ -- well-posed by standard SDE theory~\cite{StochasticDiffMao2008}.
\modt{For~\cref{eqn:nhd-mean}, expanding the $\md \Zcal_t$~term via~\cref{eqn:truth-noisy-observations}, $\tilde{U}_t^0$ is seen to be a non-homogeneous linear SDE too, and therefore well-posed -- \mods{the system~\eqref{eqn:nhd-mean} to \cref{eqn:nhd-stochastic-modes} therefore admits a unique solution.}}

	\modt{Next, denoting} \modt{$\eta_t~=~\mathrm{Law}(\widetilde{\Xcal}_t | \Acal_{\Zcal_t} )$}, \modt{we verify that the (conditional)} covariance satisfies $P_{\eta_t} = \varphi_t(P_0)$, \modt{in other words it satisfies the modified Riccati equation~\eqref{eqn:dlr-vkbp-mean}.}
	\modt{To this end, consider}
	\modt{
	\begin{align*}
		\md (\tilde{\bY}_t^{\top} \tilde{\bY}_t) 
		&= \md \tilde{\bY}_t^{\top} \tilde{\bY}_t + \tilde{\bY}_t^{\top} \md \tilde{\bY}_t + \md [\tilde{\bY}_t^{\top}, \tilde{\bY}_t]_t 
		\\
		&=  ( \tilde{\bU}_t^{\top} \tilde{A} \tilde{\bU}_t \tilde{\bY}_t^{\top} \tilde{\bY}_t  
			+ \tilde{\bY}_t^{\top} \tilde{\bY}_t  \tilde{\bU}_t^{\top} \tilde{A}^{\top} \tilde{\bU}_t
			+ \tilde{\bU}_t^{\top} \Sigma \tilde{\bU}_t 
			+ \tilde{\bU}_t^{\top} \varphi_t(P_0) H^{\top} \Gamma^{-1} H \varphi_t(P_0) \tilde{\bU}_t ) \md t
		\\
		& \quad	+ (\tilde{\bU}_t^{\top} \Soh \tilde{\bY}_t \rd \Wcal_t 
		+ \tilde{\bY}_t^{\top} \rd \Wcal_t^{\top} \Soh \tilde{\bU}_t) 
		\\
		&\quad- (\tilde{\bU}_t^{\top}\varphi_t(P_0) H^{\top} \Gmoh \md \Vcal_t \tilde{\bY}_t + \tilde{\bY}_t^{\top} \md \Vcal_t^{\top} \Gmoh H \varphi_t(P_0) \tilde{\bU}_t )   
		\\
		&=  ( 
		A_{\tilde{\bU}_t} \tilde{\bY}_t^{\top} \tilde{\bY}_t  
		+ \tilde{\bY}_t^{\top} \tilde{\bY}_t  A^{\top}_{\tilde{\bU}_t}
		+ \Sigma_{\tilde{\bU}_t} 
		+  \varphi_t(M_{Y_0}) S_{\tilde{\bU}_t} \varphi_t(M_{Y_0}) ) \md t
		\\
		& \quad 
		- (\varphi_t(M_{Y_0}) S_{\tilde{\bU}_t} \tilde{\bY}_t^{\top} \tilde{\bY}_t 
		+ \tilde{\bY}_t^{\top} \tilde{\bY}_t S_{\tilde{\bU}_t} \varphi_t(M_{Y_0})) 	\md t
		+ J \md \Wcal_t + K \md \Vcal_t,
	\end{align*}
	where\modr{, with} a slight abuse of notation, in the third and fourth line, $\md \Wcal_t$ and $\md \Vcal_t$ are treated like vectors following the usual matrix multiplication rules, and $J \md \Wcal_t, K \md \Vcal_t$ represent the corresponding $R \times R$-dimensional expressions. 
	}
	Defining $N_{\tilde{\bY}_t} = \E[\tilde{\bY}_t^{\top} \tilde{\bY}_t \st \Acal_{\Zcal_t}] = \E[\tilde{\bY}_t^{\top} \tilde{\bY}_t] $ since $\tilde{\bY}_t$ is independent of $\Zcal_t$
	\begin{align*}
		\md N_{\modt{\tilde{\bY}_t}}
		&=  ( A_{\modt{\tilde{\bU}_t}} N_{\modt{\tilde{\bY}_t}} 
		+  N_{\modt{\tilde{\bY}_t}} A^{\top}_{\modt{\tilde{\bU}_t}}
		+ \Sigma_{\modt{\tilde{\bU}_t}} ) \md t 
		\\
		& \quad -  (N_{\modt{\tilde{\bY}_t}} S_{\modt{\tilde{\bU}_t}} \varphi_t(M_{\bY_0})  
		+ \varphi_t(M_{\bY_0}) S_{\modt{\tilde{\bU}_t}} N_{\modt{\tilde{\bY}_t}} \mods{) \md t}  + \varphi_t(M_{\bY_0}) S_{\modt{\tilde{\bU}_t}} \varphi_t(M_{\bY_0}) \md t.
	\end{align*}
	Defining the error term $E_t = N_{\tilde{\bY}_t} - \varphi_t(M_{\bY_0})$ and differentiating yields the Lyapunov equation
	\begin{align*}
		\frac{\md E_t}{\md t} = (A_{\modt{\tilde{\bU}_t}} -  \mods{\varphi_t(M_{\bY_0}) S_{\modt{\tilde{\bU}_t}}}) E_t  + E_t (A^{\mods{\top}}_{\modt{\tilde{\bU}_t}} - S_{\modt{\tilde{\bU}_t}} \varphi_t(M_{\bY_0})),
	\end{align*}
	with $E_0 = N_{\bY_0} - M_{\bY_0}$. 
	Vectorising the system
	\begin{equation*}
		\frac{\md}{\md t} \mathrm{vect}(E_t) = (\bm{I}_R \otimes L_t + L_t \otimes \bm{I}_R ) \mathrm{vect}(E_t),
	\end{equation*}
	with $L_t = A_{\modt{\tilde{\bU}_t}}^{\mods{\top}} -  S_{\modt{\tilde{\bU}_t}} \varphi_t(M_{\bY_0})$, \modt{it is seen to be a non-homogeneous linear system, and in particular the dynamics \mods{is} locally Lipschitz.
		This implies $E_t = \bm{0}_R$ if and only if $E_0 = \bm{0}_R$.
		Therefore, if $N_{\modt{\tilde{\bY}}_0} = M_{\bY_0}$, then $N_{\modt{\tilde{\bY}_t}} = \varphi_t(M_{\bY_0})$ \mods{at all times}, in turn implying 
		\begin{equation*}
			P_{\eta_t} = \bU_t N_{\tilde{\bY}_t} \bU_t^{\top} = \bU_t \varphi_t(M_{\bY_0}) \bU_t^{\top} = \varphi_t(P_0).
		\end{equation*}
%		The exhibited process $(\tilde{U}^0_t,\tilde{\bU}_t^0, \tilde{\bY}_t^0)$ therefore satisfies~\cref{eqn:nhd-mean,eqn:nhd-physical-modes,eqn:nhd-stochastic-modes}, which coincide with~\cref{eqn:kb-dlr-U0,eqn:kb-dlr-U,eqn:kb-dlr-Y}; therefore,~\cref{eqn:kb-dlr-U0,eqn:kb-dlr-U,eqn:kb-dlr-Y} are well-posed.}
	\mods{The exhibited process $(\tilde{U}^0_t,\tilde{\bU}_t^0, \tilde{\bY}_t^0)$ hence satisfies~\cref{eqn:kb-dlr-U0,eqn:kb-dlr-U,eqn:kb-dlr-Y}, which coincide with~\cref{eqn:nhd-mean,eqn:nhd-physical-modes,eqn:nhd-stochastic-modes}, which are well-posed; therefore,~\cref{eqn:kb-dlr-U0,eqn:kb-dlr-U,eqn:kb-dlr-Y} are well-posed, and admit the same unique solution.}}
\end{proof}

%\modt{
%\begin{lemma} \label{lem:sup-bounds-glip-int-monot}
%	Let $a(t, x) : \R_+ \times \R^d \to \R^d$ and $b(t, x) : \R_+ \times \R^d \to \Mat{d \times d}$ such that
%	\begin{align}
%		\norm{a(t, x) - a(t, y)}^2 + \norm{b(t, x) - b(t, y)}^2_F &\leq K \norm{x - y}^2, \label{eqn:glob-lip} \\
%		\E \left[ x^{\top} a(t,x) + 2 \norm{b(t,x)}^2_F \right] \leq K (1 +  \E \norm{x}^2). \label{eqn:int-monot}
%	\end{align}
%	and let $\Xcal_t$ a solution to 
%	\begin{equation*}
%		\md \Xcal_t = a(t, \Xcal_t) \md t + b(t, \Xcal_t) \md \Wcal_t.
%	\end{equation*}
%	Then, 
%	\begin{equation*} 
%		\E \left[  \sup_{0 \leq s \leq t} \norm{\Xcal_s}^2 \right] 
%		\leq 
%		(1 + 3 \E\norm{\Xcal_0}^2) e^{3 K t (t - 4)}.
%	\end{equation*}
%\end{lemma}
%\begin{proof}
%	The proof is based on a refinement of the proof in~\cite[Lemma 3.2]{StochasticDiffMao2008}.
%\end{proof}
%}
%
%\modt{
%\begin{proposition} \label{prop:well-posed-glip-int-monot}
%	Let $a(t, x) : \R_+ \times \R^d \to \R^d$ and $b(t, x) : \R_+ \times \R^d \to \Mat{d \times d}$ verify~\cref{eqn:glob-lip,eqn:int-monot}.
%	Then, there exists a unique solution to 
%	\begin{equation*}
%		\md \Xcal_t = a(t, \Xcal_t) \md t + b(t, \Xcal_t) \md \Wcal_t, \quad \Xcal_0 = \xi.
%	\end{equation*}
%\end{proposition}
%\begin{proof}
%	The proof is based on a refinement of the proof in~\cite[Theorem 3.1]{StochasticDiffMao2008}.
%\end{proof}
%}

\subsection{Proofs of Section 5}

\modt{	
	The characterisation of the quadratic variation of the martingale $\Ncal_t$ in~\cref{eqn:sample-gram-martingale} will be useful~\modr{to prove~\cref{lem:beta-gamma bounds,lem:trMYt-sup-bounds,lem:dEt-bound,lem:bounded-moments-xdlr}.}
	\mods{For convenience, the results are summarised in the following lemma.}
	\begin{lemma}
		\mods{Label the elements of $\Ncal_t$ in~\cref{eqn:sample-gram-martingale} as}
	\begin{equation} \label{eqn:Ncal-ij-detail}
		\mods{	(\Ncal_t)_{ij} = \frac{1}{\sqrt{{\rb{N}}-1}} \sum_{p=1}^{{\rb{N}}} (\alpha_{ij}^{(p)} + \alpha_{ji}^{(p)}) \md \hWcal_t^{\star,(p)} + (\beta_{ij}^{(p)} + \beta_{ji}^{(p)}) \md \hVcal_t^{\star,(p)}.}
	\end{equation}
	\mods{Then, the martingale} $\Ncal_t$ has quadratic variation
	\begin{align}
		C_{\Ncal}((i,j),(k,\ell)) &= \frac{\md}{\md t} \left[ (\Ncal_t)_{ij}, (\Ncal_t)_{k\ell} \right]_t \nonumber %\label{eqn:qv-nt-indx} 
		\\
					  %&= \frac{1}{{\rb{N}}-1} \sum_{p=1}^{\rb{N}} (\alpha_{ij}^{(p)} + \alpha_{ji}^{(p)}) (\alpha_{k\ell}^{(p)} + \alpha_{\ell k}^{(p)})^{\top}
					  %+ (\beta_{ij}^{(p)} + \beta_{ji}^{(p)}) (\beta_{k\ell}^{(p)} + \beta_{\ell k}^{(p)})^{\top}  \nonumber
					  %\\
					  %&\quad \modt{-} \frac{1}{({\rb{N}}-1){\rb{N}}}\sum_{\modt{p, p' =1}}^{\rb{N}} (\alpha_{ij}^{(p)} + \alpha_{ji}^{(p)}) (\alpha_{k\ell}^{(p')} + \alpha_{\ell k}^{(p')})^{\top}
					  %+ (\beta_{ij}^{(p)} + \beta_{ji}^{(p)}) (\beta_{k\ell}^{(p')} + \beta_{\ell k}^{(p')})^{\top}  \nonumber
					  %\\
					  &= \frac{1}{{\rb{N}}-1} \sum_{p=1}^{\rb{N}} (\alpha_{ij}^{(p)} + \alpha_{ji}^{(p)}) (\alpha_{k\ell}^{(p)} + \alpha_{\ell k}^{(p)})^{\top}
					  + (\beta_{ij}^{(p)} + \beta_{ji}^{(p)}) (\beta_{k\ell}^{(p)} + \beta_{\ell k}^{(p)})^{\top}. \label{eqn:qv-ncal}
	\end{align}
	\mods{Furthermore, it holds}
	\begin{align} \label{eqn:qv-nt-tr}
		\sum_{i=1}^R \sum_{k=1}^R C_{\Ncal} ((i,k),(k,i)) %&= \mods{\frac{1}{{\rb{N}}-1}}\sum_{i,k=1}^R \sum_{p=1}^{\rb{N}} (\alpha_{ik}^{(p)} + \alpha_{ki}^{(p)}) (\alpha_{ki}^{(p)} + \alpha_{i k}^{(p)})^{\top} 
		%+ (\beta_{ik}^{(p)} + \beta_{ki}^{(p)}) (\beta_{ki}^{(p)} + \beta_{i k}^{(p)})^{\top} \nonumber \\
						       &= \sum_{i,k=1}^R(\Sigma_{\bU_t})_{ik} (\hP_{\hbY_t})_{k i} + (\Sigma_{\bU_t})_{ii} (\hP_{\hbY_t})_{k k}  \nonumber
						       \\
						       & \quad + (\Sigma_{\bU_t})_{kk} (\hP_{\hbY_t})_{ii} + (\Sigma_{\bU_t})_{ki} (\hP_{\hbY_t})_{i k}  \nonumber
						       \\ 
						       & \quad + (\hP_{\hbY_t} S_{\bU_t} \hP_{\hbY_t})_{ik} (\hP_{\hbY_t})_{ki} + (\hP_{\hbY_t} S_{\bU_t} \hP_{\hbY_t})_{ii} (\hP_{\hbY_t})_{kk} \nonumber
						       \\
						       & \quad + (\hP_{\hbY_t} S_{\bU_t} \hP_{\hbY_t})_{kk} (\hP_{\hbY_t})_{ii} + (\hP_{\hbY_t} S_{\bU_t} \hP_{\hbY_t})_{ki} (\hP_{\hbY_t})_{ik}.
	\end{align}

	\end{lemma}
	\begin{proof}
		\mods{	Developing }
	\begin{align*}
		C_{\Ncal}((i,j),(k,\ell)) 
					  &= \frac{1}{{\rb{N}}-1} \sum_{p=1}^{\rb{N}} (\alpha_{ij}^{(p)} + \alpha_{ji}^{(p)}) (\alpha_{k\ell}^{(p)} + \alpha_{\ell k}^{(p)})^{\top}
					  + (\beta_{ij}^{(p)} + \beta_{ji}^{(p)}) (\beta_{k\ell}^{(p)} + \beta_{\ell k}^{(p)})^{\top}  
					  \\
					  &\quad \modt{-} \frac{1}{({\rb{N}}-1){\rb{N}}}\sum_{\modt{p, p' =1}}^{\rb{N}} (\alpha_{ij}^{(p)} + \alpha_{ji}^{(p)}) (\alpha_{k\ell}^{(p')} + \alpha_{\ell k}^{(p')})^{\top}
					  + (\beta_{ij}^{(p)} + \beta_{ji}^{(p)}) (\beta_{k\ell}^{(p')} + \beta_{\ell k}^{(p')})^{\top},
	\end{align*}
	the equality~\eqref{eqn:qv-ncal} follows by noting that all the terms of the second line vanish since each coefficient depends linearly on $Y_t^{k,(p)}$ for some $k,p$, and $\sum_{p=1}^{\rb{N}} \mods{\widehat{Y}}_t^{k,(p)} = 0.$

	\mods{\Cref{eqn:qv-nt-tr} follows by using the relations}
	\begin{equation} \label{eqn:alpha-beta}
		\begin{aligned}
			\frac{1}{{\rb{N}}-1}\sum_{p=1}^{\rb{N}} \alpha_{ij}^{(p)} (\alpha_{k \ell}^{(p)})^{\top} &= (\Sigma_{\bU_t})_{ik} (\hP_{\hbY_t})_{j \ell},  \\
			\frac{1}{{\rb{N}}-1} \sum_{p=1}^{\rb{N}}  \beta_{ij}^{(p)} (\beta_{k \ell}^{(p)})^{\top} &= (\hP_{\hbY_t} S_{\bU_t} \hP_{\hbY_t})_{ik} (\hP_{\hbY_t})_{j\ell}.
		\end{aligned}
	\end{equation}
	\end{proof}
}

\begin{proof}[Proof of~\cref{lem:beta-gamma bounds}]
	Taking the trace of~\cref{eqn:d-sample-Gram} 
	\begin{equation*}
		\md \tr(\hP_{\hbY_t}) = \left( \tr\left(  ( A_{\bU_t} + A_{\bU_t}^{\top} ) \hP_{\hbY_t} \right) + \tr(\Sigma_{\bU_t}) - \tr(\hP_{\hbY_t} S_{\bU_t} \hP_{\hbY_t}) \right)\md t + \md \tilde{\Ncal_t},
	\end{equation*}
	where $\tilde{\Ncal_t} = \frac{1}{\sqrt{{\rb{N}}-1}} \sum_{i=1}^R (\Ncal_t)_{ii}$, and using~\cref{eqn:qv-ncal,eqn:alpha-beta} it holds
	\begin{align}
		\modt{\frac{\md}{\md t}} [\tilde{\Ncal_t}, \tilde{\Ncal_t}]_t &= \frac{\modt{4}}{{\rb{N}}-1} \sum_{ij} (\Sigma_{\bU_t})_{ij} (\hP_{\hbY_t})_{ij} + (\hP_{\hbY_t} S_{\bU_t} \hP_{\hbY_t})_{ij} (\hP_{\hbY_t})_{ij}) \nonumber
				         		 \\
									      &= \frac{\modt{4}}{{\rb{N}}-1} \left( \tr(\Sigma_{\bU_t} \hP_{\hbY_t}) + \tr(\hP_{\hbY_t} S_{\bU_t} \hP_{\hbY_t}^{\modt{2}}) \right) \nonumber
							 \\
									      &\leq \frac{\modt{4}}{{\rb{N}}-1} \left( \modt{\lambda_{\max}(\Sigma_{\bU_t})} \tr(\hP_{\hbY_t}) + \rho \tr(\hP_{\hbY_t}^{\modt{3}}) \right). \nonumber
							 \\
									      &\mods{\leq \frac{4}{{\rb{N}}-1} \tr(\hP_{\hbY_t})  \left( \lambda_{\max}(\Sigma_{\bU_t}) + \rho \tr(\hP_{\hbY_t})^{2}) \right).} \label{eqn:dNtilde-quad-var-bound}
	\end{align}
	Similarly,
	\begin{multline*}
		\tr\left(  ( A_{\bU_t} + A_{\bU_t}^{\top} ) \hP_{\hbY_t} \right) + \tr(\Sigma_{\bU_t}) - \tr(\hP_{\hbY_t} S_{\bU_t} \hP_{\hbY_t}) 
		\\
		\leq  
		\lambda_{\max}(A + A^{\top}) \tr(\hP_{\hbY_t}) + \tr(\Sigma_{\bU_t}) - \frac{\rho}{R} (\tr(\hP_{\hbY_t}))^2,
	\end{multline*}
	\modt{having used
		\begin{equation*}
			\tr(\hP_{\hbY_t} S_{\bU_t} \hP_{\hbY_t}) = \rho \norm{\hP_{\hbY_t}}^2_F 
			= \rho \left(  \sum_{i=1}^R \lambda_i(\hP_{\hbY_t})^2 \right) 
			\geq \frac{\rho}{R} \left(  \sum_{i=1}^R \lambda_i(\hP_{\hbY_t}) \right)^2
			\geq \frac{\rho}{R} \tr(\hP_{\hbY_t})^2.
		\end{equation*}
	}
	\noindent Therefore, the assumptions for~\cref{lem:fost-lyap} \mods{(a)} hold true with
	\begin{align*}
		\modt{\gamma} &= \modt{0} & 3\alpha &= \lambda_{\max}(A + A^{\top}) 	& \beta  &= \frac{\rho}{R} 	& r &= \tr(\Sigma_{\bU_t}) \\
			      & &        \tau_0  &= \frac{\mods{4\lambda_{\max}(\Sigma_{\bU_t})}}{{\rb{N}}-1}   & \tau_1 &= \modt{0}  & \tau_2 &= \modt{\frac{4 \rho}{{\rb{N}}-1}}.
	\end{align*}
	and the result follows.
\end{proof}

\modt{
\begin{proof}[Proof of~\cref{lem:trMYt-sup-bounds}]
	By the characterisation $\md \tr(\hP_{\hbY_t}) = a(t,\tr(\hP_{\hbY_t})) \md t + \tilde{\Ncal}_t$ in~\cref{lem:beta-gamma bounds},
	\begin{multline*}
		\tr(\hP_{\hbY_t}) \cdot a = \tr(\hP_{\hbY_t}) \cdot \left( \tr\left(  ( A_{\bU_t} + A_{\bU_t}^{\top} ) \hP_{\hbY_t} \right) + \tr(\Sigma_{\bU_t}) - \tr(\hP_{\hbY_t} S_{\bU_t} \hP_{\hbY_t}) \right)
		\\
		\leq
		(\lambda_{\max}(A + A^{\top}) + \frac{1}{2}) \tr(\hP_{\hbY_t})^2 + \frac{1}{2}\tr(\Sigma_{\bU_t})^2 - \frac{\rho}{R} (\tr(\hP_{\hbY_t}))^3,
	\end{multline*}
	and the quadratic variation \mods{is bounded as in~\cref{eqn:dNtilde-quad-var-bound}.}
	%\begin{multline*}
	%	\frac{1}{2} \frac{\md}{\md t}[\tilde{\Ncal_t}, \tilde{\Ncal_t}]_t  \leq \frac{1}{P-1} \left( \modt{\lambda_{\max}(\Sigma_{\bU_t})} \tr(\hP_{\hbY_t}) + \rho \tr(\hP_{\hbY_t}^{\modt{3}}) \right) 
	%	\\
	%	\leq \frac{1}{P-1}\left( \frac{1}{2} \modt{\lambda_{\max}(\Sigma_{\bU_t})^2} + \frac{1}{2}\tr(\hP_{\hbY_t})^2 + \rho \tr(\hP_{\hbY_t})^3 \right).
	%\end{multline*}
	\mods{Consequently,}
	\begin{multline*}
		\mods{\tr(\hP_{\hbY_t}) \cdot a(t, \tr(\hP_{\hbY_t})) + \frac{1}{2} \frac{\md }{\md t} [\tilde{\Ncal_t}, \tilde{\Ncal_t}]_t}  \\
		\mods{\leq} (\lambda_{\max}(A + A^{\top}) + 1) \tr(\hP_{\hbY_t})^2 + \mods{\tr(\Sigma_{\bU_t})^2} + \rho \left(\frac{\mods{2}}{{\rb{N}}-1} - \frac{1}{R} \right) \tr(\hP_{\hbY_t})^3 
		\\
		\leq 
		K(1 + \tr(\hP_{\hbY_t})^2),
	\end{multline*}
	\mods{having used $\lambda_{\max}(\Sigma_{\bU_t}) \leq \tr(\Sigma_{\bU_t})$ and an $\varepsilon$-Young inequality in the first line, and the assumption \modr{$4R + 1 < {\rb{N}}$} in both lines}.
	Next, derive
	\begin{multline*}
		\tr(\hP_{\hbY_t})^2 
		=  \tr(\hP_{\hbY_0})^2 + 2\int_0^t \left( \tr(\hP_{\hbY_s}) \cdot a+  \frac{1}{2} \md [\mods{\tilde{\Ncal}_s},\tilde{\Ncal}_s]_s \right) \md s + 2 \int_0^t \tr(\hP_{\hbY_s}) \md \tilde{\Ncal}_s
		\\
		\leq \tr(\hP_{\hbY_0})^2 + 2 K \int_0^t \left(1 +  \tr(\hP_{\hbY_s})^2 \right) \md s + 2 \int_0^t \tr(\hP_{\hbY_s}) \md \tilde{\Ncal}_s.
	\end{multline*}
	Taking the supremum over $[0,t]$ and then the expectation, 
	\begin{multline*}
		\E [\sup_{0 \leq s \leq t} \tr(\hP_{\hbY_t})^2] 
		\\
		\leq 
		\left( \E[ \tr(\hP_{\hbY_0})^2 ] + \mods{2Kt} \right)
		+ 2 K \int_0^t \E[\sup_{0 \leq s \leq t} \tr(\hP_{\hbY_s})^2 ] \md s
		+ 2 \E\left[\sup_{0 \leq s \leq t} \left| \int_{0}^s  \tr(\hP_{\hbY_s}) \md \tilde{\Ncal}_s \right|\right]
		\\
		\leq 
		C(t, \E[ \tr(\hP_{\hbY_0})^2 ]) 
		+ 2 \mods{K} \int_0^t \E[\sup_{0 \leq s \leq t} \tr(\hP_{\hbY_s})^2 ] \md s
		\\
		+ 8 \E \left[ \left( \int_0^t \tr(\hP_{\hbY_s})^2 \mods{\md} [\tilde{\Ncal}_s, \tilde{\Ncal}_s]_s \md s \right)^{\oh} \right].
	\end{multline*}
	To bound the last term, notice that 
	\begin{align*}
		8 \E \left[ \left( \int_0^t \tr(\hP_{\hbY_s})^2 \mods{\md }[\tilde{\Ncal}_s, \tilde{\Ncal}_s]_s \md s \right)^{\oh} \right] 
		&\leq 
		8 \E \left[ \left(\sup_{0 \leq s \leq t} \tr(\hP_{\hbY_s})\right) \cdot \left( \int_0^t \mods{\frac{\md}{\md s}} [\tilde{\Ncal}_s, \tilde{\Ncal}_s]_s \md s \right)^{\oh} \right] 
		\\
		%&\leq 
		%8 \E \left[ \left(\sup_{0 \leq s \leq t} \tr(\hP_{\hbY_s})\right) \cdot \left( \int_0^t [\tilde{\Ncal}_s, \tilde{\Ncal}_s]_s \md s \right)^{\oh} \right] 
		%\\
		&\leq
	\frac{1}{2}\E \left[ \sup_{0 \leq s \leq t} \tr(\hP_{\hbY_s})^2 \right] 
	\\
		&\quad+ \frac{\mods{32}}{{\rb{N}}-1}  \int_0^t  \left( \modt{\lambda_{\max}(\Sigma_{\bU_s})} \E[ \tr(\hP_{\hbY_s})] + \rho \E [\tr(\hP_{\hbY_s})^{\mods{3}}] \right) \md s.
	\end{align*}
	\mods{C}ombining the two previous terms yields
	\begin{multline*}		
		\E [\sup_{0 \leq s \leq t} \tr(\hP_{\hbY_s})^2] 
	\leq C\left(t, \lambda_{\max}(\Sigma), \rho, \E[ \tr(\hP_{\hbY_0})^2 ], \int_0^t \E[ \tr(\hP_{\hbY_s})^3 ] \md s \right)
	\\
		+ 2 K \int_0^t \E[\sup_{0 \leq s \leq t} \tr(\hP_{\hbY_s})^2 ] \md s,
	\end{multline*}
	and since the terms involving $\E[\tr(\hP_{\hbY_s})^3]$ are bounded by~\cref{lem:beta-gamma bounds}, the claim follows by applying a Gronwall lemma~\cite{NumericalApproQuarte1994}.
\end{proof}
}

\begin{proof}[Proof of \cref{prop:vec-Yt-bounds}]
	\rb{It holds
	\begin{equation*}
		\E[\sup_{0 \leq s \leq t} \norm{\mvec(\hbY_s)}^2] 
		= 
		({\rb{N}}-1) \E[\sup_{0 \leq s \leq t} \tr(\hP_{\hbY_s})] 
		\leq ({\rb{N}}-1) \E[\sup_{0 \leq s \leq t} \tr(\hP_{\hbY_s})^2]^{\oh},
	\end{equation*}
	which is bounded by~\cref{lem:trMYt-sup-bounds}.
	Furthermore, 
	\begin{equation*}
		\bbP \left( \sup_{0 \leq s \leq t} \norm{\mvec(\hbY_s)} > n \right) 
		\leq 
		\frac{\E[\sup_{0 \leq s \leq t} \norm{\mvec(\hbY_s)}^2]}{n^2},
	\end{equation*}
implying (by the Borell-Cantelli lemma) the a.s. boundedness of $\norm{\mvec(\hbY_t)}$.}
\end{proof}

\begin{proof}[Proof of~\cref{lem:dEt-bound}]
	\modt{Define} $E_t = \hP_{\hbY_t} - M_{\bY_t}$. 
	It holds
	\begin{align*}
		\md E_t  
			&= \left( A_{\bU_t} E_t + E_t A_{\bU_t}^{\top} - \hP_{\hbY_t} S_{\bU_t} \hP_{\hbY_t} + M_{\bY_t} S_{\bU_t} M_{\bY_t} \right) \md t + \frac{1}{\sqrt{{\rb{N}}-1}} \md \Ncal_t
			\\
			&= \left( (A_{\bU_t} - M_{\bY_t} S_{\bU_t}) E_t + E_t (A_{\bU_t} - \hP_{\hbY_t} S_{\bU_t} )^{\top} \right)\md t + \frac{1}{\sqrt{{\rb{N}}-1}} \md \Ncal_t.
	\end{align*}
	Consequently, 
	\begin{align} \label{eqn:dEtEt}
		\md (E_t^{\top} E_t) &= \md E_t^{\top} E_t + E_t^{\top} \md E_t + \md [E^{\top}_t, E_t]_t \nonumber \\
				     &= E_t^{\top} (A_{\bU_t} + A_{\bU_t}^{\top} - \mods{ M_{\bY_t} S_{\bU_t} - S_{\bU_t} M_{\bY_t} }) E_t \md t
				     + \md [E^{\top}_t, E_t]_t \nonumber
		\\
				     &\quad + (A_{\bU_t} - \hP_{\hbY_t} S_{\bU_t} )E_t^{\top} E_t   \md t
				     + E_t^{\top}E_t (A_{\bU_t} - \hP_{\hbY_t} S_{\bU_t})^{\top} \md t \nonumber
				     \\
				     &\quad
				     \modt{+ \frac{1}{\sqrt{{\rb{N}}-1}} \left( E_t^{\top} \md \Ncal_t + \md \Ncal_t^{\top} E_t \right)}
	\end{align}
	where in \modt{a} slight abuse of notation
	\begin{align*}
		\md [E^{\top}_t, E_t]_t(i,j) = \frac{1}{{\rb{N}}-1}\sum_{k=1}^R C_{\Ncal} ((\modt{k,i}),(k,j))\modt{,}
	\end{align*}
	\modt{where~$C_{\Ncal}$ is defined via~\cref{eqn:qv-ncal}.}
	Therefore, taking the trace of~\cref{eqn:dEtEt},
	\begin{align*}
	\md \norm{E_t}^2_{\modt{F}} &= \tr\left((A_{\bU_t} + A_{\bU_t}^{\top} - \mods{M_{\bY_t} S_{\bU_t} - S_{\bU_t} M_{\bY_t}}) E_t^{\top} E_t\right) \mods{\md t} 
\\
				    &\quad + \tr((A_{\bU_t} + A_{\bU_t}^{\top} -  \mods{\hP_{\hbY_t} S_{\bU_t} -  S_{\bU_t} \hP_{\hbY_t}}) E_t^{\top} E_t) \md t
					  \\& \quad+ \sum_{i=1}^R \sum_{k=1}^R C_{\Ncal} ((i,k),(k,i)) \md t 
					  + \frac{2}{\sqrt{{\rb{N}}-1}} \md \mathcal{E}_t 
					  \\
					    &=  \tr\left( \left( \mods{2(A_{\bU_t} + A_{\bU_t}^{\top})} - (M_{\bY_t} + \hP_{\hbY_t}) S_{\bU_t} - \mods{S_{\bU_t} (M_{\bY_t} + \hP_{\hbY_t})}\right) E_t^{\top} E_t \right) \md t 
					    + \frac{2}{\sqrt{{\rb{N}}-1}} \md \mathcal{E}_t
					    \\
					    & \quad + \frac{2}{{\rb{N}}-1} \left ( \tr\left( (\Sigma_{\bU_t} + \hP_{\hbY_t} S_{\bU_t} \hP_{\hbY_t}) \hP_{\hbY_t}\right) + \tr(\Sigma_{\bU_t} + \hP_{\hbY_t} S_{\bU_t} \hP_{\hbY_t} ) \tr(\hP_{\hbY_t})\right) \md t,
	\end{align*}
	\modt{having used~\cref{eqn:qv-nt-tr} to expand the double sum on $C_{\Ncal}$.}
	Finally, \mods{denoting $\tilde{A} = A_{\bU_t} + A_{\bU_t}^{\top}$ and $\tilde{B} = (M_{\bY_t} + \hP_{\hbY_t}) S_{\bU_t} - S_{\bU_t} (M_{\bY_t} + \hP_{\hbY_t})$, 
	the matrix $2(A_{\bU_t} + A_{\bU_t}^{\top}) - (M_{\bY_t} + \hP_{\hbY_t}) S_{\bU_t} - S_{\bU_t} (M_{\bY_t} + \hP_{\hbY_t}) = 2 \tilde{A} - \tilde{B}$ is symmetric}
	hence	
	\begin{equation*}
		\lambda_{\max}(\mods{2 \tilde{A} - \tilde{B}}) \leq \mods{2}\lambda_{\max}(A_{\bU_t} + A_{\bU_t}^{\top}) \leq \mods{2}\lambda_{\max}(A + A^{\top}),
	\end{equation*}
	the former inequality \mods{being} given by Weyl's inequality.
	Using this to upper-bound	
	\begin{align*}
		\tr\left((\mods{2 \tilde{A} - \tilde{B}}) E_t^{\top} E_t \right)  \leq \lambda_{\max} (\mods{2 \tilde{A} - \tilde{B}}) \norm{E_t}^2_F
	\end{align*}
	yields the first result. The quadratic variation of $\Ecal_t$ is computed by using the expression~\eqref{eqn:qv-ncal} for the quadratic variation of the martingale $\Ncal_t$ defined in~\cref{eqn:Ncal-ij-detail},
	%\begin{multline*}
	%	\frac{1}{P-1} \sum_{i,j}^R \sum_{k,\ell}^R \sum_{p,p'}^P \sum_{m,m'}^R E_{ij} E_{k\ell} [(\alpha_{ij}^{(p)} + \alpha_{ji}^{(p)})(m)]  [(\alpha_{k\ell}^{(p')} + \alpha_{\ell k}^{(p')})(m')] \md [\hWcal_t^{\star,(p)}(m), \hWcal_t^{\star,(p')}(m')]_t	
	%	\\
	%	= \frac{1}{P-1} \sum_{i,j}^R \sum_{k,\ell}^R \sum_{p,p'}^P  E_{ij} E_{k\ell} (\alpha_{ij}^{(p)} + \alpha_{ji}^{(p)}) (\alpha_{k\ell}^{(p')} + \alpha_{\ell k}^{(p')})^{\top} \left( \left(1 - \frac{1}{P} \right) \delta_{p,p'} + \frac{\delta_{p\neq p'}}{P}  \right)
	%	\\
	%	= \frac{1}{P-1} \sum_{i,j}^R \sum_{k,\ell}^R E_{ij} E_{k\ell} \sum_{\modt{p=1}}^P (\alpha_{ij}^{(p)} + \alpha_{ji}^{(p)}) (\alpha_{k\ell}^{(p)} + \alpha_{\ell k}^{(p)})^{\top} 
	%	\\
	%	= \sum_{i,j}^R \sum_{k,\ell}^R E_{ij} E_{k\ell} \left( (\Sigma_{\bU_t})_{ik} (\hP_{\hbY_t})_{j \ell} + 
	%		(\Sigma_{\bU_t})_{i \ell} (\hP_{\hbY_t})_{j k}  + 
	%		(\Sigma_{\bU_t})_{jk} (\hP_{\hbY_t})_{i \ell} +
	%		(\Sigma_{\bU_t})_{j \ell} (\hP_{\hbY_t})_{i k} 
	%	\right)
	%	 \\ 
	%	 = 4\tr(E_t \hP_{\hbY_t} E_t \Sigma_{\bU_t}), 
	%\end{multline*}
	\begin{align*}
		\mods{
		\frac{\md}{\md t} [\Ecal_t, \Ecal_t]_t }
		&= \mods{\sum_{i,j}^R \sum_{k,\ell}^R E_{ij} E_{kl}  \frac{\md}{\md t} [(\Ncal_t)_{ij}, (\Ncal_t)_{k \ell}]_t}
			\\
		&=\mods{ \frac{1}{{\rb{N}}-1} \sum_{i,j}^R \sum_{k,\ell}^R E_{ij} E_{kl} \left( \sum_{p=1}^{\rb{N}} (\alpha_{ij}^{(p)} + \alpha_{ji}^{(p)}) (\alpha_{k\ell}^{(p)} + \alpha_{\ell k}^{(p)})^{\top} 
					  + (\beta_{ij}^{(p)} + \beta_{ji}^{(p)}) (\beta_{k\ell}^{(p)} + \beta_{\ell k}^{(p)})^{\top} \!\! \right) }
		\\
		&= \sum_{i,j}^R \sum_{k,\ell}^R E_{ij} E_{k\ell} \left( (\Sigma_{\bU_t})_{ik} (\hP_{\hbY_t})_{j \ell} + 
			(\Sigma_{\bU_t})_{i \ell} (\hP_{\hbY_t})_{j k}  + 
			(\Sigma_{\bU_t})_{jk} (\hP_{\hbY_t})_{i \ell} +
			(\Sigma_{\bU_t})_{j \ell} (\hP_{\hbY_t})_{i k} 
		\right)
		\\
		&\hspace{1em}\mods{+
\sum_{i,j}^R \sum_{k,\ell}^R E_{ij} E_{k\ell} 
		\bigg(  (\hP_{\hbY_t} S_{\bU_t} \hP_{\hbY_t})_{ik} (\hP_{\hbY_t})_{j \ell} + 
		(\hP_{\hbY_t} S_{\bU_t} \hP_{\hbY_t})_{i \ell} (\hP_{\hbY_t})_{j k} }
			\\
		&\hspace{10em}
			\mods{+ (\hP_{\hbY_t} S_{\bU_t} \hP_{\hbY_t})_{jk} (\hP_{\hbY_t})_{i \ell} +
			(\hP_{\hbY_t} S_{\bU_t} \hP_{\hbY_t})_{j \ell} (\hP_{\hbY_t})_{i k} 
		\bigg)}	
		 \\ 
		&= 4\tr(E_t \hP_{\hbY_t} E_t \Sigma_{\bU_t}) + \mods{4 \tr(E_t \hP_{\hbY_t}  E_t \hP_{\hbY_t} S_{\bU_t} \hP_{\hbY_t} )} \\
					 &= 4\tr\left(E_t \hP_{\hbY_t} E_t ( \Sigma_{\bU_t} + \hP_{\hbY_t} S_{\bU_t} \hP_{\hbY_t} ) \right),
	\end{align*}
	\mods{having used the relations~\eqref{eqn:alpha-beta} in the third equality.}
	%and, performing the same computation for the second term (with the $\beta$-coefficients) yields 
	%\begin{align*}
	%	\mods{\frac{\md}{\md t}} [\Ecal_t, \Ecal_t]_t &= 4\tr(E_t \hP_{\hbY_t} E_t \Sigma_{\bU_t})  + 4 \tr(E_t \hP_{\hbY_t}  E_t \hP_{\hbY_t} S_{\bU_t} \hP_{\hbY_t} ) \\
	%				 &= 4\tr\left(E_t \hP_{\hbY_t} E_t ( \Sigma_{\bU_t} + \hP_{\hbY_t} S_{\bU_t} \hP_{\hbY_t} ) \right).  
	%\end{align*}
\end{proof}

\begin{proof}[Proof of~\cref{lem:bounded-moments-xdlr}]
	\begin{multline*}
		\md \norm{\Xcal_t}^2 = \bigg\{ \Xcal_t^{\top} (A + A^{\top}) \Xcal_t  
			\modt{+ 2 \Xcal_t^{\top} f }
			+  \modt{(\Xcal_t^{\top} P_t S \Xcal_t^{\signal} + (\Xcal_t^{\signal})^{\top} P_t S \Xcal_t ) } 
			\\
			\modt{- \Xcal_t^{\top} (P_t S + S P_t) \Xcal_t }
		+ \tr(\Pi_{\bU_t} \Sigma \Pi_{\bU_t})  + \modt{2} \tr(P_t S P_t)  \bigg\} \modt{\md t}
		\\
		+ \big\{ \modt{2} \Xcal_t^{\top} \Pi_{\bU_t} \Soh \md \Wcal_t + \modt{2} \Xcal_t^{\top} P_t H^{\top} \Gmoh \md ( \widetilde{\Vcal}_t - \md \Vcal_t) \big\}
		\modt{= \Lcal_t \md t + \md \Mcal_t}
	\end{multline*}
	hence  
	\begin{multline*}
		\Lcal_t 
		\leq \modt{2 \norm{f} \norm{\Xcal_t} + }  \frac{1}{2}\lambda_{\max}(A + A^{\top}) \norm{\Xcal_t}^2 + \left( \left( \frac{2 \norm{S \Xcal_t^{\signal}}^2}{|\lambda_{\max}(A + A^{\top})|} + \modt{2} \rho \right) \norm{P_t}^2_F + \tr(\Sigma) \right)
		\\
		\leq \modt{2 \norm{f} \norm{\Xcal_t} + } \frac{1}{2}\lambda_{\max}(A + A^{\top}) \norm{\Xcal_t}^2 + \left(  \frac{\modt{4} \norm{S \Xcal_t^{\signal}}^4}{|\lambda_{\max}(A + A^{\top})|^2} + \modt{4} \rho^2  + \modt{\frac{\norm{P_t}^4_F}{2}} + \tr(\Sigma) \right)
	\end{multline*}
	having used ($\varepsilon$-)Young inequalities and \rb{$\md [\Mcal_t, \Mcal_t] \leq \left( \modt{4} \tr(\Sigma) + \modt{8} \rho \norm{P_t}^2_F  \right) \norm{\Xcal_t}^2$.}
	%\begin{equation*}
	%	\md [\Mcal_t, \Mcal_t] \leq \left( \modt{4} \tr(\Sigma) + \modt{8} \rho \norm{P_t}^2_F  \right) \norm{\Xcal_t}^2.
	%\end{equation*}
	\rb{\cref{lem:gram-ex-bounds} guarantees $\sup_{t \geq 0}\norm{P_t}^n_F < \infty$ and since $\sup_{t \geq 0} \E [ \norm{S \Xcal_t}^n ] < \infty$, \cref{lem:fost-lyap} yields the result.}
\end{proof}

\section*{Acknowledgments}
The authors acknowledge the use of AI-assisted tools for minor editing and language polishing.

\bibliographystyle{siamplain}
\bibliography{references} 

@article{DataAssimilatiPt2Sonder2013,
  author = {Sondergaard, Thomas and Lermusiaux, Pierre F. J.},
  doi = {10.1175/mwr-d-11-00296.1},
  issue = {6},
  journal = {Monthly Weather Review},
  month = {6},
  pages = {1761--1785},
  publisher = {American Meteorological Society},
  title = {{Data} {Assimilation} with {Gaussian} {Mixture} {Models} {Using} the {Dynamically} {Orthogonal} {Field} {Equations}. {Part} {II}: {Applications}},
  url = {http://dx.doi.org/10.1175/mwr-d-11-00296.1},
  volume = {141},
  year = {2013},
}

@article{DataAssimilatiSonder2013,
  author = {Sondergaard, Thomas and Lermusiaux, Pierre F. J.},
  doi = {10.1175/mwr-d-11-00295.1},
  issue = {6},
  journal = {Monthly Weather Review},
  month = {6},
  pages = {1737--1760},
  publisher = {American Meteorological Society},
  title = {{Data} {Assimilation} with {Gaussian} {Mixture} {Models} {Using} the {Dynamically} {Orthogonal} {Field} {Equations}. {Part} {I}: {Theory} and {Scheme}},
  url = {http://dx.doi.org/10.1175/mwr-d-11-00295.1},
  volume = {141},
  year = {2013},
}

@article{DynamicalLowRKoch2007,
  author = {Koch, Othmar and Lubich, Christian},
  doi = {10.1137/050639703},
  issue = {2},
  journal = {SIAM Journal on Matrix Analysis and Applications},
  language = {en},
  month = {1},
  pages = {434--454},
  publisher = {Society for Industrial \& Applied Mathematics (SIAM)},
  title = {{Dynamical} {Low}-{Rank} {Approximation}},
  url = {http://dx.doi.org/10.1137/050639703},
  volume = {29},
  year = {2007},
}

@article{AnEnsemBuehner2017,
  author = {Buehner, Mark and McTaggart-Cowan, Ron and Heilliette, Sylvain},
  doi = {10.1175/mwr-d-16-0106.1},
  issue = {2},
  journal = {Monthly Weather Review},
  month = {2},
  pages = {617--635},
  publisher = {American Meteorological Society},
  title = {{An} {Ensemble} {Kalman} {Filter} for {Numerical} {Weather} {Prediction} {Based} on {Variational} {Data} {Assimilation}: {VarEnKF}},
  url = {https://doi.org/10.1175/mwr-d-16-0106.1},
  volume = {145},
  year = {2017},
}

@inproceedings{OnANewLowRYamada2021,
  author = {Yamada, Shuto and Ohki, Kentaro},
  booktitle = {{Proceedings} of the {ISCIE} {International} {Symposium} on {Stochastic} {Systems} {Theory} and its {Applications}},
  organization = {The ISCIE Symposium on Stochastic Systems Theory and Its Applications},
  pages = {16--20},
  title = {{On} a new low-rank {Kalman}-{Bucy} filter and its convergence property},
  volume = {2021},
  year = {2021},
}

@article{MultilevelDaBeiser2025,
  author = {Beiser, Florian and Holm, H\aa{}vard Heitlo and Lye, Kjetil Olsen and Eidsvik, Jo},
  doi = {10.1016/j.jcp.2025.113722},
  journal = {Journal of Computational Physics},
  language = {en},
  month = {3},
  pages = {113722},
  publisher = {Elsevier BV},
  title = {{Multi}-level data assimilation for ocean forecasting using the shallow-water equations},
  url = {https://doi.org/10.1016/j.jcp.2025.113722},
  volume = {524},
  year = {2025},
}

@article{ErrorAnalysisMushar2015,
  author = {Musharbash, Eleonora and Nobile, Fabio and Zhou, Tao},
  doi = {10.1137/140967787},
  issue = {2},
  journal = {SIAM Journal on Scientific Computing},
  language = {en},
  month = {1},
  pages = {A776--A810},
  publisher = {Society for Industrial \& Applied Mathematics (SIAM)},
  title = {{Error} {Analysis} of the {Dynamically} {Orthogonal} {Approximation} of {Time} {Dependent} {Random} {PDEs}},
  url = {http://dx.doi.org/10.1137/140967787},
  volume = {37},
  year = {2015},
}

@article{ConcreteEnsembKelly2015,
  author = {Kelly, David and Majda, Andrew J. and Tong, Xin T.},
  doi = {10.1073/pnas.1511063112},
  issue = {34},
  journal = {Proceedings of the National Academy of Sciences},
  language = {en},
  month = {8},
  pages = {10589--10594},
  publisher = {Proceedings of the National Academy of Sciences},
  title = {{Concrete} ensemble {Kalman} filters with rigorous catastrophic filter divergence},
  url = {https://doi.org/10.1073/pnas.1511063112},
  volume = {112},
  year = {2015},
}

@article{BlendedPaMajda2014,
  author = {Majda, Andrew J. and Qi, Di and Sapsis, Themistoklis P.},
  doi = {10.1073/pnas.1405675111},
  issue = {21},
  journal = {Proceedings of the National Academy of Sciences},
  language = {en},
  month = {5},
  pages = {7511--7516},
  publisher = {Proceedings of the National Academy of Sciences},
  title = {{Blended} particle filters for large-dimensional chaotic dynamical systems},
  url = {https://doi.org/10.1073/pnas.1405675111},
  volume = {111},
  year = {2014},
}

@article{ARankAdaptiveCeruti2022,
  author = {Ceruti, Gianluca and Kusch, Jonas and Lubich, Christian},
  doi = {10.1007/s10543-021-00907-7},
  issue = {4},
  journal = {BIT Numerical Mathematics},
  language = {en},
  month = {12},
  pages = {1149--1174},
  publisher = {Springer Science and Business Media LLC},
  title = {{A} rank-adaptive robust integrator for dynamical low-rank approximation},
  url = {https://doi.org/10.1007/s10543-021-00907-7},
  volume = {62},
  year = {2022},
}

@article{DynamicallyOrtSapsis2009,
  author = {Sapsis, Themistoklis P. and Lermusiaux, Pierre F.J.},
  doi = {10.1016/j.physd.2009.09.017},
  issue = {23-24},
  journal = {Physica D: Nonlinear Phenomena},
  language = {en},
  month = {12},
  pages = {2347--2360},
  publisher = {Elsevier BV},
  title = {{Dynamically} orthogonal field equations for continuous stochastic dynamical systems},
  url = {http://dx.doi.org/10.1016/j.physd.2009.09.017},
  volume = {238},
  year = {2009},
}

@article{MultilevelEnseCherno2021,
  author = {Chernov, Alexey and Hoel, H\aa{}kon and Law, Kody J. H. and Nobile, Fabio and Tempone, Raul},
  doi = {10.1007/s00211-020-01159-3},
  issue = {1},
  journal = {Numerische Mathematik},
  language = {en},
  month = {1},
  pages = {71--125},
  publisher = {Springer Science and Business Media LLC},
  title = {{Multilevel} ensemble {Kalman} filtering for spatio-temporal processes},
  url = {https://doi.org/10.1007/s00211-020-01159-3},
  volume = {147},
  year = {2021},
}

@article{SequentialDataEvense1994,
  author = {Evensen, Geir},
  doi = {10.1029/94jc00572},
  issue = {C5},
  journal = {Journal of Geophysical Research: Oceans},
  language = {en},
  month = {5},
  pages = {10143--10162},
  publisher = {American Geophysical Union (AGU)},
  title = {{Sequential} data assimilation with a nonlinear quasi-geostrophic model using {Monte} {Carlo} methods to forecast error statistics},
  url = {https://doi.org/10.1029/94jc00572},
  volume = {99},
  year = {1994},
}

@article{AnAdaptiveSilva2025,
  author = {Silva, Francesco A. B. and Pagliantini, Cecilia and Veroy, Karen},
  doi = {10.1137/24m1653690},
  issue = {1},
  journal = {SIAM/ASA Journal on Uncertainty Quantification},
  language = {en},
  month = {3},
  pages = {140--170},
  publisher = {Society for Industrial \& Applied Mathematics (SIAM)},
  title = {{An} {Adaptive} {Hierarchical} {Ensemble} {Kalman} {Filter} with {Reduced} {Basis} {Models}},
  url = {https://doi.org/10.1137/24m1653690},
  volume = {13},
  year = {2025},
}

@article{OnTheRelaxHa2018,
  author = {Ha, Seung-Yeal and Ko, Dongnam and Ryoo, Seung-Yeon},
  doi = {10.1007/s10955-018-2091-0},
  issue = {5},
  journal = {Journal of Statistical Physics},
  language = {en},
  month = {9},
  pages = {1427--1478},
  publisher = {Springer Science and Business Media LLC},
  title = {{On} the {Relaxation} {Dynamics} of {Lohe} {Oscillators} on {Some} {Riemannian} {Manifolds}},
  url = {https://doi.org/10.1007/s10955-018-2091-0},
  volume = {172},
  year = {2018},
}

@article{OnNumericalPrLiJi2008,
  author = {Li, Jia and Xiu, Dongbin},
  doi = {10.1016/j.cma.2008.03.022},
  issue = {43-44},
  journal = {Computer Methods in Applied Mechanics and Engineering},
  language = {en},
  month = {8},
  pages = {3574--3583},
  publisher = {Elsevier BV},
  title = {{On} numerical properties of the ensemble {Kalman} filter for data assimilation},
  url = {https://doi.org/10.1016/j.cma.2008.03.022},
  volume = {197},
  year = {2008},
}

@article{AStronglyCBlomker2018,
  author = {Bl\"{o}mker, Dirk and Schillings, Claudia and Wacker, Philipp},
  doi = {10.1137/17m1132367},
  issue = {4},
  journal = {SIAM Journal on Numerical Analysis},
  language = {en},
  month = {1},
  pages = {2537--2562},
  publisher = {Society for Industrial \& Applied Mathematics (SIAM)},
  title = {{A} {Strongly} {Convergent} {Numerical} {Scheme} from {Ensemble} {Kalman} {Inversion}},
  url = {https://doi.org/10.1137/17m1132367},
  volume = {56},
  year = {2018},
}

@article{DualDynamicallMushar2018,
  author = {Musharbash, Eleonora and Nobile, Fabio},
  doi = {10.1016/j.jcp.2017.09.061},
  journal = {Journal of Computational Physics},
  language = {en},
  month = {2},
  pages = {135--162},
  publisher = {Elsevier BV},
  title = {{Dual} {Dynamically} {Orthogonal} approximation of incompressible {Navier} {Stokes} equations with random boundary conditions},
  url = {http://dx.doi.org/10.1016/j.jcp.2017.09.061},
  volume = {354},
  year = {2018},
}

@article{TheRankRedSchmidt2023,
  author = {Schmidt, Jonathan and Hennig, Philipp and Nick, J\"{o}rg and Tronarp, Filip},
  journal = {Advances in Neural Information Processing Systems},
  pages = {61364--61376},
  title = {{The} rank-reduced {Kalman} filter: {Approximate} dynamical-low-rank filtering in high dimensions},
  volume = {36},
  year = {2023},
}

@article{AMollifiedEnsBergem2010,
  author = {Bergemann, Kay and Reich, Sebastian},
  doi = {10.1002/qj.672},
  issue = {651},
  journal = {Quarterly Journal of the Royal Meteorological Society},
  language = {en},
  month = {7},
  pages = {1636--1643},
  publisher = {Wiley},
  title = {{A} mollified ensemble {Kalman} filter},
  url = {https://doi.org/10.1002/qj.672},
  volume = {136},
  year = {2010},
}

@article{AProjectorSplLubich2014,
  author = {Lubich, Christian and Oseledets, Ivan V.},
  doi = {10.1007/s10543-013-0454-0},
  issue = {1},
  journal = {BIT Numerical Mathematics},
  language = {en},
  month = {3},
  pages = {171--188},
  publisher = {Springer Science and Business Media LLC},
  title = {{A} projector-splitting integrator for dynamical low-rank approximation},
  url = {http://dx.doi.org/10.1007/s10543-013-0454-0},
  volume = {54},
  year = {2014},
}

@article{HighDimensionaSunH2024,
  author = {Sun, Hao-Xuan and Wang, Shouxia and Zheng, Xiaogu and Chen, Song Xi},
  doi = {10.1002/qj.4846},
  issue = {765},
  journal = {Quarterly Journal of the Royal Meteorological Society},
  language = {en},
  month = {10},
  pages = {4870--4884},
  publisher = {Wiley},
  title = {{High}-dimensional ensemble {Kalman} filter with localization, inflation, and iterative updates},
  url = {https://doi.org/10.1002/qj.4846},
  volume = {150},
  year = {2024},
}

@article{ADetermSakov2008,
  author = {Sakov, Pavel and Oke, Peter R.},
  doi = {10.1111/j.1600-0870.2007.00299.x},
  issue = {2},
  journal = {Tellus A: Dynamic Meteorology and Oceanography},
  month = {1},
  pages = {361},
  publisher = {Stockholm University Press},
  title = {{A} deterministic formulation of the ensemble {Kalman} filter: an alternative to ensemble square root filters},
  url = {https://doi.org/10.1111/j.1600-0870.2007.00299.x},
  volume = {60},
  year = {2008},
}

@article{AnUnconventionCeruti2022,
  author = {Ceruti, Gianluca and Lubich, Christian},
  doi = {10.1007/s10543-021-00873-0},
  issue = {1},
  journal = {BIT Numerical Mathematics},
  language = {en},
  month = {3},
  pages = {23--44},
  publisher = {Springer Science and Business Media LLC},
  title = {{An} unconventional robust integrator for dynamical low-rank approximation},
  url = {http://dx.doi.org/10.1007/s10543-021-00873-0},
  volume = {62},
  year = {2022},
}

@article{OnTheMathematBishop2023,
  author = {Bishop, Adrian N. and Del Moral, Pierre},
  doi = {10.1007/s00498-023-00357-2},
  issue = {4},
  journal = {Mathematics of Control, Signals, and Systems},
  language = {en},
  month = {12},
  pages = {835--903},
  publisher = {Springer Science and Business Media LLC},
  title = {{On} the mathematical theory of ensemble (linear-{Gaussian}) {Kalman}\textendash{}{Bucy} filtering},
  url = {http://dx.doi.org/10.1007/s00498-023-00357-2},
  volume = {35},
  year = {2023},
}

@article{NonlinearModelChatur2010,
  author = {Chaturantabut, Saifon and Sorensen, Danny C.},
  doi = {10.1137/090766498},
  issue = {5},
  journal = {SIAM Journal on Scientific Computing},
  language = {en},
  month = {1},
  pages = {2737--2764},
  publisher = {Society for Industrial \& Applied Mathematics (SIAM)},
  title = {{Nonlinear} {Model} {Reduction} via {Discrete} {Empirical} {Interpolation}},
  url = {https://doi.org/10.1137/090766498},
  volume = {32},
  year = {2010},
}

@article{AHybridParticGrooms2021,
  author = {Grooms, Ian and Robinson, Gregor},
  doi = {10.1371/journal.pone.0248266},
  issue = {3},
  journal = {PLOS ONE},
  language = {en},
  month = {3},
  pages = {e0248266},
  publisher = {Public Library of Science (PLoS)},
  title = {{A} hybrid particle-ensemble {Kalman} filter for problems with medium nonlinearity},
  url = {https://doi.org/10.1371/journal.pone.0248266},
  volume = {16},
  year = {2021},
}

@article{OnDimensionSolonen2016,
  author = {Solonen, Antti and Cui, Tiangang and Hakkarainen, Janne and Marzouk, Youssef},
  doi = {10.1088/0266-5611/32/4/045003},
  issue = {4},
  journal = {Inverse Problems},
  month = {4},
  pages = {045003},
  publisher = {IOP Publishing},
  title = {{On} dimension reduction in {Gaussian} filters},
  url = {https://doi.org/10.1088/0266-5611/32/4/045003},
  volume = {32},
  year = {2016},
}

@article{GlobalAnaWei1994,
  author = {Wei-Yong Yan, None and Helmke, U. and Moore, J.B.},
  doi = {10.1109/72.317720},
  issue = {5},
  journal = {IEEE Transactions on Neural Networks},
  pages = {674--683},
  publisher = {Institute of Electrical and Electronics Engineers (IEEE)},
  title = {{Global} analysis of {Oja}'s flow for neural networks},
  url = {https://doi.org/10.1109/72.317720},
  volume = {5},
  year = {1994},
}

@article{GlobalTheBucy1967,
  author = {Bucy, R.S.},
  doi = {10.1016/s0022-0000(67)80025-4},
  issue = {4},
  journal = {Journal of Computer and System Sciences},
  language = {en},
  month = {12},
  pages = {349--361},
  publisher = {Elsevier BV},
  title = {{Global} theory of the {Riccati} equation},
  url = {https://doi.org/10.1016/s0022-0000(67)80025-4},
  volume = {1},
  year = {1967},
}

@article{StabilityPropeKazash2021,
  author = {Kazashi, Yoshihito and Nobile, Fabio and Vidli\v{c}kov\'{a}, Eva},
  doi = {10.1007/s00211-021-01241-4},
  issue = {4},
  journal = {Numerische Mathematik},
  language = {en},
  month = {12},
  pages = {973--1024},
  publisher = {Springer Science and Business Media LLC},
  title = {{Stability} properties of a projector-splitting scheme for dynamical low rank approximation of random parabolic equations},
  url = {http://dx.doi.org/10.1007/s00211-021-01241-4},
  volume = {149},
  year = {2021},
}

@article{AMathematicalPerspGeshkovski2023,
  author = {Geshkovski, Borjan and Letrouit, Cyril and Polyanskiy, Yury and Rigollet, Philippe},
  journal = {arXiv preprint arXiv:2312.10794},
  title = {{A} mathematical perspective on transformers},
  year = {2023},
}

@article{LowRankApproximatedTsuzukiB2024,
  author = {Tsuzuki, Daiki and Ohki, Kentaro},
  journal = {IEEE Control Systems Letters},
  publisher = {IEEE},
  title = {{Low}-rank approximated {Kalman}--{Bucy} filters using {Oja}'s principal component flow for linear time-invariant systems},
  year = {2024},
}

@article{CatastrHarlim2010,
  author = {Harlim, John and Majda, Andrew J.},
  journal = {Communications in Mathematical Sciences},
  number = {1},
  pages = {27 -- 43},
  publisher = {International Press of Boston},
  title = {{Catastrophic} filter divergence in filtering nonlinear dissipative systems},
  volume = {8},
  year = {2010},
}

@article{DynamicalLowZoccolan2023,
  archiveprefix = {arXiv},
  author = {Kazashi, Yoshihito and Nobile, Fabio and Zoccolan, Fabio},
  doi = {10.1090/mcom/3999},
  eprint = {2308.11581v4},
  journal = {Mathematics of Computation},
  month = {Aug},
  number = {353},
  pages = {1335--1375},
  primaryclass = {math.NA},
  title = {{Dynamical} low-rank approximation for stochastic differential equations},
  url = {http://arxiv.org/abs/2308.11581v4},
  volume = {94},
  year = {2025},
}

@article{BlendedRedSapsis2013,
  author = {Sapsis, Themistoklis P. and Majda, Andrew J.},
  doi = {10.1016/j.physd.2013.05.004},
  journal = {Physica D: Nonlinear Phenomena},
  language = {en},
  month = {9},
  pages = {61--76},
  publisher = {Elsevier BV},
  title = {{Blended} reduced subspace algorithms for uncertainty quantification of quadratic systems with a stable mean state},
  url = {https://doi.org/10.1016/j.physd.2013.05.004},
  volume = {258},
  year = {2013},
}

@inproceedings{LowRankApproxiTsuzuki2024,
  author = {Tsuzuki, Daiki and Ohki, Kentaro},
  booktitle = {2024 {SICE} {Festival} with {Annual} {Conference} ({SICE} {FES})},
  organization = {IEEE},
  pages = {242--247},
  title = {{Low}-rank approximated {Kalman} filter using {Oja}'s principal component flow for discrete-time linear systems},
  year = {2024},
}

@article{OnTheStabilitDelMo2018,
  author = {Del Moral, P. and Tugaut, J.},
  doi = {10.1214/17-aap1317},
  issue = {2},
  journal = {The Annals of Applied Probability},
  month = {4},
  publisher = {Institute of Mathematical Statistics},
  title = {{On} the stability and the uniform propagation of chaos properties of {Ensemble} {Kalman}\textendash{}{Bucy} filters},
  url = {http://dx.doi.org/10.1214/17-aap1317},
  volume = {28},
  year = {2018},
}

@inproceedings{ComparisonOfEstYamada2021,
  author = {Yamada, Shuto and Ohki, Kentaro},
  booktitle = {2021 60th {Annual} {Conference} of the {Society} of {Instrument} and {Control} {Engineers} of {Japan} ({SICE})},
  organization = {IEEE},
  pages = {462--467},
  title = {{Comparison} of estimation error between two different low-rank {Kalman}-{Bucy} filters},
  year = {2021},
}

@article{LongTimeStabiDeWil2018,
  author = {de Wiljes, Jana and Reich, Sebastian and Stannat, Wilhelm},
  doi = {10.1137/17m1119056},
  issue = {2},
  journal = {SIAM Journal on Applied Dynamical Systems},
  language = {en},
  month = {1},
  pages = {1152--1181},
  publisher = {Society for Industrial \& Applied Mathematics (SIAM)},
  title = {{Long}-{Time} {Stability} and {Accuracy} of the {Ensemble} {Kalman}--{Bucy} {Filter} for {Fully} {Observed} {Processes} and {Small} {Measurement} {Noise}},
  url = {http://dx.doi.org/10.1137/17m1119056},
  volume = {17},
  year = {2018},
}

@article{DataAssimilatiHoutek1998,
  author = {Houtekamer, P. L. and Mitchell, Herschel L.},
  doi = {10.1175/1520-0493(1998)126<0796:dauaek>2.0.co;2},
  issue = {3},
  journal = {Monthly Weather Review},
  language = {en},
  month = {3},
  pages = {796--811},
  publisher = {American Meteorological Society},
  title = {{Data} {Assimilation} {Using} an {Ensemble} {Kalman} {Filter} {Technique}},
  url = {https://doi.org/10.1175/1520-0493(1998)126<0796:dauaek>2.0.co;2},
  volume = {126},
  year = {1998},
}

@article{WellPosedKelly2014,
  author = {Kelly, D.T.B and Law, K.J.H and Stuart, A.M.},
  doi = {10.1088/0951-7715/27/10/2579},
  issue = {10},
  journal = {Nonlinearity},
  month = {10},
  pages = {2579--2603},
  publisher = {IOP Publishing},
  title = {{Well}-posedness and accuracy of the ensemble {Kalman} filter in discrete and continuous time},
  url = {https://doi.org/10.1088/0951-7715/27/10/2579},
  volume = {27},
  year = {2014},
}

@article{ComputLyapReich2001,
  author = {Bridges, Thomas J. and Reich, Sebastian},
  doi = {10.1016/s0167-2789(01)00283-4},
  issue = {3-4},
  journal = {Physica D: Nonlinear Phenomena},
  language = {en},
  month = {8},
  pages = {219--238},
  publisher = {Elsevier BV},
  title = {{Computing} {Lyapunov} exponents on a {Stiefel} manifold},
  url = {https://doi.org/10.1016/s0167-2789(01)00283-4},
  volume = {156},
  year = {2001},
}

@article{DegenerateCarrassi2017,
  author = {Bocquet, Marc and Gurumoorthy, Karthik S. and Apte, Amit and Carrassi, Alberto and Grudzien, Colin and Jones, Christopher K. R. T.},
  doi = {10.1137/16m1068712},
  issue = {1},
  journal = {SIAM/ASA Journal on Uncertainty Quantification},
  language = {en},
  month = {1},
  pages = {304--333},
  publisher = {Society for Industrial \& Applied Mathematics (SIAM)},
  title = {{Degenerate} {Kalman} {Filter} {Error} {Covariances} and {Their} {Convergence} onto the {Unstable} {Subspace}},
  url = {https://doi.org/10.1137/16m1068712},
  volume = {5},
  year = {2017},
}

@article{LyapunovPalatella2013,
  author = {Palatella, Luigi and Carrassi, Alberto and Trevisan, Anna},
  doi = {10.1088/1751-8113/46/25/254020},
  issue = {25},
  journal = {Journal of Physics A: Mathematical and Theoretical},
  month = {6},
  pages = {254020},
  publisher = {IOP Publishing},
  title = {{Lyapunov} vectors and assimilation in the unstable subspace: theory and applications},
  url = {https://doi.org/10.1088/1751-8113/46/25/254020},
  volume = {46},
  year = {2013},
}

@misc{DynamicalLTrigo2026,
  archiveprefix = {arXiv},
  author = {Nobile, Fabio and Riffaud, S\'{e}bastien and Trigo Trindade, Thomas},
  eprint = {2602.06614},
  primaryclass = {math.NA},
  title = {{Dynamical} {Low}-{Rank} {Ensemble} {Kalman} filter for {State}/{Parameter} estimation},
  url = {https://arxiv.org/abs/2602.06614},
  year = {2026},
}

@article{PerformanceMajda2018,
  author = {Majda, Andrew J. and Tong, Xin T.},
  doi = {10.1002/cpa.21722},
  issue = {5},
  journal = {Communications on Pure and Applied Mathematics},
  language = {en},
  month = {5},
  pages = {892--937},
  publisher = {Wiley},
  title = {{Performance} of {Ensemble} {Kalman} {Filters} in {Large} {Dimensions}},
  url = {https://doi.org/10.1002/cpa.21722},
  volume = {71},
  year = {2018},
}

@article{StateobsZhuoyuan2025,
  author = {Li, Zhuoyuan and Dong, Bin and Zhang, Pingwen},
  doi = {10.1016/j.jcp.2025.114240},
  journal = {Journal of Computational Physics},
  language = {en},
  month = {10},
  pages = {114240},
  publisher = {Elsevier BV},
  title = {{State}-observation augmented diffusion model for nonlinear assimilation with unknown dynamics},
  url = {https://doi.org/10.1016/j.jcp.2025.114240},
  volume = {539},
  year = {2025},
}

@article{ASingularDinh1998,
  author = {Tuan Pham, Dinh and Verron, Jacques and Christine Roubaud, Marie},
  doi = {10.1016/s0924-7963(97)00109-7},
  issue = {3-4},
  journal = {Journal of Marine Systems},
  language = {en},
  month = {10},
  pages = {323--340},
  publisher = {Elsevier BV},
  title = {{A} singular evolutive extended {Kalman} filter for data assimilation in oceanography},
  url = {https://doi.org/10.1016/s0924-7963(97)00109-7},
  volume = {16},
  year = {1998},
}

@article{StateEstBrian2001,
  author = {Farrell, Brian F. and Ioannou, Petros J.},
  doi = {10.1175/1520-0469(2001)058<3666:seuaro>2.0.co;2},
  issue = {23},
  journal = {Journal of the Atmospheric Sciences},
  language = {en},
  month = {12},
  pages = {3666--3680},
  publisher = {American Meteorological Society},
  title = {{State} {Estimation} {Using} a {Reduced}-{Order} {Kalman} {Filter}},
  url = {https://doi.org/10.1175/1520-0469(2001)058<3666:seuaro>2.0.co;2},
  volume = {58},
  year = {2001},
}

@article{EnsembleKIvo2025,
  author = {Pasmans, Ivo and Chen, Yumeng and Sebastian Finn, Tobias and Bocquet, Marc and Carrassi, Alberto},
  doi = {10.1002/qj.70070},
  journal = {Quarterly Journal of the Royal Meteorological Society},
  language = {en},
  month = {11},
  publisher = {Wiley},
  title = {{Ensemble} {Kalman} filter in latent space using a variational autoencoder pair},
  url = {https://doi.org/10.1002/qj.70070},
  year = {2025},
}

@article{ReducedYuming2023,
  author = {Chen, Yuming and Sanz-Alonso, Daniel and Willett, Rebecca},
  doi = {10.1088/1361-6420/acff14},
  issue = {12},
  journal = {Inverse Problems},
  month = {12},
  pages = {124001},
  publisher = {IOP Publishing},
  title = {{Reduced}-order autodifferentiable ensemble {Kalman} filters},
  url = {https://doi.org/10.1088/1361-6420/acff14},
  volume = {39},
  year = {2023},
}

@article{TheRedRankGillijns2006,
  author = {Gillijns, S. and Bernstein, D.S. and De Moor, B.},
  doi = {10.3182/20060329-3-au-2901.00202},
  issue = {1},
  journal = {IFAC Proceedings Volumes},
  language = {en},
  pages = {1252--1257},
  publisher = {Elsevier BV},
  title = {{The} {Reduced} {Rank} {Transform} {Square} {Root} {Filter} for {Data} {Assimilation}},
  url = {https://doi.org/10.3182/20060329-3-au-2901.00202},
  volume = {39},
  year = {2006},
}

@article{ASolutionDissanayake2001,
  author = {Dissanayake, M.W.M.G. and Newman, P. and Clark, S. and Durrant-Whyte, H.F. and Csorba, M.},
  doi = {10.1109/70.938381},
  issue = {3},
  journal = {IEEE Transactions on Robotics and Automation},
  month = {6},
  pages = {229--241},
  publisher = {Institute of Electrical and Electronics Engineers (IEEE)},
  title = {{A} solution to the simultaneous localization and map building ({SLAM}) problem},
  url = {https://doi.org/10.1109/70.938381},
  volume = {17},
  year = {2001},
}

@article{SmartGridTodescato2020,
  author = {Todescato, Marco and Carli, Ruggero and Schenato, Luca and Barchi, Grazia},
  doi = {10.3390/en13195148},
  issue = {19},
  journal = {Energies},
  language = {en},
  month = {10},
  pages = {5148},
  publisher = {MDPI AG},
  title = {{Smart} {Grid} {State} {Estimation} with {PMUs} {Time} {Synchronization} {Errors}},
  url = {https://doi.org/10.3390/en13195148},
  volume = {13},
  year = {2020},
}

@article{AssessingSzunyogh2005,
  author = {Szunyogh, Istvan and Kostelich, Eric J. and Gyarmati, G. and Patil, D. J. and Hunt, Brian R. and Kalnay, Eugenia and Ott, Edward and Yorke, James A.},
  doi = {10.3402/tellusa.v57i4.14721},
  issue = {4},
  journal = {Tellus A: Dynamic Meteorology and Oceanography},
  month = {1},
  pages = {528},
  publisher = {Stockholm University Press},
  title = {{Assessing} a local ensemble {Kalman} filter: perfect model experiments with the {National} {Centers} for {Environmental} {Prediction} global model},
  url = {https://doi.org/10.3402/tellusa.v57i4.14721},
  volume = {57},
  year = {2005},
}

@article{AccountingLewis2015,
  author = {Mitchell, Lewis and Carrassi, Alberto},
  doi = {10.1002/qj.2451},
  issue = {689},
  journal = {Quarterly Journal of the Royal Meteorological Society},
  language = {en},
  month = {4},
  pages = {1417--1428},
  publisher = {Wiley},
  title = {{Accounting} for model error due to unresolved scales within ensemble {Kalman} filtering},
  url = {https://doi.org/10.1002/qj.2451},
  volume = {141},
  year = {2015},
}

@article{AccuracyBrett2013,
  author = {Brett, C.E.A. and Lam, K.F. and Law, K.J.H. and McCormick, D.S. and Scott, M.R. and Stuart, A.M.},
  doi = {10.1016/j.physd.2012.11.005},
  issue = {1},
  journal = {Physica D: Nonlinear Phenomena},
  language = {en},
  month = {2},
  pages = {34--45},
  publisher = {Elsevier BV},
  title = {{Accuracy} and stability of filters for dissipative {PDEs}},
  url = {https://doi.org/10.1016/j.physd.2012.11.005},
  volume = {245},
  year = {2013},
}

@book{FoundationsOfModKallenberg2021,
  author = {Kallenberg, Olav},
  doi = {10.1007/978-3-030-61871-1},
  isbn = {['9783030618704', '9783030618711']},
  journal = {Probability Theory and Stochastic Modelling},
  language = {en},
  publisher = {Springer International Publishing},
  title = {{Foundations} of {Modern} {Probability}},
  url = {https://doi.org/10.1007/978-3-030-61871-1},
  year = {2021},
}

@book{NumericalApproQuarte1994,
  author = {Quarteroni, Alfio and Valli, Alberto},
  doi = {10.1007/978-3-540-85268-1},
  isbn = {['9783540852674', '9783540852681']},
  journal = {Springer Series in Computational Mathematics},
  publisher = {Springer Berlin Heidelberg},
  title = {{Numerical} {Approximation} of {Partial} {Differential} {Equations}},
  url = {http://dx.doi.org/10.1007/978-3-540-85268-1},
  year = {1994},
}

@book{StochasticDiffMao2008,
  author = {Mao, Xuerong},
  isbn = {9781904275343},
  language = {English},
  publisher = {Woodhead Publishing},
  title = {{Stochastic} {Differential} {Equations} and {Applications}},
  year = {2008},
}

@book{StochasticPartLiuW2015,
  author = {Liu, Wei and R\"{o}ckner, Michael},
  isbn = {9783319223537},
  publisher = {Springer International Publishing AG},
  title = {{Stochastic} {Partial} {Differential} {Equations} - {An} {Introduction}},
  year = {2015},
}

@book{StatisticsOfLiptser2001,
  author = {Liptser, Robert and Shiryaev, Albert N.},
  isbn = {9783540639299},
  language = {},
  publisher = {Springer},
  title = {{Statistics} of random processes},
  year = {2001},
}

@book{IntroductionToLord2014,
  author = {Lord, Gabriel J. and Powell, Catherine E. and Shardlow, Tony},
  isbn = {9781139017329},
  publisher = {Cambridge University Press},
  title = {{Introduction} to {Computational} {Stochastic} {PDEs}},
  year = {2014},
}

@book{BrownianMKaratzas1991,
  author = {Karatzas, Ioannis and Shreve, Steven E.},
  isbn = {9780387976556},
  language = {},
  publisher = {Springer-Verlag},
  title = {{Brownian} motion and stochastic calculus},
  year = {1991},
}

@book{IntroductionToBoumal2023,
  author = {Boumal, Nicolas},
  isbn = {9781009166157},
  publisher = {Cambridge University Press},
  title = {{Introduction} to {Optimization} on {Smooth} {Manifolds}},
  year = {2023},
}

@book{DataAsEvensen2022,
  author = {Evensen, Geir and Vossepoel, Femke C. and van Leeuwen, Peter Jan},
  doi = {10.1007/978-3-030-96709-3},
  isbn = {['9783030967086', '9783030967093']},
  journal = {Springer Textbooks in Earth Sciences, Geography and Environment},
  language = {en},
  publisher = {Springer International Publishing},
  title = {{Data} {Assimilation} {Fundamentals}},
  url = {https://doi.org/10.1007/978-3-030-96709-3},
  year = {2022},
}

@misc{BenignLandsBoumal2023,
  author = {Boumal, Nicolas},
  title = {{Benign} {Landscape} of the {Brockett} {Cost} {Function}},
  note  = {\textsc{url}: \texttt{www.racetothebottom.xyz/posts/brockett-symmetric/}, visited on 2023-12-04},
  date = {2023-12-04},
  url = {www.racetothebottom.xyz/posts/brockett-symmetric/},
  langid = {en}
}

\end{document}